\documentclass[12pt]{amsart} 
\usepackage{amsmath, amsthm, amssymb, mathtools, wasysym, verbatim, bbm, color, graphics, geometry,graphicx,algorithm2e,accents,enumitem,float,subfigure}
\usepackage{todonotes}

\usepackage[compress,noadjust]{cite}

\newcommand{\RomanNumeralCaps}[1]
    {\MakeUppercase{\romannumeral #1}}

\newcommand{\bN}{\mathbb{N}}

\newcommand{\bR}{\mathbb{R}}

\newcommand\cS{\mathcal{S}}

\newcommand\dist{\operatorname{dist}}
\newcommand{\p}{\partial}
\newcommand{\epsi}{\varepsilon}
\newcommand{\ubar}[1]{\underaccent{\bar}{#1}}

\newcommand{\mg}{\mu_{\Gamma}^+}
\newcommand{\mgn}{\mu_{\Gamma}^-}

\newcommand{\av}{A[v]}

\newcommand{\bl}{\lambda^*}

\theoremstyle{plain}
\newtheorem{theorem}{Theorem}[section]
\newtheorem{lemma}{Lemma}[section]
\newtheorem{corollary}{Corollary}[section]
\newtheorem{proposition}{Proposition}[section]

\newenvironment{theoremprime}[1]
  {\renewcommand{\thetheorem}{\ref{#1}$'$}%
   \addtocounter{theorem}{-1}%
   \begin{theorem}}
  {\end{theorem}}

\theoremstyle{definition}
\newtheorem{definition}{Definition}[section]

\newtheorem{question}{Question}[section]

\theoremstyle{remark}
\newtheorem{remark}{Remark}[section]
\newtheorem{example}{Example}[section]

\makeatletter
\def\thmhead@plain#1#2#3{%
  \thmname{#1}\thmnumber{\@ifnotempty{#1}{ }\@upn{#2}}%
  \thmnote{ {\the\thm@notefont#3}}}
\let\thmhead\thmhead@plain
\makeatother

\begin{document}

\title[Liouville theorems for conformally invariant equations]{Liouville theorems for conformally invariant fully nonlinear equations. \RomanNumeralCaps{1}}

\author[B.Z. Chu]{BaoZhi Chu}
\address[B.Z. Chu]{Department of Mathematics, Rutgers University, 110 Frelinghuysen Road, Piscataway, NJ 08854-8019, USA}
\email{bc698@math.rutgers.edu}

\author[Y.Y. Li]{YanYan Li}
\address[Y.Y. Li]{Department of Mathematics, Rutgers University, 110 Frelinghuysen Road, Piscataway, NJ 08854-8019, USA}
\email{yyli@math.rutgers.edu}

\author[Z. Li]{Zongyuan Li}
\address[Z. Li]{Department of Mathematics, City University of Hong Kong, 83 Tat Chee Avenue, Kowloon Tong, Hong Kong SAR}
\email{zongyuan.li@cityu.edu.hk}

\begin{abstract} 
  A fundamental  theorem of Liouville asserts that positive entire harmonic functions in Euclidean spaces must be  constant.
  A remarkable Liouville-type theorem of Caffarelli-Gidas-Spruck states that positive entire solutions of
  $-\Delta u=u^{ \frac {n+2}{n-2} }$, $n\ge 3$,  are unique modulo M\"obius transformations.  Far-reaching extensions were established for general fully nonlinear conformally invariant equations through the works of Chang-Gursky-Yang, Li-Li, Li, and Viaclovsky. 
  In this paper, we derive necessary and sufficient conditions for 
  the validity of such Liouville-type theorems.
  This leads to necessary and sufficient conditions for local gradient estimates of solutions to hold,
  assuming a one-sided bound on the solutions, 
  for a wide class of fully nonlinear elliptic equations involving Schouten tensors. 
  A pivotal advancement in proving these Liouville-type theorems is our enhanced understanding of solutions to such equations near isolated singularities. 
   In particular, we utilize earlier results of  Caffarelli-Li-Nirenberg 
  on lower- and upper-conical singularities. 
  For general conformally invariant fully nonlinear elliptic equations, we 
  prove that 
  a viscosity super- (sub-)solution can be extended across an isolated singularity
  if and only if it is a lower- 
  (upper-)conical singularity.  
  We also provide necessary and sufficient conditions for lower- (upper-)conical behavior of a function near isolated singularities in terms of its conformal Hessian.

  As an application of our Liouville theorems and local gradient estimates, we establish new existence and compactness results for conformal metrics on a closed Riemannian manifold with prescribed symmetric functions of the Schouten (Ricci) tensor, allowing the scalar curvature of the conformal metrics to have varying signs.
  \end{abstract}
\maketitle
\section{Introduction}
A fundamental theorem of Liouville asserts  that  positive harmonic functions in $\mathbb R^n$ must be  constant.
A remarkable  Liouville-type theorem of Caffarelli, Gidas and Spruck \cite{CGS} in $\mathbb R^n$, $n\ge 3$,
states that a positive $C^2$ solution of
\begin{equation}
    -\Delta u=n(n-2) u^{ \frac {n+2}{n-2} },
    \quad \mbox{in}\ \ \mathbb R^n,
    \label{A1}
\end{equation}
is of the form
\begin{equation}
u(x) = \left( \frac a {1+a^2 |x-\bar x|^2} \right)^{ 
\frac {n-2} 2 },
\label{A2}
\end{equation}
where $a>0$ and $ \bar x\in \bR^n$.
The result was proven by Obata \cite{MR0303464}
under an additional assumption that $|y|^{n-2} u(y/|y|^{2})$ is $C^2$ near $y=0$, and proven by 
Gidas, Ni and Nirenberg \cite{MR0544879} under 
a weaker hypothesis that $u(x)=O(|x|^{2-n})$ for large $|x|$.

Equation (\ref{A1}) is conformally invariant and the Liouville-type theorem gives the uniqueness of positive entire solutions modulo conformal transformations as explained below.

 A map $\varphi:~\bR^n\cup \{\infty\} \rightarrow \bR^n \cup\{\infty\}$ is called a M\"obius transformation if it is a finite composition of translations
 ($x\mapsto x+ \bar{x}~ \text{where}~\bar{x}\in \bR^n$), dilations
 ($x\mapsto ax~ \text{where}~a>0$),  and an inversion ($x\mapsto 
 x/|x|^2$).
 For a M\"obius transformation $\varphi$ and a positive function $u$ in $\bR^n$, $n\ge 3$, it is known that 
\begin{equation*}
   \lambda(A^{u_\varphi}) =\lambda(A^u)\circ \varphi,
\end{equation*}
where $ u_{\varphi} \coloneqq |J_\varphi|^{ \frac {n-2}{2n} } u \circ \varphi$, $J_\varphi$ is the Jacobian of $\varphi$,
\begin{equation}\label{au}
    A^u \coloneqq -\frac{2}{n-2}u^{-\frac{n+2}{n-2}}\nabla^2 u +\frac{2n}{(n-2)^2} u^{-\frac{2n}{n-2}} \nabla u \otimes \nabla u - \frac{2}{(n-2)^2} u^{-\frac{2n}{n-2}} \left| \nabla u\right|^2 I, 
\end{equation}
 $I$ is the $n\times n$ identity matrix, and $\lambda(M)$ denotes the eigenvalues of the symmetric matrix $M$ modulo permutations.
It follows that 
$
f(\lambda(A^{u_\varphi})) =f(\lambda(A^u))\circ \varphi
$    
for any symmetric function $f(\lambda)$.
Here by symmetric function, we mean $f$ is invariant under permutation of $\lambda_i$'s.  In particular, if $u$ is an entire solution of $f(\lambda(A^u))=1$, so is $u_\varphi$ for any M\"obius transformation $\varphi$.   Equation  (\ref{A1}) corresponds to $f(\lambda)=\lambda_1+\cdots+\lambda_n$ and its solutions,   given by (\ref{A2}),
can be generated from any one of them using M\"obius transformations.

The $n\times n$ symmetric matrix function $A^u$, sometimes referred to as the conformal Hessian of $u$, is a canonical object: It was proved in \cite{LiLiCPAM2003} that if a second order differential operator $H(\cdot,u,\nabla u,\nabla^2 u)$ has the property that
$
H(\cdot, u_\varphi, \nabla u_\varphi, 
\nabla^2 u_\varphi)=H(\cdot, u, \nabla u, \nabla^2u)\circ \varphi
$
for any positive function $u$ and any M\"obius transformation $\varphi$, then
$
H(\cdot, u, \nabla u, \nabla^2u)\equiv  f(\lambda(A^u))
$
for some symmetric function $f$.

The following extension of 
the Liouville-type theorem of Caffarelli, Gidas and Spruck
to general conformally invariant second order fully nonlinear elliptic equations was given by Li and Li in \cite{LiLiacta}.

Let 
\begin{equation}   
\Gamma\subset \bR^n \ \mbox{be an open convex symmetric cone with vertex at the origin}
\label{A3}
\end{equation}
satisfying
\begin{equation}
    \Gamma_n\subset \Gamma \subset \Gamma_1,
    \label{A4}
\end{equation}
where $\Gamma_1:= \{ \lambda\in \bR^n \mid \sum_{i=1}^n \lambda_i>0\}$ and 
$\Gamma_n:=\{\lambda\in \bR^n\ |\ \lambda_1, \cdots, \lambda_n>0\}$.  
Naturally, by symmetric set, we mean $\Gamma$ is invariant under interchange of any two $\lambda_i$.

\medskip
\noindent {\bf Theorem A.}\ (\cite{LiLiacta})\
For $n\geq 3$, let $(f,\Gamma)$ satisfy \eqref{A3},
\eqref{A4}, and
\begin{equation}\label{nincreasing}
    f\in C^1(\Gamma)~\text{is a symmetric function},~\frac{\p f}{\p \lambda_i}>0~\text{in}~\Gamma,~\forall 1\leq i\leq n.
\end{equation}
Assume that $u\in C^2(\bR^n)$ satisfies 
$f(\lambda(A^u))=1$ in $\bR^n$.
Then $u(x)\equiv a^{(n-2)/2}(1+b^2 |x-\bar{x}|^2)^{(2-n)/2}$,
for some $\bar{x}\in \bR^n$ and some $a,b>0$ satisfying $f(2b^2 a^{-2} \vec{e})=1$ with $\vec{e}=(1,\dots,1)$.

Equation $f(\lambda(A^u))=1$ is a fully nonlinear elliptic equation of $u$. Fully nonlinear elliptic
equations involving $f(\lambda(\nabla^2u)) $ have been much investigated since  the work of  Caffarelli, Nirenberg and Spruck \cite{MR806416}.
Important examples of $(f, \Gamma)$ satisfying 
\eqref{A3}--\eqref{nincreasing} include
$(f, \Gamma)=(\sigma_k^{1/k}, \Gamma_k)$, $k=1,2,\cdots, n$, where
\begin{equation*}
    \sigma_k(\lambda)\coloneqq \sum\limits_{1\leq i_1< \cdots < i_k\leq n} {\lambda_{i_1}\cdots \lambda_{i_k} },\quad \Gamma_k \coloneqq \{\lambda\in\bR^n \mid \sigma_{l}(\lambda)>0,~\forall 1\leq l \leq k\}.
\end{equation*}

Specific cases of Theorem A were previously established. 
In the case of $n \geq 3$ and  $(f, \Gamma)
= (\sigma_1, \Gamma_1)$,
 it is the aforementioned Liouville-type theorem of Caffarelli, Gidas, and Spruck;
 see also the previously mentioned results of Obata and Gidas-Ni-Nirenberg. 
When $n\ge 3$, $2\le k\le n$, and $(f, \Gamma)= (\sigma_k, \Gamma_k)$, it was proved by Viaclovsky \cite{MR1738176,MR1694380} 
under an additional assumption that $|y|^{n-2} u(y/|y|^{2})$ is $C^2$ near $y=0$. 
When  $n=4$ and $(f, \Gamma)=(\sigma_2, \Gamma_2)$, it was proved by
Chang, Gursky and Yang \cite{MR1945280}.
When $n\geq 3$, $2\leq k\leq n$, and $(f,\Gamma)=(\sigma_k,\Gamma_k)$, it was proved by Li and Li \cite{LiLiCPAM2003}.
Theorem A was strengthened by Li, Nguyen and Wang \cite{MR4181003}  to include 
entire viscosity solutions which are approximable by $C^{1,1}$
solutions on larger and larger balls, and, in particular, for entire $C^{1,1}_{loc}$
solutions. For other related works, see \cite{MR2055838}, \cite{EGM}, and the references therein.

The Liouville theorem for positive harmonic functions in $\bR^n$ was extended 
by Li \cite{Li2009} to positive locally Lipschitz  viscosity solutions of $\lambda(A^u)
\in \partial \Gamma $ in $ \bR^n$ for cone $\Gamma$ satisfying \eqref{A3} and \eqref{A4}; see also an earlier paper \cite{Liarma2007}. 
It was proved by Li, Nguyen and Wang \cite{MR3813247} that $C^0$ viscosity solutions of $\lambda(A^u) \in \p\Gamma$ in an open set is locally Lipschitz for such cones.
The notion of viscosity solutions, as well as viscosity sub- and super-solutions,
of $\lambda(A^u)\in \partial \Gamma $
was introduced by Li \cite[Definition 1.1]{Li2009} in 2006; see \cite{MR2487853} and \cite{MR4127162} for related concepts. 

\medskip

\noindent {\bf Theorem B.}\
  (\cite{Li2009}, \cite{MR3813247})\  For $n\ge 3$, 
 let $\Gamma$ satisfy (\ref{A3}) and (\ref{A4}).
 Assume that $u$ is a positive continuous viscosity solution of 
     $\lambda(A^u)\in \partial \Gamma$ in $\bR^n$.
 Then $u$ must be constant.
 
 \medskip

 A stronger result than Theorem B is

\medskip

\noindent {\bf Theorem C.}\  (\cite{Li2009}, \cite{MR3813247})\  For $n\ge 3$, 
 let $\Gamma$ satisfy (\ref{A3}) and (\ref{A4}).
 Assume that $u$ is a positive continuous viscosity solution of 
     $\lambda(A^u)\in \partial \Gamma$ in $\bR^n\setminus \{0\}$.
 Then $u$ is radially symmetric and non-increasing in the radial direction, i.e.
 $u(x)\ge u(y)$ for all $x, y\in \bR^n\setminus \{0\}$ satisfying 
 $|x|\le |y|$.
\medskip

Theorem B in the case $n\geq 3$ and $\Gamma=\Gamma_1$ is the Liouville theorem for positive harmonic functions in $\mathbb R^n$,
while the case $n=4$, $\Gamma=\Gamma_2$ and $u\in C^{1,1}_{loc}(\bR^4)$ was established by
Chang, Gursky and Yang \cite{MR1945280}. 
Results analogues to Theorem A--C in dimension $n=2$ was obtained 
by Li, Lu and Lu in \cite{Li2021ALT} and \cite{MR4458997}.

The above mentioned Liouville-type theorems have 
played significant  roles in
establishing a number of results in the study of conformal geometry including 
those on  the Yamabe problem, the Nirenberg problem,  and their fully nonlinear versions.
In this paper, we substantially extend Theorem A--C by discovering necessary and sufficient conditions for the conclusions of these theorems to hold, and give both proofs and counterexamples.
Additionally, we provide a  unified approach for the results 
in dimensions $n\ge 3$ and $n=2$.

\subsection{Main results}

In dimensions $n\geq 2$, let $v$ be a $C^2$ function on $\bR^n$. The M\"obius Hessian of $v$ is defined as the following $n\times n$ symmetric matrix: 
\begin{equation*} 
    \av = e^{-2v} \left( - \nabla^2 v + \nabla v \otimes \nabla v - \frac{1}{2} |\nabla v|^2 I \right),
\end{equation*}
where $I$ is the $n\times n$ identity matrix.
For a M\"obius transformation $\varphi$ and a function $v$, define
\begin{equation*}
    v^{\varphi} := v \circ \varphi + \frac{1}{n}\log |J_\varphi|,
\end{equation*}
where $J_\varphi$ is the Jacobian of $\varphi$. It is known that $\av$ has the following M\"obius invariance
\begin{equation*}
    \lambda(A[v^\varphi])=\lambda(\av)\circ \varphi.
\end{equation*}
In dimensions $n\ge 3$, $A^u= A[v]$
and $u_\varphi= e^{ \frac {n-2}{2} v^\varphi}$
for $u=e^{ \frac {n-2} {2} v}$.

In this paper, we consider cones $\Gamma$ more general than those given by (\ref{A3}) and \eqref{A4}:
\begin{equation} \label{eqn-230331-0110}
 \begin{cases}
   \Gamma \subsetneqq \bR^n \,\, \text{is a non-empty open symmetric cone with vertex at the origin},\\
   \Gamma+\Gamma_n\subset\Gamma.
    \end{cases}
\end{equation}

In condition \eqref{eqn-230331-0110}, neither $\Gamma\subset\Gamma_1$ nor the convexity of $\Gamma$ is assumed.
On the other hand, it is easy to see that \eqref{eqn-230331-0110} implies $\Gamma_n\subset \Gamma\subset \bR^n\setminus (-\overline{\Gamma_n})$. 

For constant $p\geq 0$,
consider the equation
\begin{equation}\label{nequation}
         f(\lambda(\av))= e^{-pv}~~\text{on}~\bR^n,
\end{equation}
where 
\begin{equation} \label{eqn-240223-0305}
\begin{cases}
f \in C^{0,1}_{loc}(\Gamma)\,\,\text{is a symmetric function and satisfies} ~~
    \frac{\p f}{\p \lambda_i} \geq c(K)>0,~\forall i\\
    \text{a.e. on compact subset}\,\,K ~~   \text{of}~\Gamma.
\end{cases}
\end{equation}
Clearly, condition \eqref{nincreasing} implies \eqref{eqn-240223-0305}.
Denote 
\begin{equation*}
    \bl\coloneqq(1,-1,\ldots,-1).
\end{equation*}
We have the following Liouville-type theorem.

\begin{theorem}\label{thm-230525-0321}
 For $n\geq 2$ and $p\geq 0$, let $(f,\Gamma)$ satisfy \eqref{eqn-230331-0110} and \eqref{eqn-240223-0305} with $\bl \notin \overline{\Gamma}$.
 Assume that $v\in C^{1,1}_{loc}(\bR^n)$ satisfies \eqref{nequation} almost everywhere. Then $p=0$ and
     \begin{equation}\label{nbubble}
         v(x)\equiv \log \big( \frac{a}{1+b^2 |x-\bar{x}|^2}\big),
     \end{equation}
     where $\bar{x}\in \bR^n$ and $a,b>0$ satisfy $f(2b^2 a^{-2} \vec{e})=1$ with $\vec{e} =(1,\dots,1)$.
\end{theorem}
Note that for $v$ defined in \eqref{nbubble}, $\av = 2b^2a^{-2}I$.
\begin{remark} \label{rmk-230615-1129}
The condition $\bl \notin \overline{\Gamma}$ in Theorem \ref{thm-230525-0321} is optimal: Whenever $\Gamma$ satisfies \eqref{eqn-230331-0110} with $\bl \in \overline{\Gamma}$, there exist a smooth function $f$ satisfying \eqref{nincreasing} and a smooth solution $v$ of $f(\lambda(\av)) = 1$ in $\bR^n$ which is not of the form \eqref{nbubble}. For details, see Section \ref{sec-230622-2309}.
\end{remark}

The condition $\bl\notin\overline\Gamma$ is equivalent to stating that for some $\mu>1$, the inequality $\lambda_1+\mu\lambda_2>0$ holds for any $\lambda\in\Gamma$ satisfying $\lambda_1\geq \dots\geq \lambda_n$. See Lemma \ref{240901-1559}.

Theorem \ref{thm-230525-0321} in the case $\Gamma\subset\Gamma_1$, which includes $(f,\Gamma)=(\sigma_k,\Gamma_k)$ as particular cases, was previously known.
For $n\geq 3$, see Theorem A (with $v=\frac{2}{n-2}\log u$) and the discussions below it. For $n=2$, see \cite{Li2021ALT}. 

Next we state our Liouville-type theorems for viscosity solutions of the following equation:
\begin{equation} \label{eqn-221204-0213}
    \lambda(\av) \in \p \Gamma.
\end{equation}

For the definition of viscosity solutions, see Definition \ref{def-230607-0451}.
Equation \eqref{eqn-221204-0213} is in general a degenerate elliptic equation for which neither the strong maximum principle nor the Hopf Lemma holds, as shown by Li and Nirenberg \cite{li2009miscellany}.

\begin{theorem}\label{thm-230329-1327}
    For $n\geq 2$, let $\Gamma$ satisfy \eqref{eqn-230331-0110} with $\bl\notin\p\Gamma$. Then any continuous viscosity solution of \eqref{eqn-221204-0213} in $\bR^n$ must be constant. 
\end{theorem}

\begin{remark}\label{remark-1.3}
    The assumption $\bl\notin\p\Gamma$ is optimal in the sense that whenever $\bl\in\p\Gamma$, the conclusion of the above theorem  fails. This can be seen from the fact that 
    $\lambda(\av) = \frac{1}{2} |\nabla v|^2 e^{-2v}  \bl\in\p\Gamma$ for every linear function $v$.
\end{remark}

Since $\Gamma$ satisfying \eqref{eqn-230331-0110} is equivalent to the cone $\bR^n\setminus(-\overline{\Gamma})$ satisfying \eqref{eqn-230331-0110}, and since $\lambda(\av)\in\p\Gamma$ is equivalent to $\lambda(-\av)\in\p(\bR^n\setminus(-\overline{\Gamma}))$, Theorem \ref{thm-230329-1327} is equivalent to:
 \begin{theoremprime}{thm-230329-1327}\label{1327-prime}
    For $n\geq 2$, let $\Gamma$ satisfy \eqref{eqn-230331-0110} with $-\bl\notin\p\Gamma$. Then any continuous viscosity solution of 
    $\lambda(-\av)\in\p \Gamma$ in $\bR^n$ must be constant. 
 \end{theoremprime}

\begin{remark}\label{remark-240701-1040}
    For $n\geq 2$, if $\Gamma$ satisfies \eqref{eqn-230331-0110} and $\bl\in\Gamma$ ($-\bl\notin \overline{\Gamma}$), then any viscosity solution of $\lambda(\av)\in\bR^n\setminus\Gamma$ ($\lambda(-\av)\in\overline{\Gamma}$) in $\bR^n$ must be constant. See Theorem \ref{thm-230528-0404}. On the other hand, if $\bl\notin\Gamma$ ($-\bl\in \overline{\Gamma}$), then every linear function is an entire solution of $\lambda(\av)\in\bR^n\setminus\Gamma$ ($\lambda(-\av)\in\overline{\Gamma}$).   
    \end{remark}

It follows from the above that any entire solution of $\lambda(-\av)\in\p\Gamma_k$ must be constant if and only if $k\neq n/2$, and any entire solution of $\lambda(-\av)\in\overline{\Gamma_k}$ must be constant if and only if $k>n/2$. The reason is that, for $n\geq 2$ and $1\leq k\leq n$, $-\bl\in \p\Gamma_k$ if and only if $k=n/2$ and $-\bl\notin\overline{\Gamma_k}$ if and only if $k>n/2$.

In fact, we have proved a stronger result than Theorem \ref{thm-230329-1327}, which is on the radial symmetry of solutions of $\lambda(\av)\in\p\Gamma$ in $\bR^n\setminus\{0\}$:
\begin{theorem}\label{230918-1857}
    For $n\geq 2$, let $\Gamma$ satisfy \eqref{eqn-230331-0110} with $\bl\notin\p\Gamma$. Then any continuous viscosity solution of \eqref{eqn-221204-0213} in $\bR^n\setminus\{0\}$ is radially symmetric and non-increasing in the radial direction, i.e.
 $v(x)\ge v(y)$ for all $x, y\in \bR^n\setminus \{0\}$ satisfying 
 $|x|\le |y|$.
 \end{theorem}

 \begin{remark}\label{remark-240701-1041}
     The assumption $\bl\notin\p\Gamma$ in the above theorem is optimal since, as mentioned earlier, every linear function $v$ is a solution of $\lambda(\av)\in \p\Gamma$ when $\bl\in\p\Gamma$.
\end{remark}
 
Theorem \ref{230918-1857} is equivalent to
 \begin{theoremprime}{230918-1857}\label{1857-prime}
For $n\geq 2$, let $\Gamma$ satisfy \eqref{eqn-230331-0110} with $-\bl\notin\p\Gamma$. Then any continuous viscosity solution of $\lambda(-\av)\in\p\Gamma$ in $\bR^n\setminus\{0\}$ is radially symmetric and non-increasing in the radial direction, i.e.
 $v(x)\ge v(y)$ for all $x, y\in \bR^n\setminus \{0\}$ satisfying 
 $|x|\le |y|$.
 \end{theoremprime}

Theorem \ref{thm-230329-1327} and  \ref{230918-1857} 
in the case $\Gamma\subset\Gamma_1$ were previously known. For $n\geq 3$, see Theorem B and C (with $v=\frac{2}{n-2}\log u$) and discussions below them. For $n=2$, see \cite{MR4458997}.

Examples of Theorem \ref{thm-230525-0321}--\ref{230918-1857}, where $\Gamma$ is not necessarily contained in $\Gamma_1$, are provided in \S \ref{intro-example-sec}, with further examples given in \S \ref{examplesection}.

Theorem \ref{thm-230525-0321}--\ref{230918-1857} hold not only for cones $\Gamma$ but can be generalized to suitable sets that are not necessarily cones.
This and analogous results for Theorem \ref{thm-230525-0321}--\ref{230918-1857} in the half space are established in our forthcoming papers.

\subsection{Applications of Liouville theorems to local gradient estimates assuming a one-sided bound on solutions}

Let $(M,g)$ be a compact, smooth, connected Riemannian manifold (without boundary) of dimension $n\ge 3$.
The Yamabe conjecture was solved through the works of Yamabe \cite{Yamabe1960}, Trudinger \cite{Trudinger1968}, Aubin \cite{Aubin1976} and Schoen \cite{Schoen1984}:
There exist constant scalar curvature metrics on $M$ which are pointwise conformal to $g$.  The solution space was well understood when the total scalar curvature of $(M,g)$ is non-positive.  On the other hand,   the situation is strikingly delicate
when the total scalar curvature of $(M,g)$ is 
positive: 
The solution space is compact with respect to $C^m$ norms for any $m$ if $(M,g)$ is non-locally conformally flat and has dimension $n\le 24$. However, for $n\ge 25$, 
smooth non-locally conformally flat metrics exist on the $n$-dimensional sphere $\mathbb S^n$, 
for which the solution space is not compact in the $L^\infty$ norm.  
Another closely related phenomena is that the
solution space is compact with respect to $C^m$ norms for any $m$ and in any dimensions $n\ge 3$
if the Weyl tensor and its first covariant derivatives
do not vanish simultaneously at any point on $M$, while  
smooth metrics exist on $\mathbb S^n$ 
for which  the Weyl tensor and its up to third  covariant derivatives
do not vanish simultaneously at any point on $\mathbb S^n$,  but 
the solution space is not compact in the $L^\infty$ norm.  
The question of compactness of the solution space remains largely open when 
  the Weyl tensor and its up to second order covariant derivatives
do not vanish simultaneously at any point on $M$.
These are achieved through the works 
\cite{Brendle08JAMS, Brendle_Marques, Druet04, Khuri_Marques_Schoen, Li_Zhang_2, Li_Zhang_3, Li_Zhu_99, Marques05, Marques09, Monyicatilli_notes, Schoen_on_the_number}.
See also a more recent work \cite{GongLi}, 
where it is proved that in dimensions 
$n\ge 25$
there exist a smooth non-locally conformally flat 
metric $g$ on $\mathbb S^n$, and a sequence of 
smooth metrics $ g_i$ which are pointwise conformal to $g$,
with scalar curvature $R_{g_i}\equiv 1$ and $\text{volume}(g_i)\to \infty$.

On a Riemannian manifold $(M,g)$ of dimension $n\geq 3$,
consider the Schouten tensor
\begin{equation*}
    A_g=\frac{1}{n-2}\big( Ric_g-\frac{R_g}{2(n-1)}g\big),
\end{equation*}
where $Ric_g$ and $R_g$ denote, respectively, the Ricci tensor and the scalar curvature. We use $\lambda(A_g)=(\lambda_1(A_g),\dots,\lambda_n(A_g))$ to denote the eigenvalues of $A_g$ with respect to $g$. 

For a positive function $u$ on $M$, let $g_u\coloneqq u^{\frac{4}{n-2}}g$ be a conformal change of the metric $g$. Direct computation gives that 
\begin{equation*}
A_{g_u}=-\frac{2}{n-2} u^{-1} \nabla^2 u + \frac{2n}{(n-2)^2} u^{-2} d u\otimes d u -\frac{2}{(n-2)^2} u^{-2}|\nabla u|^2 g + A_g,
\end{equation*}
where all covariant derivatives and norms on the right hand side are with respect to $g$. In particular, when $\bar{g} = |dx|^2$ is the Euclidean metric on $\bR^n$, 
\begin{equation*}
    A_{\bar{g}_{ u} } =u^{\frac{4}{n-2}} A^u_{ij}dx^i dx^j,
\end{equation*}
where $A^u$ is the conformal Hessian of $u$ defined in \eqref{au}. In this case, $\lambda(A_{\bar{g}_u})=\lambda(A^u)$.

We expect that Theorem \ref{thm-230525-0321}--\ref{230918-1857} will play a significant role
in the study of the following
equations.
\begin{equation}\label{pequ}
        f(\lambda(A_{g_u}))=h, ~~ u>0,~~\text{on}~M, \tag{$P$}
\end{equation}
and
\begin{equation}\label{equ}
    f(\lambda(-A_{g_u}))=h, ~~ u>0,~~\text{on}~M, \tag{$N$}
\end{equation}
where $h$ is a positive function on $M$. 
When $(f, \Gamma)=(\sigma_1, \Gamma_1)$ and $h\equiv 1$, equations \eqref{pequ} and \eqref{equ} are the Yamabe equations in the positive and negative cases respectively.

We discuss local gradient estimates for equations \eqref{pequ} and \eqref{equ}.  The following conditions on $(f,\Gamma)$ will be assumed: $\Gamma$ satisfies \eqref{eqn-230331-0110} and $f$ satisfies 
\begin{equation}\label{basicf}
     \left\{
	       \begin{array}{lr}
                \text{$f\in C^0(\overline{\Gamma})\cap C^1 (\Gamma)$ is 
                symmetric in $\lambda_i$,}\\[2ex]
	       	\text{$f>0$, \ $ \p_{\lambda_i}f>0 $
         $~\forall 1\leq i\leq n,$ $\text{in}$ $\Gamma$, and~ $f=0$ ~on $\partial{\Gamma}$,
         }
	       	\end{array}
	\right.	
\end{equation}

\begin{equation} \label{eqn-230530-1131}
    \sum_i \p_{\lambda_i}f(\lambda) \geq \delta \big( 1+ \big|\lambda_i \cdot \p_{\lambda_i}f(\lambda)\big| \big)\,\,\text{on}\,\,\{ \lambda \in \Gamma\mid  f (\lambda) \in (0, C]\},~~ \forall C > 0,
\end{equation}
where $\delta=\delta(C) > 0$, and
\begin{equation} \label{eqn-230602-1040}
    \liminf_{s\rightarrow \infty} f(s\lambda) = +\infty \,\, \text{uniformly on}\,\, \{\lambda \in \Gamma\mid f(\lambda) \geq C\}, ~~ \forall C>0.
\end{equation}

For a homogeneous of degree $1$ function $f$ in $\Gamma$, conditions \eqref{eqn-230530-1131}
and \eqref{eqn-230602-1040} are the same as $\sum_i f_{\lambda_i} \geq  \delta>0$ in $\Gamma$. In particular, $(f,\Gamma)=(\sigma_k^{1/k},\Gamma_k)$ satisfies \eqref{basicf}--\eqref{eqn-230602-1040}.

Fully nonlinear elliptic equations involving the Schouten tensor have been investigated extensively since the works of Viaclovsky \cite{MR1738176, MR1694380, MR1925503} and Chang-Gursky-Yang
\cite{CGY_annals02,
MR1945280,MR2055838}.    
In particular, Equations \eqref{pequ} and \eqref{equ} have been much studied  when $\Gamma\subset\Gamma_1$.
For instance, the existence and compactness of smooth solutions
of  equation  \eqref{pequ}  have been obtained for general $(f,\Gamma)$ on locally conformally flat Riemannian manifolds, and when the problem has a variational structure (including
$(f, \Gamma)=(\sigma_2^{1/2}, \Gamma_2)$) on general Riemannian manifolds, or when 
$\Gamma$ satisfies $-\bl\notin\Gamma$ 
(including 
$(f, \Gamma)=(\sigma_k^{1/k}, \Gamma_k)$, $k\ge n/2$).
For equation  \eqref{equ},  the existence of Lipschitz solutions are known for general $(f, \Gamma)$.
See \cite{MR1945280, MR2290138, MR1978409, MR1976082, MR2078344, MR2373147, LiLiCPAM2003, LiLiacta, MR3165241, MR4210287, MR2362323, MR1925503},  and the references therein.  
On the other hand, there is little
 study of equations 
 \eqref{pequ} and \eqref{equ} when $\Gamma$ is not contained in $\Gamma_1$.  

In the study of equations
\eqref{pequ} and \eqref{equ}, local gradient estimates of solutions
assuming a one-sided bound of the solutions
are 
useful and delicate.  
An application of Theorem \ref{thm-230329-1327} and \ref{1327-prime} gives such estimates.

\begin{theorem}[(Local gradient estimates for equation  \eqref{pequ})] \label{localgradest}
    Let $(B_1,g)$ be a $C^3$ Riemannian geodesic ball and $(f,\Gamma)$ satisfy \eqref{eqn-230331-0110} and \eqref{basicf}--\eqref{eqn-230602-1040} with $\bl\notin\p\Gamma$. Let $h\in C^1(B_1)$ be positive
    and $u\in C^3(B_1)$ satisfy equation
    $\eqref{pequ}$ in $B_1$. Then
	\begin{equation}\label{estimatee}
		\left| \nabla_g \log u \right|_g \leq C \quad \text{in } ~ B_{1/2},
	\end{equation}
	where $C$ is a constant depending only on $(f,\Gamma)$, an upper bound of $\sup_{B_{1}} u$ and $\left\| h \right\|_{C^1(B_1)}$, and a bound of $g$ and the Riemann curvature
 tensor together with its first covariant derivative with respect to $g$. 
\end{theorem}

\begin{theorem}[(Local gradient estimates for equation  \eqref{equ})] \label{nlocalgradest}
 Let $(B_1,g)$ be a $C^3$ Riemannian geodesic ball and $(f,\Gamma)$ satisfy \eqref{eqn-230331-0110} and \eqref{basicf}--\eqref{eqn-230602-1040} with $-\bl\notin\p\Gamma$. Let $h\in C^1(B_1)$ be positive
    and $u\in C^3(B_1)$ satisfy equation
    $\eqref{equ}$ in $B_1$. Then \eqref{estimatee} holds with $C$ of the same dependence as that of Theorem \ref{localgradest}. Moreover, if in addition $\inf_{B_1}h\geq h_0>0$, then $C$ in \eqref{estimatee}, depending on $h_0$, is independent of $\sup_{B_1} u$.
\end{theorem}

\begin{remark}
     Note that in both Theorem \ref{localgradest} and \ref{nlocalgradest}, we assume neither concavity nor convexity on $f$, and $C$ is independent of $\inf\limits_{B_1}u$. 
\end{remark}

\begin{remark}\label{231006-1030}
         The condition 
         $\bl\notin\p\Gamma$ in Theorem \ref{localgradest} (resp. $-\bl\notin\p\Gamma$ in Theorem \ref{nlocalgradest}) is optimal in the sense that whenever the condition is violated, for any homogeneous of degree $1$ function $f$ satisfying the hypothesis of the theorem, there exists a sequence of positive functions $\{u_i\}\subset C^\infty(B_1)$ such that $\sup_{B_1}u_i\leq 1$ and $h_i\coloneqq f(\lambda(A^{u_i}))>0$ (resp. $h_i\coloneqq f(\lambda(-A^{u_i}))>0$) converges in $C^m(B_1)$ for any $m\in\bN$. However, $|\nabla \log u_i|\to\infty$ uniformly on $B_{1/2}$. See Example \ref{exp-230606-0509}--\ref{exp-230606-0459}.
\end{remark}

\begin{remark}\label{remark-1.8}
    When $f$ is homogeneous of degree  $1$, replacing the function $h$ in equations \eqref{pequ} and \eqref{equ} by $h(\cdot,u)$ with
    \begin{equation*}
        s^{4/(n-2)} h(x,s)\in C^1(B_1\times(0,\infty))\cap L^{\infty}(B_1\times (0,b)),
    \end{equation*}
    for any $b>1$, Theorem \ref{localgradest} and \ref{nlocalgradest} still hold, with $C$ depending also on the function $h$. This is easy to see from the proof of the theorems.
\end{remark}

\begin{remark}\label{tend-to-believe}
    We tend to believe that the assumption $f\in C^1(\Gamma)$ in Theorem \ref{localgradest} and \ref{nlocalgradest} can be weaken to $f\in C^{0,1}_{loc}(\Gamma)$.
\end{remark}

When $(f,\Gamma)=(\sigma_k^{1/k},\Gamma_k)$, $k=n/2$, local gradient estimates for equation \eqref{equ} assuming a one-sided bound on solutions fail (see Remark \ref{231006-1030}). This is an  unexpected phenomena. 
For $n\ge 3$ and all $2\le k\le n$, gradient estimates for equation \eqref{equ} with $(f,\Gamma)=(\sigma_k^{1/k},\Gamma_k)$ on closed manifolds were proved by Gursky and Viaclovsky \cite{MR1976082}. 
If solutions are bounded from both below and above,
local gradient estimates do hold (see Theorem \ref{ahequlocalgradest}).
For the equations studied in \cite{MR1976082}
on  closed manifolds, solutions are  bounded from both below and above.

A more general result
than Theorems \ref{localgradest} and \ref{nlocalgradest}, applicable to unified dimensions $n\geq 2$, is proved in Theorem \ref{thm-230603-1124}. 
Theorem  \ref{localgradest}
was proved by Guan and Wang \cite{MR1976045} for 
$n\geq 3$, $2\leq k\leq n$, and $(f,\Gamma)=(\sigma_k^{1/k},\Gamma_k)$, and was proved by Li \cite{Li2009} for
$\Gamma \subset \Gamma_1$ and homogeneous $f$.  The theorem in dimension $n=2$ 
was proved by Li, Lu and Lu \cite{MR4458997}  for $\Gamma\subset \Gamma_1$
and homogeneous $f$.  For other related works, see 
\cite{MR2204639, MR2306044, Liarma2007, Li2009, MR2243678}, and the references therein.

After establishing Theorems \ref{localgradest} and \ref{nlocalgradest} and demonstrating their optimality (as discussed in Remark \ref{231006-1030}), we came across the thesis of Khomrutai \cite{MR2827324}. In the thesis, Khomrutai provided a proof for Theorem \ref{nlocalgradest} with $(f,\Gamma)=(\sigma_k^{1/k},\Gamma_k)$ in the case of $n\ge 3$ and $2\le k<n/2$, as well as in the case where $k=n-1 $ or $n$, utilizing a completely different method.
More recently, Duncan and Nguyen \cite{Duncan_Nguyen_23} have proved Theorem \ref{nlocalgradest} for $\Gamma\subset\Gamma_1$ with $-\bl\in\Gamma$ and for homogeneous concave $f$ using methods in the spirit of that in \cite{MR2827324}.

\subsection{Applications of Liouville theorems to existence and compactness problems}
We expect that rigidity results Theorem \ref{thm-230525-0321}--\ref{230918-1857} will play crucial roles in answering the following question.

\begin{question}\label{quest-1.1-schouten}
Let $(M^n, g)$ be a closed, smooth Riemannian manifold of dimension $n\geq 3$, and $\Gamma$ satisfy \eqref{eqn-230331-0110}. Assume that $\lambda(A_g)\in\Gamma$ on $M^n$. For which symmetric function $f$ defined on $\Gamma$ does the equation 
\begin{equation}\label{q-schouten-critical}
  {f}(\lambda(A_{g_u}))=1\quad \text{on}~M^n 
\end{equation}
have a positive solution $u$?
\end{question}

 When $\Gamma\subset\Gamma_1$ and $f$ is a concave function in $\Gamma$, Question \ref{quest-1.1-schouten} has been studied extensively in the literature mentioned earlier. 
 For instance, when 
 $(f,\Gamma)=(\sigma_k^{1/k},\Gamma_k)$, Question \ref{quest-1.1-schouten} is the Yamabe problem $(k=1)$ and the $\sigma_k$--Yamabe problem in the positive case.
 On the other hand, when $\Gamma$ is not contained in $\Gamma_1$, there have been no existence or compactness results.
 Towards answering Question \ref{quest-1.1-schouten}, 
 we present the following new existence and compactness results based on our Liouville theorems, Theorem 
 \ref{thm-230525-0321} and 
 \ref{thm-230329-1327}.
 We introduce the following conditions on $(f,\Gamma)$.
\begin{equation}\label{lip-f-schouten} 
\begin{cases}
    f\in C^{0,1}_{loc}(\Gamma)\cap C^0(\overline\Gamma)~\text{is symmetric, homogeneous of degree $1$,} \\ f\big|_{\p\Gamma}=0,~\text{and}~\frac{\p f}{\p \lambda_i}\geq \delta~\text{a.e.}~\Gamma,~\forall i,~\text{for some constant}~\delta>0,
    \end{cases}
\end{equation}
and
\begin{equation}\label{fgamma-convex}
    f~\text{is locally convex in}~\Gamma~\text{and}~\bR^n\setminus\Gamma~\text{is convex}.
\end{equation}
The convexity of $\bR^n\setminus\Gamma$ implies that $\Gamma_1\subset\Gamma$, and $\Gamma$ is not contained in $\Gamma_1$ unless $\Gamma=\Gamma_1$.
\begin{theorem}\label{exist-lcf}
    Let $(M^n,g)$ be a closed, smooth, locally conformally flat $n$-dimensional Riemannian manifold with positive scalar curvature, $n\geq 3$, and $(f,\Gamma)$
    satisfy \eqref{eqn-230331-0110}, \eqref{lip-f-schouten} and \eqref{fgamma-convex}.
     Then, for some $\alpha\in(0,1)$ depending only on $(M^n,g)$ and $(f,\Gamma)$, there exists a positive function $u\in C^{2,\alpha}(M)$ satisfying \eqref{q-schouten-critical}.
     Moreover, if $(M^n,g)$ is not conformally diffeomorphic to the standard sphere, 
    all $C^2$ solutions $u$ of \eqref{q-schouten-critical} satisfy 
    \begin{equation}\label{compactness-1.6-schouten}
     \|u\|_{C^{2,\alpha}(M^n,g)}+ \|1/u\|_{C^{2,\alpha}(M^n,g)}\leq C,
    \end{equation}
    where $C>0$ is some constant depending only on $(M^n,g)$ and $(f,\Gamma)$.
\end{theorem}

One of the difficulties in solving equation \eqref{q-schouten-critical} is to obtain the compactness of the solution space along a homotopy from $(f,\Gamma)$ to $(\sigma_1,\Gamma_1)$. In the proof of Theorem \ref{exist-lcf}, the positive scalar curvature assumption on $(M^n,g)$ is used. 
If the condition of positive scalar curvature is weakened to $\lambda(A_g)\in\Gamma$, one possible strategy for solving \eqref{q-schouten-critical}
is to consider a subcritical approximation of the equation and analyze the behavior of solutions as $q\to 0+$:
 \begin{equation}\label{q-schouten-subcritical}
        {f}(\lambda(A_{g_u}))=u^{-q}\quad \text{on}~M^n.
    \end{equation}
We have the following compactness and existence results.

\begin{theorem}\label{exist-nlcf-compactness}
     Let
    $(M^n,g)$ be a closed, smooth $n$-dimensional Riemannian manifold, $n\geq 3$,   $(f,\Gamma)$
    satisfy \eqref{eqn-230331-0110}, \eqref{lip-f-schouten}, \eqref{fgamma-convex} and $\bl\notin\overline\Gamma$, and $0<q<4/(n-2)$ be a constant.
    Assume that $\lambda(A_g)\in\Gamma$ on $M^n$ and $f\in C^1(\Gamma)$.
    Then,  for any $\alpha\in(0,1)$, 
    all positive $C^2$ solutions of \eqref{q-schouten-subcritical} satisfy \eqref{compactness-1.6-schouten} with $C$ depending only on $(M^n,g)$, $(f,\Gamma)$, $q$ and $\alpha$.
\end{theorem}

\begin{remark}
    In Theorem \ref{exist-nlcf-compactness},
    the $C^1$ regularity of $f$ is only required for the application of Theorem \ref{localgradest}. We believe that $f\in C^{0,1}_{loc}$ would suffice.
\end{remark}

\begin{theorem}\label{exist-nlcf}
     Let
    $(M^n,g)$ be a closed, smooth $n$-dimensional Riemannian manifold with positive scalar curvature, $n\geq 3$,   $(f,\Gamma)$
    satisfy \eqref{eqn-230331-0110}, \eqref{lip-f-schouten}, \eqref{fgamma-convex} and $\bl\notin\overline\Gamma$, and $0<q<4/(n-2)$ be a constant.
    Then,  for some $\alpha\in(0,1)$ depending only on $(M^n,g)$, $(f,\Gamma)$ and $q$, 
    there exists some positive solution $u\in C^{2,\alpha}(M)$  of \eqref{q-schouten-subcritical}. 
\end{theorem}

In Theorem \ref{exist-lcf}--\ref{exist-nlcf}, if $ f$ is in $C^{m,\beta}_{loc}(\Gamma)$, $m=1,2,\dots$ and $\beta\in (0,1)$, then estimate \eqref{compactness-1.6-schouten} of solutions $u$ can be strengthened to 
$\|u\|_{C^{m+2,\beta}(M^n,g)}+ \|1/u\|_{C^{m+2,\beta}(M^n,g)}\leq C$.

It is easy to see that the positive scalar curvature assumption on $(M^n,g)$ in Theorem \ref{exist-lcf} and \ref{exist-nlcf} can be weakened to 
the positivity of the first eigenvalue of the conformal Laplacian $-\Delta_g +
\frac {n-2}{ 4(n-1) } R_g$.

\subsection{Examples}\label{intro-example-sec}
   In this subsection, we provide some examples of Theorem \ref{thm-230525-0321}--\ref{exist-nlcf} by taking appropriate $(f,\Gamma)$. Further examples can be found in \S \ref{examplesection}. Some of these examples are given in terms of Ricci tensor. They follow from the above theorems by a linear transformation, see Appendix \ref{riccisection} for details as well as equivalent reformulations of our theorems in terms of Ricci tensor.

For any symmetric subset $\Omega$ of $\bR^n$, a continuous function defined on $\{\lambda\in\Omega\mid\lambda_1\geq \dots\geq\lambda_n\}$ corresponds to a continuous symmetric function on $\Omega$. Therefore,
throughout the paper, we only specify the definition of a symmetric function in the region $\{\lambda_1\geq\dots\geq\lambda_n\}$. 
Recall that $\bar{g}=|dx|^2$ is the Euclidean metric on $\bR^n$ and $\bar{g}_u\coloneqq u^{\frac{4}{n-2}}\bar{g}$.

\begin{example} 
For $n\geq 3$ and $2\leq i\leq n$, 
\begin{equation*}
\lambda_i(Ric_{\bar{g}_u})=1,~ u>0,~ \text{and}~u\in C^{1,1}_{loc}~ \text{a.e. in}~ \bR^n~ \implies u\equiv e^{\frac{n-2}{2}v},
\end{equation*}
where $v$ is of the form \eqref{nbubble} with $\bar x\in\bR^n$, $a,b>0$ satisfying $4(n-1)b^2 a^{-2}=1$; 
\begin{equation*}
\lambda_i(Ric_{\bar{g}_u})=0,~ u>0,~\text{and}~ u\in C^0~ \text{in}~ \bR^n~ \implies u~\text{is constant};
\end{equation*}
\begin{equation*}
\lambda_i(Ric_{\bar{g}_u})=0,~ u>0,~\text{and}~ u\in C^0~ \text{in}~ \bR^n\setminus\{0\}~ \implies~  u~ \text{is radially symmetric}.
\end{equation*}
However, none of the above results hold when $i=1$, as shown by counterexamples given in Remark \ref{rmk-230615-1129} and \ref{remark-1.3}.
\end{example}
\begin{example}
 For $n\geq 3$ and $1\leq i< j \leq n$,
\begin{equation*}
(\lambda_i+\dots+\lambda_j)(Ric_{\bar{g}_u})=1,~ u>0,~\text{and}~u\in C^{1,1}_{loc}~ \text{a.e. in}~ \bR^n~ \implies u\equiv e^{\frac{n-2}{2}v},
\end{equation*}
where $v$ is of the form \eqref{nbubble} with $\bar x\in\bR^n$, $a,b>0$ satisfying $4(n-1)(j-i+1)b^2 a^{-2}=1$; 
\begin{equation*}
(\lambda_i+\dots+\lambda_j)(Ric_{\bar{g}_u})=0,~ u>0,~ \text{and}~u\in C^0~ \text{in}~ \bR^n~ \implies u~\text{is constant};
\end{equation*}
\begin{equation*}
(\lambda_i+\dots+\lambda_j)(Ric_{\bar{g}_u})=0,~ u>0,~\text{and}~ u\in C^0~ \text{in}~ \bR^n\setminus\{0\}~ \implies~  u~ \text{is radially symmetric}.
\end{equation*}
\end{example}

\medskip

    For $n\geq 3$ and $p=1,\dots,n-1$, let
\begin{equation*}
    G_p(\lambda)\coloneqq p \sum\limits_{i\leq n-p}\lambda_i + (n-p)\sum\limits_{i>n-p}\lambda_i.
\end{equation*} 
On a locally conformally flat Riemannian manifold $(M^n,g)$, the quantity $G_p(\lambda(A_g))$, sometimes referred to as the $p$-Weitzenb\"ock curvatures, arises naturally from the 
Weitzenb\"ock formula for $p$-forms $\omega$:
\begin{equation*}
    \bigtriangleup \omega=\nabla^*\nabla \omega + G_p(\lambda(A_g)) \omega,
\end{equation*}
where $\bigtriangleup=dd^*+d^*d$ is the Hodge-de Rham Laplacian and $\nabla^*\nabla$ is the connection Laplacian; see \cite{MR2306044} and the references therein.
\begin{example}\label{240831-1248}
    For $n\geq 3$ and $p=1,\dots,n-2$, \begin{equation*}
G_p(\lambda(A_{\bar{g}_u}))=1,~u>0,~\text{and}~ u\in C^{1,1}_{loc}~ \text{a.e. in}~ \bR^n~ \implies~u\equiv e^{\frac{n-2}{2}v},
\end{equation*}
where $v$ is of the form \eqref{nbubble} with $\bar x\in\bR^n$, $a,b>0$ satisfying $2p(n-p)b^2 a^{-2}=1$; 
\begin{equation*}
G_p(\lambda(A_{\bar{g}_u}))=0,~ u>0,~\text{and}~u\in C^0~ \text{in}~ \bR^n~ \implies u~\text{is constant};
\end{equation*}
\begin{equation*}
   G_p(\lambda(A_{\bar{g}_u}))=0,~u>0,~\text{and}~u\in C^0~ \text{in}~ \bR^n\setminus\{0\}~ \implies~  u~ \text{is radially symmetric}.
\end{equation*}
However, none of the above results hold when $p=n-1$ , as shown by counterexamples given in Remark \ref{rmk-230615-1129} and \ref{remark-1.3}.
\end{example}

\begin{example}
        Let $(M^n,g)$ be a closed, smooth, $n$-dimensional Riemannian manifold with positive scalar curvature, $n\geq 3$. For $0<q<4/(n-2)$ and $2\leq i\leq n$, there exists a positive function $u\in C^{2,\alpha}(M^n)$ such that
        \begin{equation*}
            (\lambda_1+\dots+\lambda_i)(Ric_{g_u})=u^{-q}\quad\text{on}~M^n.
        \end{equation*}
        Assume further that $(M^n,g)$ is locally conformally flat. For $1\leq i\leq n$, there exists a positive function $u\in C^{2,\alpha}(M^n)$ such that
        \begin{equation*}
            (\lambda_1+\dots+\lambda_i)(Ric_{g_u})=1 \quad\text{on}~M^n.
        \end{equation*}        
\end{example}

\begin{example}
        Let $(M^n,g)$ be a closed, smooth, $n$-dimensional Riemannian manifold with positive scalar curvature, $n\geq 3$. For $0<q<4/(n-2)$ and $n/2\leq p\leq n-2$, there exists a positive function $u\in C^{2,\alpha}(M^n)$ such that
        \begin{equation*}
            G_p(\lambda(A_{g_u}))=u^{-q} \quad\text{on}~M^n.
        \end{equation*}
        Assume further that $(M^n,g)$ is locally conformally flat. For $n/2\leq p\leq n-1$, there exists a positive function $u\in C^{2,\alpha}(M^n)$ such that
        \begin{equation*}
            G_p(\lambda(A_{g_u}))=1\quad\text{on}~M^n.
        \end{equation*}        
\end{example}

\subsection{Ideas of the proofs}
We now describe the strategies and challenges involved 
in establishing the Liouville-type theorems and local gradient estimates in dimensions $n\geq 3$.
Here and throughout the paper, we use $B_r(x)$ to denote the ball of radius $r$ centered at $x$ in $\bR^n$, $B_r$ means $B_r(0)$, and $B$ means $B_1(0)$ unless otherwise stated.

Theorem \ref{thm-230525-0321} is proved by the method of moving spheres, a variant of the method of moving planes,  as in the proof of Theorem A in \cite{LiLiacta}, see also
\cite{MR2001065} and \cite{Li-Zhu_Duke}.
The proof of Theorem A crucially relied on the superharmonicity of $u$ in 
addressing the singularity of $u$ at infinity.
The novelty in the proof of Theorem \ref{thm-230525-0321} 
is the treatment 
of the singularity of $u$ at infinity, where $u$ is not necessarily superharmonic since $\Gamma$ is not necessarily contained in $\Gamma_1$.  Our treatment makes use of earlier results by Caffarelli, Li and Nirenberg in \cite{CLN3} where they introduced the concept of lower- (upper-)conical to extend a supersolution (subsolution) of a fully nonlinear degenerate elliptic equation in a punctured ball across the puncture. 
We prove that if $u$ is lowerconical at $0$ then ``$f(\lambda(A^u))\geq u^{-p}$ in $B\setminus\{0\}$'' $\Rightarrow$ ``$f(\lambda(A^u))\geq u^{-p}$ in $B$''.  We also prove that when $\bl\notin\overline\Gamma$, ``$\lambda(A^u)\in\overline\Gamma$ in $B\setminus\{0\}$'' $\Rightarrow$ ``$\liminf_{x\to 0}u(x)>0$ and $u$ is lowerconical at $0$''. 
Moreover, we demonstrate that whenever $\bl\in\overline\Gamma$, there are counterexamples for both assertions.
The proof of Theorem \ref{thm-230525-0321} also makes use of the strong comparison principle for fully nonlinear elliptic equations for a $C^{1,1}_{loc}$ (lower semi-continuous) supersolution and an upper semi-continuous ($C^{1,1}_{loc}$) subsolution as developed in \cite{CLN3} and \cite{MR4181003}.

The proof of Theorem \ref{thm-230329-1327} is  divided into two cases: $\bl\notin\overline\Gamma$ and $\bl\in\Gamma$.  
 When $\bl\notin\overline\Gamma$, the proof follows the strategy in \cite{Liarma2007,Li2009}: 
 It is sufficient to prove that 
``$\lambda(A^w)\in\p\Gamma$ in $B$, $\lambda(A^u)\in\p\Gamma$ in $B\setminus \{0\}$,  and $w\leq u$ on $\p B$''
 $\Rightarrow$ ``$w\leq u$ in $B\setminus \{0\}$''.
 Once again, the proof in \cite{Liarma2007,Li2009} crucially relied on
 the superharmonicity of $u$
 stemming from the condition $\Gamma\subset \Gamma_1$ there. 
Our novelty in this case is a similar treatment of the isolated singularity as that of Theorem \ref{thm-230525-0321}. We prove that if $u$ is lowerconical at $0$ then ``$\lambda(A^u)\in\overline{\Gamma}$ in $B\setminus\{0\}$'' $\Rightarrow$ ``$\lambda(A^u)\in\overline\Gamma$ in $B$". We also make use of the above mentioned criteria of lowerconical behavior of $u$ near $0$ (see Proposition \ref{removablesingularity}), together with a comparison principle established by Li, Nguyen and Wang in \cite{MR3813247}. However
 the strategy does not work when $\bl\in\Gamma$, since
 for any such $\Gamma$, there exist smooth $u$ and $w$ on
 $B\setminus\{0\}$ and $B$ respectively satisfying 
 $\lambda(A^w)\in\p\Gamma$ in $B$, $\lambda(A^u)\in\p\Gamma$ in $B\setminus \{0\}$,  and $w=u$ on $\p B$, but $w> u$ in $B\setminus \{0\}$.  Our proof in this case uses a different strategy. 
 In fact, we have proved a more general result: Solutions of 
 $\lambda(A^u)\in \bR^n\setminus\Gamma$ in $\bR^n$ 
 must be constant (see Theorem \ref{thm-230528-0404}).

The proof of Theorem \ref{230918-1857} is also divided into two cases:   $\bl\notin\overline\Gamma$ and $\bl\in\Gamma$.  
When $\bl\notin\overline\Gamma$, the proof is in similar spirit to that of
Theorem \ref{thm-230329-1327} in the corresponding case.
The proof in the case 
 $\bl\in\Gamma$ is  intricate and relies on the utilization of several key results. 
 First, there is a criterion
 for the upperconical 
 behavior of $u$ (see Proposition \ref{230805-2330}); 
 Second, if $u$ is upperconical at $0$ then 
 ``$\lambda(A^u)\in \mathbb R^n\setminus \Gamma$ in 
 $B\setminus\{0\}$'' $\Rightarrow $ ``$\lambda(A^u)\in \mathbb R^n\setminus \Gamma$ in 
 $B$"; Thirdly,  a Harnack inequality holds: 
 ``$\lambda(A^u)\in\p\Gamma$ in $B_2$''  $\Rightarrow$ ``$|\nabla \log u|\leq C(\Gamma)$ in $B_1$" (see Proposition \ref{prop-230614-1100}); 
 The last is Theorem \ref{thm-230528-0404}.

Theorem \ref{localgradest} and \ref{nlocalgradest} are proved using a method developed in \cite{Li2009} to establish local gradient estimates 
of solutions assuming a one-sided bound of the solutions:  It suffices to 
establish a Liouville-type theorem and to derive
local gradient estimates assuming a
two-sided bound of solutions. We prove local gradient estimates
assuming a two-sided bound of solutions in Theorem \ref{ahequlocalgradest}, which holds for more general equations on more general cones $\Gamma$. Using Theorem \ref{thm-230329-1327} and \ref{1327-prime}, we prove
the local gradient estimates assuming a one-sided  bound (see Theorem \ref{thm-230603-1124}). 
It is important to note that
 whenever the required 
 Liouville-type theorem necessary to implement 
the above strategy 
fails (i.e., when $\bl\in\p\Gamma$ for equation  \eqref{pequ}
and $-\bl\in\p\Gamma$
for equation \eqref{equ} respectively), the 
 local gradient estimates assuming a one-sided bound of solutions actually do not hold, even though
the estimates assuming a two-sided bound remain valid.

The proof of Theorem \ref{exist-lcf} utilizes the theorem of Schoen and Yau in \cite{Schoen-Yau} on the existence of developing maps, where the assumptions of local conformal flatness and positive scalar curvature are used.
The $C^0$ and $C^1$ estimates for solutions are obtained as in \cite[Proof of Theorem 1.1]{LiLiacta}.
These estimates enable us to apply Caffarelli's theorems \cite[Theorem 2,3]{Caff89} to obtain $C^{1,\alpha}$ and $C^{2,\alpha}$ estimates. The compactness part of Theorem \ref{exist-lcf} is established.
For the existence part of Theorem \ref{exist-lcf}, we
make a homotopy connecting $(f,\Gamma)$ to $(\sigma_1,\Gamma_1)$ and apply the degree theory for fully nonlinear elliptic equations in \cite{Li_degree}. 
The total degree of solutions of $(\sigma_1,\Gamma_1)$ is $-1$
since the scalar curvature is positive. 
The existence part of Theorem \ref{exist-lcf} then follows from the homotopy invariance of the degree.
To prove Theorem \ref{exist-nlcf-compactness},
we first establish an upper bound of solutions by applying Theorem \ref{thm-230525-0321} and local derivative estimates. 
Theorem \ref{localgradest} enables us to obtain the $C^0$ and $C^1$ estimates of solutions. The remaining parts of Theorem \ref{exist-nlcf-compactness} are proved similarly to Theorem \ref{exist-lcf}. Theorem \ref{exist-nlcf} is proved by Theorem \ref{exist-nlcf-compactness} and a degree argument as in the proof of Theorem \ref{exist-lcf}.

\smallskip

As a byproduct of our studies on isolated singularities, we obtain the following
\begin{theorem}\label{230805-2045}
    For $n\geq 3$, let $\Gamma$ satisfy \eqref{eqn-230331-0110} with $\bl\notin\overline{\Gamma}$ and $-\bl\in\Gamma$. Assume that $u$ is a positive continuous viscosity solution of
    $\lambda(A^u)\in\p\Gamma$ in $B\setminus\{0\}$ satisfying $u(x)=o(|x|^{2-n})~\text{as}~x\rightarrow 0$. Then $u$ can be extended as a positive function in
     $C^{0,1}_{loc}(B)$ which  satisfies $\lambda(A^u)\in\p\Gamma$ in $B$.
\end{theorem} 

For $1\leq k<n/2$, $\Gamma=\Gamma_k$ and $\Gamma=\bR^n\setminus(-\overline{\Gamma_k})$ satisfy the conditions of above theorem. 
When $\Gamma=\Gamma_1$, equation
$\lambda(A^u)\in\p\Gamma$  is $\Delta u=0$. When $\Gamma=\Gamma_k$, $2\leq k<n/2$, Li and Nguyen \cite{Bocherpaper} proved that $u$ can be extended to a positive function in $C^{0,\beta}_{loc}(B)$ for any $\beta\in(0,1)$, though it was not known that $\lambda(A^u)\in\p\Gamma_k$ in $B$. 

The condition ``$\bl\notin\overline{\Gamma}$ and $-\bl\in\Gamma$" is equivalent to stating that $\{\lambda\in\bR^n\mid \lambda_n+\widetilde{\mu}\lambda_{n-1}>0\}\subset\Gamma\subset\{\lambda\in\bR^n\mid\lambda_1+\mu\lambda_2>0\}$ for some $\mu,\widetilde{\mu}\in[0,\infty)$. See Lemma \ref{240901-1559}.

  The assumptions $\bl\notin\overline{\Gamma}$ and $-\bl\in\Gamma$ 
  in Theorem \ref{230805-2045} are optimal in the sense that if one of them is not satisfied then 
   there exists a positive smooth solution of $\lambda(A^u)\in\p\Gamma$ in $B\setminus\{0\}$ satisfying $u(x)=o(|x|^{2-n})$ as $x\to 0$ which can not be extended as
   a solution in $B$.  Indeed a solution in $B$ must be both lowerconical and upperconical at $0$
   (see Remark 
    \ref{necessityremark}), while solutions in $B\setminus\{0\}$ given by
    Example \ref{reex1}--\ref{reex4} do not have such property. Moreover, the assumption $u(x)=o(|x|^{2-n})$ is also optimal since $A^u$ vanishes if we take $u(x)=|x|^{2-n}$.
    
\subsection{Organization of the paper}
In $\S \ref{sec-2}$ we collect some preliminaries
which include  the definition of viscosity solutions, a comparison principle, and two natural numbers associated with cone $\Gamma$. In $\S \ref{sec-3}$ we classify all radial viscosity solutions of equation \eqref{eqn-221204-0213}, some of which serve as comparison functions in our studies of isolated singularities. In particular, we obtain in $\S\ref{sec-3.3}$ necessary and sufficient conditions of solvability of the Dirichlet problem for equation \eqref{eqn-221204-0213} on any annulus with constant boundary conditions. In $\S \ref{sec-4}$ we discuss the removability of isolated singularities for equations  \eqref{nequation} and \eqref{eqn-221204-0213}, where the concepts of lower- and upper-conical singularities are given. In \S \ref{lowerconicalcrit} and \S \ref{upperconicalcrit},
we derive criteria for lower- and upper-conical behaviors and discuss their optimality.
Theorem \ref{thm-230525-0321} is proved in \S \ref{sec-5}. Theorem \ref{thm-230329-1327} is proved in \S \ref{sec-6.1}. Theorem \ref{230918-1857} and \ref{230805-2045} are proved in \S \ref{sec-6.2}. Theorem \ref{localgradest} and \ref{nlocalgradest} are proved in \S \ref{sec-7}.
Theorem \ref{exist-lcf}--\ref{exist-nlcf} are proved in \S \ref{sec-exist-and-compactness}. In \S \ref{examplesection} we give further examples of $(f,\Gamma)$ where the $\Gamma$'s are not necessarily contained in $\Gamma_1$.

\section{Preliminaries}\label{sec-2}

\subsection{Viscosity solutions and comparison principles}
Let $\Omega\subset\bR^n$ be an open set. We denote $LSC(\Omega)$ as the set of lower semi-continuous functions on $\Omega$, that is, $v\in LSC(\Omega)$ if and only if $v:\Omega\longrightarrow (-\infty, +\infty]$, $v \not\equiv +\infty$ in $\Omega$, satisfies
\[\liminf\limits_{x\rightarrow x_0} v(x) \geq v(x_0)~~\text{for every}~x_0\in \Omega.\] 
We say $v\in USC(\Omega)$ whenever $-v\in LSC(\Omega)$. 

\begin{definition} \label{def-230607-0451}
For $\Gamma\subset\bR^n$ satisfying \eqref{eqn-230331-0110}
, $n\geq 2$, a function $v: \Omega \rightarrow (-\infty,\infty]$ (respectively, $v: \Omega \rightarrow [-\infty, \infty)$) is said to satisfy
		\begin{equation*}
		    \lambda (A[v]) \in \overline{\Gamma} \,\,  \text{in}\,\,\Omega \quad (\text{respectively}, \lambda (A[v]) \in\,\,\bR^n\setminus \Gamma \,\,  \text{in}\,\,\Omega)
      \end{equation*}
		in the viscosity sense, if $v \in LSC(\Omega)$ (respectively, $v \in USC(\Omega)$)
  and for any $x_0\in \Omega$  and any $\varphi \in C^2$ satisfying $\varphi (x_0) = v (x_0)$ and $\varphi (x) \leq v(x)$ near $x_0$ (respectively, $\varphi(x_0) = v(x_0)$, $\varphi (x) \geq v(x)$ near $x_0$),  there holds
		\begin{equation*}
		    \lambda(A[v]) (x_0) \in \overline{\Gamma} \,\, (\text{respectively,} \,\, \bR^n\setminus \Gamma).
		\end{equation*}
  We also call such $v$ a viscosity supersolution (respectively, subsolution) of $\lambda(\av) \in \p \Gamma$ in $\Omega$.
        We say that $v$ is a viscosity solution of $\lambda(\av)\in \p \Gamma$ in $\Omega$ if $v$ is both a viscosity subsolution and a viscosity supersolution in $\Omega$. In particular, a viscosity solution $v$ is continuous in $\Omega$. 
\end{definition}

It is a standard fact that, due to $\Gamma+\Gamma_n\subset \Gamma$, any $C^{1,1}$ solution $v$
(in almost everywhere sense) is also a viscosity solution; see e.g. \cite{Caffarelli1995FullyNE}.

\begin{theorem}[(\cite{MR3813247})]\label{thm-220721-0849}
For $n\ge 2$, let 
    $\Omega\subset\bR^n$ be a bounded open set and $\Gamma$ satisfy \eqref{eqn-230331-0110}. Assume that $v_1\in LSC(\overline{\Omega})$,   $ v_2\in USC(\overline{\Omega})$, 
    $v_1$ $\geq$ ($>$, respectively) $v_2$ on $\p\Omega$, 
    $\lambda(A[v_1])\in \overline{\Gamma}$ and 
    $ \lambda (A[v_2])\in \mathbb{R}^n \setminus \Gamma$ 
    in $\Omega$.  Then
   $v_1$ $\geq$ ($>$, respectively) $v_2$ in $\Omega$.
\end{theorem}
The above theorem is a consequence of \cite[Theorem~3.2]{MR3813247}.  Indeed, denote
$
    U := \{M\in S^{n \times n}: \lambda( M) \in \Gamma \}$ and $
\psi_i=2^{-1}e^{-2 v_i}$. Then 
$
   A[v_i] = F[\psi_i]:= 
   \nabla^2 \psi_i - (2\psi_i)^{-1} \nabla \psi_i \otimes \nabla\psi_i - (4\psi_i)^{-1} |\nabla \psi_i|^2 I$.  
   We know that $F[\psi_1]\in \overline U$ and $F[\psi_2] \notin U$ in $\Omega$.  For 
   small  $c>0$, $c\psi_1<\psi_2$ on $\overline \Omega$.
    By \cite[Theorem~3.2]{MR3813247} and using $\psi_1\le \psi_2$ on
   $\partial \Omega$,
   we may increase $c$ to $1$ and obtain $c\psi_1<\psi_2$ in $\Omega$ for all $c<1$.

\subsection{Two useful numbers related to $\Gamma$}
For a cone $\Gamma$ satisfying \eqref{eqn-230331-0110}, the following two numbers are useful:
\begin{equation} \label{eqn-220910-0831}
      \mgn\coloneqq \inf \{c\mid  (c,-1,\dots,-1)\in \overline{\Gamma}\}, 
    \ 
    \ 
      \mg \coloneqq \sup \{c\mid  (-c,1,\dots,1)\in \overline{\Gamma}\}.
\end{equation}
Here we follow the convention: $\inf \emptyset = +\infty$. 
By \eqref{eqn-230331-0110}, it holds that $0\leq \mgn,\mg\leq +\infty$.

Actually, we have $\mgn = \beta ( \in [0,\infty) )$ if $(\beta,-1,\ldots,-1) \in \p\Gamma$ and $\mgn = +\infty$ if $(1,0,\ldots,0) \in \p\Gamma$. Similarly, $\mg = \beta ( \in [0,\infty) )$ if $(-\beta,1,\ldots,1) \in \p\Gamma$ and $\mg = +\infty$ if $(-1,0,\ldots,0) \in \p\Gamma$.
In particular, $\mgn=1$ (respectively, $\mgn>1$, $\mgn<1$) if and only if $\bl\in\p\Gamma$ (respectively, $\bl\notin\overline{\Gamma}$, $\bl\in\Gamma$).   Similarly, 
   $\mg=1$ (respectively, $\mg>1$, $\mg<1$) if and only if $-\bl\in\p\Gamma$ (respectively, $-\bl\in\Gamma$, $-\bl\notin\overline\Gamma$).

The two numbers $\mu_\Gamma^+$ and $\mu_\Gamma^-$ are related by a bijection from  the set of 
cones satisfying \eqref{eqn-230331-0110} to itself.
Let  $\Phi : \Gamma \mapsto \bR^n\setminus (-\overline{\Gamma})$. It is easy to check that 
$\Phi^2 = \operatorname{id}$, and
\begin{equation*}
    \mg=\mu^-_{\Phi(\Gamma)}.
\end{equation*}

Clearly when $\Gamma \subset \Gamma'$, we have $\mu_{\Gamma}^+ \leq \mu_{\Gamma'}^+$ and $\mu_{\Gamma}^- \geq \mu_{\Gamma'}^-$.
    It is easy to see that 
     \begin{equation*}
        \{\mg\mid \Gamma \,\,\text{satisfies} \,\,\eqref{eqn-230331-0110} \}
        =
        \{\mgn\mid \Gamma \,\,\text{satisfies} \,\,\eqref{eqn-230331-0110} \}
        =
        [0,+\infty].
    \end{equation*}
    The definitions of $\mg$ and $\mgn$ are consistent with \cite{Bocherpaper} where $\Gamma \subset \Gamma_1$ was assumed. 
    Note that 
    \begin{equation*}
        \{\mg \mid  \Gamma \subset \Gamma_1\,\,\text{satisfying}\,\,\eqref{eqn-230331-0110} \} = [0,n-1],
        \,\,
        \{\mgn \mid  \Gamma \subset \Gamma_1\,\,\text{satisfying}\,\,\eqref{eqn-230331-0110} \} = [n-1,+\infty].
    \end{equation*} 

    For $\Gamma_k$ defined in the introduction, $\mu_{\Gamma_k}^+ = \mu_{\bR^n\setminus(-\overline{\Gamma_k})}^- = (n-k)/k$ and  $\mu_{\Gamma_k}^- = \mu_{\bR^n\setminus(-\overline{\Gamma_k})}^+ = +\infty$.

    As in \S \ref{intro-example-sec}, we order $\lambda=(\lambda_1,\dots,\lambda_n)$ as $\lambda_1\geq\dots\geq\lambda_n$.
     \begin{lemma}\label{240901-1559}
         For $n\geq 2$, let $\mu\in[0,\infty)$ be a constant.
         \begin{enumerate} [label=(\alph*)]   
    \item The largest cone $\Gamma$ satisfying \eqref{eqn-230331-0110} and $(\mu,-1,\dots,-1)\in\p\Gamma$ is $\{\lambda\in\bR^n\mid \lambda_1+\mu\lambda_2>0\}$.\\
    \item The smallest cone $\Gamma$ satisfying \eqref{eqn-230331-0110} and $(-\mu,1,\dots,1)\in\p\Gamma$ is $\{\lambda\in\bR^n\mid \lambda_n+\mu\lambda_{n-1}>0\}$.    \end{enumerate}
     \end{lemma}
     \begin{proof}
         It suffices to prove one of the two statements since the other one follows by taking $\bR^n\setminus\overline{-\Gamma}$. We provide a proof for part (a).
         It is clear that $\Gamma_\mu\coloneqq \{\lambda\in\bR^n\mid\lambda_1+\mu\lambda_2>0\}$ satisfies \eqref{eqn-230331-0110} and $(\mu,-1\dots,-1)\in\p\Gamma_\mu$. 
         Let $\Gamma$ satisfy \eqref{eqn-230331-0110} with $(\mu,-1,\dots,-1)\in\p\Gamma$. We will prove that for any 
         $\lambda\in\Gamma$, it holds
         $\lambda_1+\mu\lambda_2>0$. Suppose the contrary that  $\lambda_1+\mu\lambda_2\leq 0$ for some $\lambda\in\Gamma$. Since $-\Gamma_n\cap\Gamma=\emptyset$ and $\lambda_1 \geq \lambda_2$, we have $\lambda_1>0>\lambda_2$.
         By $\Gamma+\Gamma_n\subset\Gamma$ and $\lambda_2 \geq \lambda_i$ for any $2\leq i\leq n$, we have $(\lambda_1,\lambda_2,\dots,\lambda_2)\in\Gamma$. By the cone property of $\Gamma$, $(-{\lambda_1}/{\lambda_2},-1,\dots,-1)\in\Gamma$. By $-{\lambda_1}/{\lambda_2}\leq \mu$ and $\Gamma+\Gamma_n\subset\Gamma$, we have $(\mu,-1,\dots,-1)\in\Gamma$. A contradiction. 
     \end{proof}

\section{Classification of radial viscosity solutions}\label{sec-3}

In this section, we study the solvability of viscosity solution $v \in C^0(\{ a \leq |x| \leq b\} )$ of the Dirichlet problem 
\begin{equation}\label{annulusDirichletp}
\begin{cases}
    \lambda(\av)\in \p\Gamma,~~\text{in}~~\{a<|x|<b\},\\
    v \vert_{\p B_a(0)}=\alpha,~~v \vert_{\p B_b(0)}=\beta,
\end{cases}
\end{equation}
where $0 < a < b < \infty$ and $\alpha, \beta \in \bR$. We also give explicitly all solutions of \eqref{annulusDirichletp} when it is solvable. 

By the comparison principle and the  rotation invariance, viscosity solutions of \eqref{annulusDirichletp} are unique and radially symmetric, i.e., $v = v(|x|)$. Throughout this section and the whole article, we always write $r=|x|$ and do not distinguish $v(r)$ with $v(x)$ when it is radially symmetric.

\subsection{Smooth radial solutions} \label{sec-230521-1141}

Recall the definitions of $\mg, \mgn \in [0,\infty]$ in \eqref{eqn-220910-0831}.
\begin{lemma}\label{lem-220510-0218}
For $n\geq 2$, let $\Gamma$ satisfy \eqref{eqn-230331-0110}. Any possible $C^2$ radial solution $v$ of $\lambda(\av) \in \p\Gamma$ is one of the followings (or a restriction on a subinterval on their domains of definitions):
\begin{enumerate} [label=(\alph*)]   
    \item $v = C_1 + C_2 \log r,~r\in (0,\infty)$ with $C_1 \in \bR, C_2 \in (-2,0)$, when $\mg = 1$;
    \item $v = \frac{2}{\mg - 1} \log (C_3 r^{-\mg + 1} + C_4)$, $r\in (0,\infty)$ with $C_3,C_4>0$, when $\mg \in [0,1) \cup (1,\infty)$;
    \item $v = C_5 + C_6 \log r$, $r\in(0,\infty)$ with $C_5 \in \bR, C_6 \in (-\infty, -2) \cup (0,\infty)$, when $\mgn = 1$;  
    \item $v = \frac{2}{\mgn - 1} \log(C_7 r^{-\mgn +1} + C_8)$, $r\in \{s: C_7s^{-\mgn +1} + C_8>0\}$ with $C_7 \cdot C_8 < 0$, when $\mgn \in [0,1) \cup (1,\infty)$;
    \item $v = C$ or $v = C - 2\log r$, $r\in (0,\infty)$ with $C \in \bR$.
\end{enumerate}
\end{lemma}
When $\Gamma \subset \Gamma_1$ and $n \geq 3$, viscosity solutions are classified in \cite[Theorem~2.2]{Bocherpaper}, and they are all smooth.

\begin{proof}[Proof of Lemma \ref{lem-220510-0218}]
First, a calculation gives
\begin{equation*}
    \av = e^{-2v} \left( \nu I_n + (V-\nu) \frac{x}{r}\otimes\frac{x}{r} \right),
\end{equation*}
where
	\begin{equation*}
		V = -v'' + \frac{1}{2}(v')^2
  \,\,\text{and}\,\,
  \nu = -\frac{1}{r}v' - \frac{1}{2}(v')^2 = - \frac{1}{2} v' (v + 2\log r)'.
	\end{equation*}
It follows that
 \begin{equation}\label{230930-1027}
     \lambda(\av) = e^{-2v} \left( V, \nu, \ldots, \nu \right).
 \end{equation}
\textbf{Claim}: if $\nu \neq 0$ on an interval $(a,b) \subset (0,\infty)$, then for some real numbers $C_1, C_2$, and $\mu \neq 1$, either $v = C_1 + C_2 \log r$ or $v = \frac{2}{\mu - 1}\log(C_1 r^{1-\mu} + C_2)$.

To prove the claim, note that since $\nu$ does not change sign, the equation $\lambda(\av) \in \p\Gamma$ takes the form
\begin{equation} \label{eqn-230521-0223}
    0 = V + \mu \nu \left( =  -v'' + \frac{1-\mu}{2}(v')^2 - \frac{\mu}{r}v' \right) \quad \text{on}\,\,(a,b),
\end{equation}
where $\mu =  \mg$ when $\nu > 0$ and $\mu =  \mgn$ when $\nu < 0$.
Taking the substitution $w = e^{\frac{\mu - 1}{2}v}$ when $\mu \neq 1$ or $w= v$ when $\mu = 1$, we can rewrite \eqref{eqn-230521-0223} as
	\begin{equation*}
		w'' + \frac{\mu}{r}w'=0,
	\end{equation*}
	for which the solutions are given by
	\begin{equation*}
		w = \begin{cases}
			C_1\log r + C_2\quad\text{if } \mu = 1,\\
			C_3 r^{1-\mu} + C_4\quad \text{if } \mu \neq 1,\\
		\end{cases}
	\end{equation*}
	Transforming back, we have
	\begin{equation} \label{eqn-230521-0201}
		v =
  \begin{cases}
			C_1 + C_2 \log r \quad \text{if }\mu = 1,\\
			\frac{2}{\mu - 1} \log(C_3 r^{1-\mu} + C_4)  \quad \text{if }\mu \neq 1.
		\end{cases}
	\end{equation}
Hence, the claim is proved. Now suppose that $v$ is a solution defined on $(a',b')$. If on $(a',b')$, $\nu(r) \equiv 0$, i.e., $v'(v + 2\log r)' \equiv 0$, it is easy to see
	\begin{equation*}
		\text{either} \,\, v = C \quad \text{or} \quad v = C - 2\log r.
	\end{equation*}
Direct computation shows for such solutions, $V\equiv 0$. This gives the solutions in (e). 
Otherwise, there exists some $\overline{r} \in (a',b')$, such that $\nu(\overline{r}) \neq 0$. Let $(a,b)$ be the maximal interval containing $\overline{r}$, such that $\nu \neq 0$.

\textbf{Case 1}: $\nu>0$ on $(a,b)$. In this case, on $(a,b)$, $v$ is given by \eqref{eqn-230521-0201} with $\mu = \mg$. From $\nu >0$, we have $v'(v + 2\log r)' < 0$, which gives the conditions on the $C_i$'s. Moreover, from the expression one could simply check that $\nu$ never touches zero. Hence, we must have $(a,b) = (a',b')$. This gives us solutions in (a) and (b).

\textbf{Case 2}: $\nu < 0$ on $(a,b)$. Similar discussion shows that $(a,b) = (a',b')$ and $\nu$ never reaches zero, and $v$ is given by the expressions in (c) and (d).
\end{proof}

\subsection{Non-smooth viscosity solutions}\label{viscosityradialsolusection}
For a class of cones involving $\Gamma = \bR^n\setminus (-\overline{\Gamma_k})$, $k \geq 2$, we prove that $v := \max \{ -2\log r , 0\}$
is a viscosity solution. Clearly, $v$ is not $C^1$.

\begin{lemma}\label{verifygradeintjumpsolu}
     For $n\geq 2$, let $\Gamma$ satisfy \eqref{eqn-230331-0110} with $\mg = +\infty$.
     Then for any $C_1, C_2 \in \bR$, $v := \max \{C_1 -2\log r , C_2\}$
 satisfies $\lambda(\av)\in\p\Gamma$ on $\bR^n\setminus\{0\}$ in the viscosity sense.
\end{lemma}

\begin{proof}
     Without loss of generality we only need to prove for $v$ with $C_1=C_2 = 0$. By Lemma \ref{lem-220510-0218}, $-2\log|x|$ and $0$ are both solutions of $\lambda(A[w])\in\p\Gamma$ in $\bR^n\setminus\{0\}$. Therefore, $v$, as their maximum, satisfies $\lambda(\av)\in \bR^n\setminus\Gamma$ in $\bR^n\setminus\{0\}$. In the rest of the proof, it suffices to show $\lambda(\av) \in \overline{\Gamma}$ in $B_2 \setminus B_{1/2}$.
     
     For each $\mu \in (1,\infty)$, set $v_\mu(r) := 2 (\mu - 1)^{-1} \log (C_3 r^{-\mu +1} + C_4)$ in $(0,\infty)$, where $C_3 = C_4 = 2^{\mu - 1} (2^{\mu -1} +1)^{-1}$. In particular, $v_\mu (1/2) = 2\log 2 = v(1/2)$ and $v_\mu (2) = 0 = v(2)$. A calculation as in Lemma \ref{lem-220510-0218} yields $\lambda(A[v_\mu]) = C(r) (-1, 1/\mu,\ldots, 1/\mu)$, $C(r)=-2^{-1}\mu e^{-2v}v' (v+2\log r)'>0$. Since $\mg=\infty$, $\lambda(A[v_\mu])\in \overline{\Gamma}$ in $\bR^n\setminus\{0\}$. It is not hard to check that
    \begin{align*}
         v_\mu 
         = 
         \frac{2}{\mu - 1} \left( \log \frac{2^{\mu-1}}{2^{\mu - 1} + 1} + \log (r^{1-\mu} + 1) \right)
         \rightarrow v
         \,\,\text{uniformly in}\,\,[1/2,2],
     \end{align*}
     as $\mu \rightarrow \infty$. In fact, one only need to check that $v_\mu$ is non-increasing in $\mu$ and pointwisely converges to $v$. Hence, $v$ satisfies $\lambda (A[v]) \in \overline{\Gamma}$ in $B_2\setminus B_{1/2}$. The lemma is proved.     
\end{proof}

\subsection{Solvability of Dirichlet problems}\label{sec-3.3}
In this subsection, we study \eqref{annulusDirichletp}.
\begin{proposition} \label{prop-230521-0510}
For $n\geq 2$, let $\Gamma$ satisfy \eqref{eqn-230331-0110} with $\mg < +\infty$. Then there exists a unique viscosity solution of \eqref{annulusDirichletp} if and only if one of the following conditions holds:
    \begin{enumerate}[label=(\alph*)]
    \item $\beta - \alpha \in \left[ -2(\log b - \log a ) , 0 \right]$.
    \item $\beta - \alpha \in \left( -\infty, -2(\log b - \log a ) \right) \cup \left( 0, +\infty \right)$ and $\mgn < +\infty$.
    \end{enumerate}
    Moreover, all solutions are smooth. In case (a), the solutions are given by Lemma \ref{lem-220510-0218} (a,b,e). In case (b), the solutions are given by Lemma \ref{lem-220510-0218} (c,d).
\end{proposition}

\begin{proposition} \label{prop-230521-0515}
    For $n\geq 2$, let $\Gamma$ satisfy \eqref{eqn-230331-0110} with $\mg = +\infty$. Then for every $0<a<b<\infty$ and $\alpha,\beta \in \bR$, there exists a unique viscosity solution of \eqref{annulusDirichletp}.
    \begin{enumerate}[label=(\alph*)]
        \item When $\beta - \alpha \in \left( -2(\log b - \log a ) , 0 \right)$, the solution is in $C^{0,1}$ but not $C^1$, given by
            $v = \max\{ \alpha - 2 (\log r - \log a), \beta \}.$
        \item When $\beta - \alpha \in \left( -\infty, -2(\log b - \log a ) \right] \cup \left[ 0, +\infty \right)$, the solution is smooth, given by Lemma \ref{lem-220510-0218} (c,d,e).
    \end{enumerate}
\end{proposition}

\begin{remark}
Propositions \ref{prop-230521-0510} and \ref{prop-230521-0515} give  necessary and sufficient conditions on $a,b,\alpha,\beta$ for the solvability of \eqref{annulusDirichletp}.
\end{remark}
To prove Proposition \ref{prop-230521-0510}, we make use of the following lemma.
\begin{lemma} \label{lem-230525-1216}
    For $n\geq 2$, let $\Gamma$ satisfy \eqref{eqn-230331-0110} with $\mgn = \infty$. Suppose $v = v(r) \in C^0$ satisfies $\lambda(\av) \in \overline{\Gamma}$ on an annulus $B_b\setminus \overline{B_a}$ in the viscosity sense. Then $v$ is non-increasing and $v + 2\log r$ is non-decreasing.
\end{lemma}
 
\begin{remark}
    The condition $\mgn = \infty$ is optimal in the sense that for any cone $\Gamma$ with $\mgn < \infty$, there are solutions which either are strictly increasing themselves or have $v + 2\log r$ strictly decreasing. See Lemma \ref{lem-220510-0218} (c,d).
    \end{remark}
The above Lemma was proved in \cite{Bocherpaper} when $\Gamma \subset \Gamma_1$, $\Gamma$ is convex, and $n\geq 3$.
\begin{proof}[Proof of Lemma \ref{lem-230525-1216}]
    We only need to prove the monotonicity of $v$ since the monotonicity of $v + 2\log r$ follows by the monotonicity of $\hat{v} = -2\log|x| + v (x/|x|^2)$ which satisfies the hypothesis of the lemma on $B_{a^{-1}} \setminus \overline{B_{b^{-1}}}$. We only need to prove $v(b) \leq v(a)$ under additional assumptions $v\in C^0(\{a\leq |x| \leq b\})$, $a = 1$, and $v(a) = 0$, which can be achieved by restricting on subintervals and considering $v(a \cdot) - v(a)$.
    
    We prove it by contradiction. Suppose not, then $\beta := v(b) > v(a) = 0$. 
    For $\mu \in (1,\infty)$, let $v_\mu:= \frac{2}{\mu - 1} \log(C_7 r^{1-\mu} + C_8)$, where $C_7 = (1 - e^{\frac{\mu-1}{2}\beta})/(1-b^{1-\mu}) < 0$ and $C_8 = (e^{\frac{\mu-1}{2}\beta} - b^{1-\mu})/(1-b^{1-\mu}) > 0$. Clearly, $v_\mu(1) = v(1) = 0$, $v_\mu(b) = v(b) = \beta$. By \eqref{230930-1027} with $v=v_\mu$,
    \begin{equation*}
        \lambda(A[v_\mu]) = C(r) (\mu,-1,\ldots,-1),
        \,\,
        \text{where}\,\,
        C(r) = \frac{1}{2}e^{-2v_\mu} v_\mu' (v_\mu + 2\log r)' >0.
    \end{equation*}
    Since $\mgn = \infty$ i.e. $(1,0,\dots,0)\in\p\Gamma$, by     
    $\Gamma+\Gamma_n\subset\Gamma$ and the cone property of $\Gamma$, $\lambda(A[v_\mu]) \in \bR^n\setminus \overline{\Gamma}$. Hence, by the comparison principle Theorem \ref{thm-220721-0849}, $v \geq v_\mu$ on $(1,b)$.
    Since $v_\mu (r) \rightarrow \beta$ for every $r>1$ as $\mu \rightarrow \infty$ and $v\in C^0(\{1 \leq |x| \leq b\})$, we have $v(1) \geq \beta$, which is a contradiction. The lemma is proved.
\end{proof}

\begin{proof}[Proof of Proposition \ref{prop-230521-0510}]
    We first prove the ``if'' part, by matching the $C^2$ solutions in Lemma \ref{lem-220510-0218} with the boundary condition. 
    By the invariance $v (\cdot) \mapsto v (\cdot/a) - \alpha$, we may assume $a= 1$ and $\alpha = 0$. In case (a), if $\beta = 0$, then $v \equiv 0$. If $\beta = -2\log b$, then $v \equiv -2 \log r$. If $\beta \in (-2\log b, 0)$, we treat $\mg = 1$ and $\mg \neq 1$ separately. When $\mg = 1$, $v = (\beta/\log b) \log r$. When $\mg \neq 1$, $v = \frac{2}{\mg -1} \log(C_3 r^{-\mg+1} + C_4)$, where $C_3, C_4$ are determined by $C_3 + C_4 = 1$ and $C_3 b^{-\mg + 1} + C_4 = e^{\frac{\mg - 1}{2}\beta}$.
In case (b), we treat $\mgn =1$ and $\mgn\neq 1$ separately. When $\mgn =1$, $v = (\beta/\log b) \log r$. When $\mgn \neq 1$, $v = \frac{2}{\mgn -1} \log(C_7 r^{-\mgn+1} + C_8)$, where $C_7, C_8$ are determined by $C_7 + C_8 = 1$ and $C_7 b^{-\mgn + 1} + C_8 = e^{\frac{\mgn - 1}{2}\beta}$. 
As mentioned earlier, the uniqueness follows from the comparison principle Theorem \ref{thm-220721-0849}. The ``if'' part is proved.
    
    For the ``only if'' part, one only needs to prove the non-existence of viscosity solutions when $\beta - \alpha \in \left( -\infty, -2 (\log b - \log a) \right) \cup \left( 0, +\infty \right)$ and $\mgn = +\infty$. This follows from Lemma \ref{lem-230525-1216}.
\end{proof}

\begin{proof}[Proof of Proposition \ref{prop-230521-0515}]
    Again, the uniqueness follows from the comparison principle in Theorem \ref{thm-220721-0849}. 
    The solutions in case (a) come from Lemma \ref{verifygradeintjumpsolu}. For solutions in case (b), again without loss of generality we may assume $a=1$ and $\alpha = 0$. When $\beta = 0$, $v \equiv 0$; when $\beta = -2 \log b$, $v \equiv -2\log r$; when $\beta \in (-\infty,-2\log b) \cup (0,\infty)$ and $\mgn = 1$, $v = (\beta / \log b) \log r$; when $\beta \in (-\infty,-2\log b) \cup (0,\infty)$ and $\mgn \neq 1$, $v = \frac{2}{\mgn -1} \log (C_7 r^{-\mgn +1} + C_8)$. Here, $C_7$ and $C_8$ are the same with those in the proof of Proposition \ref{prop-230521-0510}.    
\end{proof}

\section{On isolated singularities}\label{sec-4}
In the current section, we study the removability of isolated singularities. The following concept was introduced by Caffarelli, Li and Nirenberg in \cite{CLN3}. Let $\Omega\subset\bR^n$ be an open set.
\begin{definition}\label{conical}
	A function $\varphi \in LSC(\Omega)$ is \textbf{lowerconical} at $x_0\in\Omega$ if either $\varphi(x_0) = +\infty$ or $\varphi(x_0) < + \infty$ and for $\forall \epsi >0$, $\forall \eta\in C^{\infty}(\Omega)$, we have
	\[\inf\limits_{x\in\Omega}((\varphi - \eta)(x)-(\varphi - \eta)(x_0)-\epsi |x-x_0|)<0.\]
	We say $\varphi$ is \textbf{upperconical} at $x_0\in\Omega$ if $-\varphi$ is lowerconical at $x_0$.
\end{definition}

\begin{lemma}\label{remove}
	For $n\geq 2$, let $\Gamma$ satisfy \eqref{eqn-230331-0110}. Suppose that $v \in LSC(B)$ (respectively, $v\in USC(B)$) is 
    lowerconical (respectively, upperconical) at $0$ and satisfies 
    $\lambda(\av)\in \overline{\Gamma}$ (respectively, $\lambda(\av)\in \bR^n\setminus\Gamma$) in $B\setminus\{0\}$ in the viscosity sense.	
    Then $\lambda(\av)\in \overline{\Gamma}$ (respectively, $\lambda(\av)\in \bR^n\setminus\Gamma$) in $B$ in the viscosity sense.
\end{lemma}

\begin{remark}\label{necessityremark}
    In fact, the above statement is reversible. Namely, if $\lambda(\av)\in \overline\Gamma$ (respectively, $\lambda(\av)\in \bR^n\setminus\Gamma$) in $B$ in the viscosity sense, then $v$ is lowerconical (respectively, upperconical) at $0$. We include a proof at the end of this section.
\end{remark}

\begin{proof}
	 We only provide a proof for the lowerconical statement since the other one can be proved similarly. The proof is an adaptation of \cite[Proof of Theorem~1.1]{CLN3}. Let $\varphi\in C^2(B)$ satisfy,
  for some $x_0\in B$, $(v-\varphi)(x_0)=0$, $v-\varphi\geq 0$ near $x_0$.  If $x_0\ne 0$, then,
  since $\lambda(\av)\in \overline \Gamma$ in $B\setminus\{0\}$,  $\lambda(A[\varphi](x_0))\in \overline \Gamma$. 
  Now we assume  $x_0=0$.  
  
  Consider, for small $\delta>0$, 
    \begin{equation*}
        \varphi_\delta(x)\coloneqq \varphi(x)-\frac{\delta}{2}|x|^2.
    \end{equation*}
    Since $v$ is lowerconical at $0$, there exists $\{x_i\}\subset B\setminus\{0\}$ such that
   $$
   (v-\varphi_\delta)(x_i)-\frac{1}{i}|x_i|<0,
   $$
   and, as shown in 
   \cite{CLN3},
   $x_i\to 0$ as $i\to \infty$.

Let 
\begin{equation*}
    \varphi^{(i)}_\delta(x)\coloneqq \varphi_\delta (x)+\frac{1}{\sqrt{i}} \frac{x_i}{|x_i|}\cdot x.
\end{equation*}
We have
\begin{equation*}
    (v-\varphi^{(i)}_\delta)(x_i)=(v-\varphi_\delta)(x_i)-\frac{1}{\sqrt{i}}|x_i|<(\frac{1}{i}-\frac{1}{\sqrt{i}})|x_i|<0.
\end{equation*}
Since
\begin{equation*}
    (v-\varphi^{(i)}_\delta)(0)=0,
\end{equation*}
and, using $v(x)\ge \varphi(x)$ for $|x|\le \delta$, 
\begin{equation*}
    (v-\varphi^{(i)}_\delta)(x)\geq \frac{1}{2}\delta^3 + O(\frac{1}{\sqrt{i}})>0,~\forall |x|= \delta ~\text{for large}~i,
\end{equation*}
there exists $\widetilde{x_i}$, $0<|\widetilde{x_i}|<\delta$, such that
\begin{equation*}
   \beta_i:=  (v-\varphi^{(i)}_\delta)(\widetilde{x_i})=\min\limits_{0<|x|<\delta}(v-\varphi^{(i)}_\delta)(x)<0.
\end{equation*}
Namely,
\begin{equation*}
    \psi^{(i)}_\delta (x)\coloneqq \varphi^{(i)}_\delta (x)+\beta_i
\end{equation*}
satisfies $v(\widetilde{x_i})=\psi^{(i)}_\delta(\widetilde{x_i})$ and $v\geq \psi^{(i)}_\delta$ near $\widetilde{x_i}$.
As shown in \cite{CLN3}, 
$\widetilde{x_i}\to 0$.
Thus, using  $\lambda(\av)\in\overline{\Gamma}$ in $B\setminus\{0\}$,  we have $
    \lambda(A[\psi_\delta^{(i)}])(\widetilde{x_i})\in\overline{\Gamma}.
    $
Sending $i$ to infinity and then $\delta $ to $0$ lead to 
$
    \lambda(A[\varphi])(0)\in\overline{\Gamma}.$
\end{proof}

Similarly, we use the notion of lowerconical to study removability of isolated singularities for $f(\lambda(\av))\geq \psi (x,v,\nabla v)$. We first give the following 
\begin{definition}
For $n\geq 2$, let $(f,\Gamma)$ satisfy \eqref{nincreasing} and \eqref{eqn-230331-0110}, and $\psi\in C^0(\Omega\times \bR \times \bR^n)$ be a positive function.  We say a function $v: \Omega \rightarrow (-\infty,\infty]$ (respectively, $v: \Omega \rightarrow [-\infty,\infty)$) satisfies
		\begin{equation*}
		    f(\lambda(\av))\geq \psi(x,v, \nabla v) \,\,  \text{in}\,\,\Omega \quad (\text{respectively},	f(\lambda(\av))\leq \psi(x,v, \nabla v)  \,\,  \text{in}\,\,\Omega)
      \end{equation*} 
		in the viscosity sense, if $v \in LSC(\Omega)$ (respectively, $v \in USC(\Omega)$)
  and for any $x_0\in \Omega$ and any $\varphi \in C^2$ satisfying $\varphi (x_0) = v (x_0)$ and $\varphi \leq v$ near $x_0$ (respectively, $\varphi(x_0) = v(x_0)$ and $\varphi \geq v$ near $x_0$),  then
		\begin{gather*}
		    \lambda(A[\varphi]) \in \Gamma~\text{and}~f(\lambda(A[\varphi]))\geq \psi(x,\varphi, \nabla \varphi)~\text{at}~x_0.\\(\text{respectively,} \,\,\text{either}~\lambda(A[\varphi])\notin \Gamma~\text{or}~\lambda(A[\varphi])\in \Gamma~\text{and}~f(\lambda(A[\varphi]))\leq \psi(x,\varphi, \nabla \varphi)~\text{at}~x_0).
		\end{gather*} 
\end{definition}

The following lemma can be proved similarly as Lemma \ref{remove}.
\begin{lemma}\label{nremove}
    For $n\geq 2$, let $(f,\Gamma)$ satisfy \eqref{nincreasing} and \eqref{eqn-230331-0110}, $\psi\in C^0( B\times\bR\times\bR^n)$ be a positive function. Suppose that $v\in LSC(B)$ is lowerconical at $0$ and satisfies 
    $f(\lambda(\av))\geq \psi (x,v,\nabla v)~\text{in}~B \setminus \{0\}~\text{in the viscosity sense}$. Then,
    $f(\lambda(\av))\geq \psi (x,v,\nabla v)~\text{in}~B~\text{in the viscosity sense}.$
\end{lemma}

In view of Lemma \ref{remove} and \ref{nremove}, we derive the criteria of $v$ being lowerconical and upperconical at a point in the following subsections.

\subsection{Criteria for lowerconical behavior}\label{lowerconicalcrit} 
\begin{proposition} \label{removablesingularity}
        For $n\geq 2$, let $\Gamma$ satisfy \eqref{eqn-230331-0110} with $\bl\notin\overline{\Gamma}$, and $v$ satisfy  
	\begin{equation*}
	    \lambda(A[v])\in \overline{\Gamma},~~\text{in}~B\setminus\{0\}~~\text{in the viscosity sense}.
	\end{equation*}
 Then $\liminf\limits_{x\rightarrow 0} v \in ( -\infty, \infty]$. Moreover, if we define $v(0) = \liminf\limits_{x\rightarrow 0} v$, then $v \in LSC (B)$ and is lowerconical at $0$.
\end{proposition}

In terms of $\lambda(\nabla^2 v)$ instead of $\lambda(\av)$, a related result is \cite[Theorem~1.4]{CLN3}, which also includes discussions on higher-dimensional singular sets. Our results also hold for higher dimensional singular sets; see our forthcoming papers, where we discuss general operator $\nabla^2 v+\alpha\nabla v\otimes\nabla v+\beta |\nabla^2 v|^2 I$, not just $\av$.
 
\begin{remark}
Proposition \ref{removablesingularity} is optimal in the sense that the conclusion fails for any cone $\Gamma$ with $\bl\in\overline{\Gamma}$, i.e. $\mgn \leq 1$, as shown below.
\end{remark}

\begin{example}\label{reex1}
    Let $v = \alpha \log r$, $\alpha \in (0,\infty)$. By \eqref{230930-1027},
    \begin{equation*}
        \lambda(A[v]) = C(r) (1,-1,\ldots,-1)\quad \text{and}\quad
        C(r) = 2^{-1} e^{-2v}  v' (v+ 2 \log r )' >0.
    \end{equation*}
    Clearly, for any cone $\Gamma$ with $\bl\notin\overline{\Gamma}$, i.e. $\mgn \leq 1$, $\lambda(A[v]) \in \overline{\Gamma}$ in $\bR^n\setminus \{0\}$. However, $\liminf_{r\rightarrow 0} v = -\infty$. In particular, $v$ is not lowerconical at $0$.
\end{example}

\begin{example}\label{reex2}
    Let $v = \frac{2}{\mu -1} \log (1- r^{1-\mu})$, $\mu \in [0,1)$. By \eqref{230930-1027},
    \begin{equation*}
        \lambda(A[v]) = C(r) (\mu,-1,\ldots,-1)
        \quad \text{and}\quad
        C(r) = 2^{-1} e^{-2v}  v' (v+ 2\log r)'  >0,
    \end{equation*}
    hence for any cone $\Gamma$ with $\mgn \leq \mu (<1)$, $\lambda(A[v]) \in \overline{\Gamma}$ in $B_1\setminus \{0\}$. However, $v \in C^0(B_1)$  fails to be lowerconical at $0$.
\end{example}

The following lemmas will be used in the proof of Proposition \ref{removablesingularity}.

\begin{lemma}\label{vanishgrad}
Let $n\geq 2$ and $v$ be a $C^1$ function defined in a neighborhood of $0\in \bR^n$ with $p\coloneqq \nabla v(0) \neq 0$.
Then there exists a M\"obius transformation $\psi$ such that 
\begin{equation*}
    \psi(0) = 0 \quad \text{and} \quad \nabla v^\psi (0) = 0.
\end{equation*}
Moreover, all such $\psi$ can be written as
\begin{equation} \label{eqn-230120-0802}
    \psi(x) = O \left( \frac{\lambda^2 (x-\bar{x}) }{|x - \bar{x}|^2} + \frac{\lambda^2 \bar{x}}{|\bar{x}|^2} \right),
\end{equation}
where $\lambda \in \bR\setminus\{0\}$, $O \in O(n)$, and $\bar{x} = \frac{\lambda^2}{2} O^T p$.
\end{lemma}

\begin{proof}
First, from \cite[Theorem~3.5.1]{MR1393195}, all M\"obius transforms mapping zero to zero can be written as either \eqref{eqn-230120-0802} or $\psi(x) = \lambda^2 O x$. Clearly, the second possibility can be ruled out as $\nabla v^\psi(0) = \lambda^2 O^T p \neq 0$. Now, for $\psi$ in the form of \eqref{eqn-230120-0802}, direct computation shows
\begin{equation*}
    \nabla v^\psi (0) 
    = 
    2 \frac{\bar{x}}{|\bar{x}|^2} 
    + 
    \frac{\lambda^2}{|\bar{x}|^2} \left( I_n - 2 \frac{\bar{x}}{|\bar{x}|} \otimes \frac{\bar{x}}{|\bar{x}|} \right) O^T p.
\end{equation*}
Matching $\nabla v^\psi (0) = 0$, we obtain the desired formula.
\end{proof}

\begin{lemma}\label{infdecrease}
 Let $v$ be a radial function satisfying the assumptions of Proposition \ref{removablesingularity}. Then, $v(r)$ is non-increasing in $r$. In particular, $\lim\limits_{r\rightarrow 0} v(r) >-\infty$.
\end{lemma}
\begin{proof}
Clearly, $\bl\notin \overline\Gamma$ i.e. $\mgn>1$ implies that 
\begin{equation}\label{nlowerconicalray}
   \{c(1+\epsi,-1,\dots,-1)\mid c\geq 0\}\cap \Gamma = \emptyset, \quad \text{for small} \  \epsi>0.
\end{equation}
\sloppy
Suppose the contrary that there exist $0<r_1<r_2$ such that $ v(r_1)< v(r_2)$. For a fixed $\delta \in (0, v(r_2) - v(r_1))$ and $\epsi$ given in \eqref{nlowerconicalray}, let
\begin{equation*}
    \underline{v} (r) := 2 \epsi^{-1} \log (C_7 r^{-\epsi} + C_8),
\end{equation*}
where $C_7 = (r_1^{-\epsi} - r_2^{-\epsi})^{-1} (e^{\epsi(v(r_1) + \delta)/2} - e^{\epsi v(r_2) /2}) < 0$ and $C_8 = (r_2^{-\epsi} - r_1^{-\epsi})^{-1} ( e^{\epsi(v(r_1) + \delta) /2} r_2^{-\epsi} - e^{\epsi v(r_2) / 2} r_1^{-\epsi} ) >0$ are chosen such that $\underline{v}(r_1) = v (r_1) + \delta$ and $\underline{v}(r_2) = v (r_2)$.
Since $(C_7 r^{-\epsi} + C_8)|_{r=r_1}>0$ and $(C_7 r^{-\epsi} + C_8)|_{r=0+} = -\infty$, there exists some $a \in (0, r_1)$, such that $\underline{v} \in C^\infty (a,r_2]$ and $\lim_{r\rightarrow a+}\underline{v}(r) = -\infty$.
By \eqref{230930-1027} with $v=\ubar{v}$,
\begin{equation*}
    \lambda(A[\underline{v}]) = C(r) (1+\epsi, -1,\ldots,-1),
    \,\, \text{where} \,\,
    C(r) = 2^{-1} e^{-2\underline{v}} \underline{v}' (\underline{v} + 2\log r)' > 0.
\end{equation*}
By \eqref{nlowerconicalray}, $\lambda(A[\underline{v}]) \in \bR^n\setminus \overline{\Gamma}$.
Since $\underline{v}(r_1) > v(r_1)$ and $\underline{v}(r_2) = v (r_2)$, an application of the comparison principle Theorem \ref{thm-220721-0849} gives $\underline{v} > v$ in $(a,r_1)$. Indeed, if $\underline{v}(r_3) \leq v(r_3)$ for some $r_3 \in (a,r_1)$, then, according to the comparison principle, we have $\underline{v} \leq v$ in $(r_3, r_2)$, which contradicts to $\underline{v}(r_1) > v(r_1)$.
It follows that $\lim_{r\rightarrow a+} v(r) = -\infty$, a contradiction.
\end{proof}
\begin{proof}[Proof of Proposition \ref{removablesingularity}]
Since for every $O\in O(n)$, $\lambda(A[v(O\cdot)])\in\overline\Gamma$ and
$\ubar{v}(r) : = \inf_{\p B_r} v(=\inf_{O\in O(n)}v(Ox))$ is lower semi-continuous in $B\setminus\{0\}$, we know $\lambda(A[\ubar{v}]) \in \overline\Gamma$ in $B\setminus \{0\}$.
By Lemma \ref{infdecrease}, $\ubar v(r)$ is non-increasing in $r$ and, after setting 
$v(0) := \liminf_{x\rightarrow 0} v(x) = \lim_{r\rightarrow 0} \ubar{v} (r)>- \infty$, $v$ is in $LSC(B)$. 

Next, we prove that $v$ is lowerconical at $0$ by contradiction. Suppose not, then $v(0)<\infty$ and there exist some $\epsi >0$, $\eta \in C^\infty (B)$, and $r_0\in(0,1)$, such that
	\begin{equation}\label{nlowerconical}
		( v - \eta)(x) - ( v -  \eta ) (0) \geq \epsi |x|,~~ \forall |x|\leq r_0.
	\end{equation}
There are two possibilities.

\noindent\textbf{Case 1:}  $\nabla\eta(0)=0$.
Inequality
\eqref{nlowerconical} implies that $\ubar{v}(r)  \geq \frac{\epsi}{2}r +v(0)>v(0)$ for small $r>0$, violating the monotonicity of $\ubar{v}$. 

\noindent\textbf{Case 2:}  $\nabla\eta(0)\neq 0$.
By Lemma \ref{vanishgrad}, there exists a M\"obius transformation $\psi$, such that $\psi (0) = 0$ and $\nabla \eta^\psi (0) = 0$.
By \eqref{nlowerconical} and the invertibility of $\nabla \psi(0)$,
\begin{equation*}
   v^\psi (x) - \eta^\psi (x) - v^\psi(0) + \eta^\psi (0) \geq \epsi |\psi(x)| \geq \epsi' |x|, \quad \forall |x|\leq r_0'
\end{equation*}
for some $\epsi' \in (0, \epsi)$ and $r_0' \in (0, r_0)$.
Since $v^\psi$ still satisfies $\lambda(A[v^\psi])\in \overline{\Gamma}$ near $0$ and $\nabla \eta^\psi(0) = 0$, by Case 1, we reach a contradiction.
\end{proof}

\subsection{Criteria for upperconical behavior}\label{upperconicalcrit}
\begin{proposition}\label{230805-2330}
           For $n\geq 2$, let $\Gamma$ satisfy \eqref{eqn-230331-0110} with $-\bl\in\Gamma$, and $v$ satisfy
	\begin{equation} \label{230806-1123}
	    \lambda(A[v])\in \bR^n\setminus\Gamma,~~\text{in}~B\setminus\{0\}~~\text{in the viscosity sense}.
	\end{equation}
 Additionally, we assume 
 \begin{equation}\label{230806-1010}
     \lim\limits_{|x|\rightarrow 0}\big( v(x)+2\log|x|\big)=-\infty.
 \end{equation}
 Then $\limsup\limits_{x\rightarrow 0} v \in [-\infty, \infty)$. Moreover, if we define $v(0) = \limsup\limits_{x\rightarrow 0} v$, then $v \in USC (B)$ and is upperconical at $0$.\end{proposition}

\begin{remark}
The assumption $-\bl\in\Gamma$ in Proposition \ref{230805-2330} is optimal in the sense that the conclusion fails for any cone $\Gamma$ with $-\bl\notin\Gamma$, i.e. $\mg \leq 1$, as shown below.
\end{remark}

\begin{example}\label{reex3}
    Let $v = \alpha \log r$, $\alpha \in (-2,0)$. By \eqref{230930-1027},
    \begin{equation*}
        \lambda(A[v]) = C(r) (-1,1,\ldots,1)\quad \text{and}\quad
        C(r) = -2^{-1} e^{-2v}  v' (v+ 2 \log r )' >0.
    \end{equation*}
    Clearly, $v$ satisfies \eqref{230806-1010}, and for any cone $\Gamma$ with $\mg \leq 1$, $\lambda(A[v]) \in \bR^n\setminus\Gamma$ in $\bR^n\setminus \{0\}$. However, $\limsup_{r\rightarrow 0} v = \infty$.
\end{example}

\begin{example}\label{reex4}
    Let $v = \frac{2}{\mu -1} \log (1+ r^{1-\mu})$, $\mu \in [0,1)$. By \eqref{230930-1027},
    \begin{equation*}
        \lambda(A[v]) = C(r) (-\mu,1,\ldots,1)
        \quad \text{and}\quad
        C(r) = -2^{-1} e^{-2v}  v' (v+ 2\log r)'  >0.
    \end{equation*}
        Clearly, $v$ satisfies \eqref{230806-1010}, and for any cone $\Gamma$ with $\mg \leq \mu (<1)$, $\lambda(A[v]) \in \bR^n\setminus\Gamma$ in $\bR^n\setminus \{0\}$. However, $v\in C^0(B)$ fails to be upperconical at $x=0$.
\end{example}

\begin{remark}
    The assumption \eqref{230806-1010} in Proposition \ref{230805-2330} is also optimal in the sense that for any cone $\Gamma$ with $-\bl\in\Gamma$ i.e. $\mg>1$ and any $a\in\bR$, there exists some $v$ satisfying $\lambda(\av)\in\bR^n\setminus\Gamma$ in $B\setminus\{0\}$ and $\lim_{|x|\to 0}(v(x)+2\log|x|)=a$. In particular, $\limsup_{x\to 0}v(x)=\infty$. See the example below.
\end{remark}
\begin{example}
    Let $v = \frac{2}{\mu -1} \log (1+ e^{\frac{\mu-1}{2}a}r^{1-\mu})$, $\mu \in (1,\infty)$. By \eqref{230930-1027},
    \begin{equation*}
        \lambda(A[v]) = C(r) (-\mu,1,\ldots,1)
        \quad \text{and}\quad
        C(r) = -2^{-1} e^{-2v}  v' (v+ 2\log r)'  >0.
    \end{equation*}
        Clearly, for any cone $\Gamma$ with $\mg \geq \mu (>1)$, $\lambda(A[v]) \in \bR^n\setminus\Gamma$ in $\bR^n\setminus \{0\}$. However, $\lim_{r\to 0}(v(r)+2\log r)=a$ and in particular, $\lim_{r\to 0}v=\infty$.
\end{example}

 To prove Proposition \ref{230805-2330}, we need following 
 \begin{lemma}\label{230806-1116}
      Let $v$ be a radial function satisfying the assumptions of Proposition \ref{230805-2330}. Then, $v(r)$ is non-decreasing in $r$. In particular, $\lim\limits_{r\rightarrow 0} v(r) <\infty$.
\end{lemma}

\begin{proof}
    Clearly, $-\bl\in\Gamma$ i.e $\mg>1$ implies that 
    \begin{equation}\label{230805-2322}
    \{c(-1-\epsi,1,\ldots,1)\mid c\geq 0\}\subset \overline{\Gamma}, \quad\text{for small}~ 
    \epsi>0.
\end{equation}
Suppose the contrary that there exist $0<r_1<r_2$ such that $v(r_1)>v(r_2)$. Let 
    \begin{equation*}
       \bar{v}(r)=2\epsi^{-1}\log(C_3 r^{-\epsi}+C_4),
    \end{equation*}
    where $C_3=2^{-1}\min\{r_2^{\epsi}e^{2^{-1}\epsi v(r_2)}, (r_1^{-\epsi}-r_2^{-\epsi})^{-1}(e^{2^{-1}\epsi v(r_1)}-e^{2^{-1}\epsi v(r_2)})\}>0$ and $C_4=e^{2^{-1}\epsi v(r_2)}-r_2^{-\epsi}C_3>0$ are chosen such that $\bar{v}(r_1)<v(r_1)$ and $\bar{v}(r_2)=v(r_2)$. By \eqref{230930-1027} with $v=\bar{v}$,
    \begin{equation*}
        \lambda(A[\bar{v}])=C(r)(-1-\epsi,1,\ldots,1),~\text{where}~C(r)=-2^{-1}e^{-2\bar{v}}\bar{v}'(\bar{v}+2\log r)'>0.
    \end{equation*}
    By \eqref{230805-2322}, $\lambda(A[\bar{v}])\in\overline{\Gamma}$ in $B\setminus\{0\}$. Since $\bar{v}(r_1)<v(r_1)$ and $\bar{v}(r_2)=v(r_2)$, an application of the comparison principle Theorem \ref{thm-220721-0849} yields $v(r)\geq \bar{v}(r),~0<r<r_1$. Hence,
    \begin{equation*}
        \limsup\limits_{r\rightarrow 0}(v(r)+2\log r)\geq \limsup\limits_{r\rightarrow 0}(\bar{v}(r)+2\log r)=2\epsi^{-1}\log C_3>-\infty.    
    \end{equation*}
    This is a contradiction with condition \eqref{230806-1010}.
\end{proof}

\begin{proof}[Proof of Proposition \ref{230805-2330}]
 The proof is similar to that of  Proposition \ref{removablesingularity}, though there is some difference
 (e.g. the need of assumption (\ref{230806-1010})). For reader's convenience, we include the details.
 Let $\bar{v}(r) : = \sup_{\p B_r} v$. Then $\bar{v}$ satisfies $\lambda(A[\bar{v}]) \in \bR^n\setminus\Gamma$ in $B\setminus \{0\}$.
By Lemma \ref{230806-1116}, $\bar v(r)$ is non-decreasing in $r$ and, after setting 
$v(0) := \limsup_{x\rightarrow 0} v(x) = \lim_{r\rightarrow 0} \bar{v} (r)< \infty$, $v$ is in $USC(B)$. Next we prove that $v$ is upperconical at $0$ by contradiction. Suppose not, then  $v(0) > -\infty$ and there exist some $\epsi >0$, $\eta \in C^\infty (B)$, and $r_0\in(0,1)$, 
	such that 
	\begin{equation}\label{230806-1112}
		( v - \eta)(x) - ( v -  \eta ) (0) \leq -\epsi |x|,~~ \forall |x|\leq r_0.
	\end{equation}

\noindent\textbf{Case 1:}  $\nabla\eta(0)=0$.
Inequality  \eqref{230806-1112} implies that 
$\bar{v}(r) \le v(0)-\frac{\epsi}{2}r<v(0)$ for small $r>0$, violating 
 the monotonicity of $\bar v$.

\noindent\textbf{Case 2:}  $\nabla\eta(0)\neq 0$.
As the Case 2 in proof of Proposition \ref{removablesingularity}, 
 there exists a M\"obius transformation $\psi$, such that $\psi (0) = 0$,  $\nabla \eta^\psi (0) = 0$, and 
\begin{equation*}
   v^\psi (x) - \eta^\psi (x) - v^\psi(0) + \eta^\psi (0) \leq -\epsi |\psi(x)| \leq -\epsi' |x|, ~ \forall |x|\leq r_0'
\end{equation*}
for some $\epsi' \in (0, \epsi)$ and $r_0' \in (0, r_0)$.
Since $v^\psi$ still satisfies \eqref{230806-1123}, \eqref{230806-1010} and $\nabla \eta^\psi(0) = 0$,
we reach a contradiction by Case 1.
\end{proof}
\subsection{Necessity of lowerconical and upperconical behavior}\label{sec-4.3}
For $n\geq 2$, let $\Gamma$ satisfy \eqref{eqn-230331-0110}. We will prove $\lambda(\av)\in\overline\Gamma$ in $B$ implies $v$ is lowerconical at $0$. Similarly, one can also prove $\lambda(\av)\in\bR^n\setminus\Gamma$ in $B$ implies $v$ is upperconical at $0$.

Clearly, $v\in LSC(B)$ and we may assume $v(0)<\infty$. Suppose the contrary that there exist some $\epsi>0$, $\eta\in C^\infty(B)$, and $r_0\in(0,1)$, satisfying \eqref{nlowerconical}. There are two possibilities.

\noindent \textbf{Case 1:} $\nabla \eta(0)=0$. By \eqref{nlowerconical},
\begin{equation}\label{230930-1426}
    v(x)>v(0),~\forall |x|>0~\text{small}.
\end{equation}
On the other hand, since any constant is a solution of \eqref{eqn-221204-0213}, we have, by applying Theorem \ref{thm-220721-0849} in $B_r\subset B$, that
\begin{equation*}
    v\geq \min\limits_{\p B_r} v~\text{in}~B_r.
\end{equation*}
In particular, $v(0)\geq \min\limits_{\p B_r}v$. A contradiction to \eqref{230930-1426}.\\
\noindent \textbf{Case 2:} $\nabla \eta(0)\neq 0$. This case can be reduced to Case 1 as the Case 2 in the proof of Proposition \ref{removablesingularity}.

\section{Liouville-type theorem for $f(\lambda(\av))=1$}\label{sec-5}

\subsection{Proof of Theorem \ref{thm-230525-0321}}\label{sec-5.1}
\begin{proof}[Proof of Theorem \ref{thm-230525-0321}]
    Our proof is by the method of moving spheres, a variant of the method of moving planes. We first show that
    \begin{equation}\label{movingconditioninfty}
        \liminf\limits_{|x|\rightarrow \infty} \big(v(x)+2\log|x|\big)>-\infty.
    \end{equation}
    Since $\lambda(\av)\in \Gamma$ in $\bR^n$, we know that, by the M\"obius invariance, $\lambda(A[v^{0,1}])\in \Gamma$ in $B\setminus\{0\}$.
Applying Proposition \ref{removablesingularity} to $v^{0,1}$ yields $\liminf\limits_{|x|\rightarrow 0} v^{0,1}(x)>-\infty$, i.e. \eqref{movingconditioninfty} holds.

    \textbf{Step 1:} Starting the moving spheres. We establish the following:
            For any $x\in \bR^n$, there exists $\lambda_{0}(x)>0$, such that
    \begin{equation*}
        v^{x,\lambda}(y)\coloneqq 2\log \left(\frac{\lambda}{|y-x|}\right) + v \left(x+\frac{\lambda^2 (y-x)}{|y-x|^2} \right) \leq v(y),
        ~
        \forall |y-x|\geq \lambda,~0<\lambda<\lambda_0(x).
    \end{equation*}

This follows from \eqref{movingconditioninfty} and $ v \in C^{0,1}_{loc}(\bR^n)$, as in the proof of \cite[Lemma 3.2]{Li2021ALT} and \cite[Lemma 2.1]{MR2001065}. 

For any $x\in \bR^n$, let
    \begin{equation*}
        \bar{\lambda}(x)\coloneqq \sup\{\mu>0 \mid \forall 0<\lambda<\mu,~v^{x,\lambda}(y)\leq v(y),~\forall |y-x|\geq \lambda\} \in (0,\infty].
    \end{equation*}
    
\textbf{Step 2:} 
In this step, we prove the following:
    If $\bar{\lambda}(x)< \infty$ for some $x\in \bR^n$, then $v^{x,\bar{\lambda}(x)}\equiv v$ in $\bR^n\setminus\{x\}$. 

    Without loss of generality, we take $x=0$, and let $\bar{\lambda}=\bar{\lambda}(0)$ and $v^{\bar{\lambda}}=v^{0,\bar{\lambda}}$. By the definition of $\bar{\lambda}$ and the continuity of $v$,
    \begin{equation*}
        v^{\bar{\lambda}}\leq v ~~\text{in}~\bR^n\setminus B_{\bar{\lambda}}.
    \end{equation*}  
A calculation gives, using \eqref{nequation} and the M\"obius invariance,
    \begin{equation} \label{eqn-230528-1202}
    f(\lambda(A[v^{\bar{\lambda}}]))=\big(\frac{\bar{\lambda}}{|y|}\big)^{2p} e^{-p v^{\bar{\lambda}}}
        \leq 
        e^{-p v^{\bar{\lambda}}}
        ~\text{in}\,\, \bR^n \setminus \overline{B_{\bar{\lambda}}}.
    \end{equation}
   Note also from \eqref{nequation}
   \begin{equation*}
       f(\lambda(\av))= e^{-pv}~~\text{in}~\bR^n \setminus \overline{B_{\bar{\lambda}}}.  \end{equation*}
Hence, by the strong maximum principle, we may assume 
\begin{equation} \label{eqn-230528-1207}
    v > v^{\bar{\lambda}} \,\,\text{in}\,\,\bR^n \setminus \overline{B_{\bar{\lambda}}}
\end{equation}
since otherwise $v^{\bar{\lambda}} \equiv v$ on $\bR^n \setminus \overline{B_{\bar{\lambda}}}$, which gives the desired conclusion of Step 2.
Moreover, by the Hopf Lemma, 
\begin{equation} \label{eqn-230528-1212}
    \frac{\p}{\p \nu} (v - v^{\bar{\lambda}}) > 0
    \quad
    \text{on}\,\,
    \p B_{\bar{\lambda}}.
\end{equation}
Next we prove
\begin{equation} \label{eqn-230528-1209}
    \liminf_{|y|\rightarrow \infty} (v-v^{\bar{\lambda}})(y) > 0.
\end{equation}
Similar to \eqref{eqn-230528-1202}, we have
    \begin{equation*}
        f(\lambda(A[v^{\bar{\lambda}}]))\geq e^{-p v^{\bar{\lambda}}}~\text{in}~B_{\bar{\lambda}} \setminus \{0\}.
    \end{equation*}
    In particular, $\lambda(A[v^{\bar{\lambda}}])\in \Gamma$ in $B_{\bar{\lambda}}\setminus\{0\}$. Therefore,
    by Proposition \ref{removablesingularity}, we obtain
    \begin{equation*}
        v^{\bar{\lambda}}~~\text{is lowerconical at }x=0~~\text{with}~~v^{\bar{\lambda}}(0)\coloneqq \liminf\limits_{|x|\rightarrow 0} v^{\bar{\lambda}}>-\infty,~~v^{\bar{\lambda}}\in LSC(B_{\bar{\lambda}}).
    \end{equation*}
By Lemma \ref{nremove}, we have 
    \begin{equation*}
        f(\lambda(A[v^{\bar{\lambda}}]))\geq e^{-p v^{\bar{\lambda}}}~\text{in}~B_{\bar{\lambda}}~~\text{in the viscosity sense}.
    \end{equation*}
Since
   \begin{equation*}
       f(\lambda(\av))= e^{-pv}~\text{in}~B_{\bar{\lambda}}.  \end{equation*}
       and by \eqref{eqn-230528-1207}, $v^{\bar{\lambda}} > v~ \text{in} ~
           B_{\bar{\lambda}} \setminus \{0\}$, we apply the strong comparison principle 
           (\cite[Theorem 3.1]{CLN3} and \cite[Theorem 2.3]{MR4181003})
           to obtain
        $\liminf\limits_{|x|\rightarrow 0}(v^{\bar{\lambda}}-v)(x)>0$,
    i.e. \eqref{eqn-230528-1209}.
    It follows from \eqref{eqn-230528-1207}, \eqref{eqn-230528-1212}, and \eqref{eqn-230528-1209} (see the proof of \cite[Lemma 2.2]{MR2001065} and 
    \cite[Lemma 3.2]{Li2021ALT}) that for some $\epsi>0$, $v^{0,\lambda}(y)\leq v(y)$, for every $|y|\geq \lambda$ and $\bar{\lambda} \leq \lambda < \bar{\lambda} + \epsi$, violating the definition of $\bar{\lambda}$.

\textbf{Step 3:} We prove that either $\bar{\lambda}(x)<\infty$ for all $x\in \bR^n$ or $\bar{\lambda}(x)=\infty$ for all $x\in\bR^n$.

Suppose for some $x\in \bR^n$, $\bar{\lambda}(x)=\infty$. By definition, $v^{x,\lambda}(y) \leq v(y)$ on $\bR^n\setminus B_{\lambda}$ for all $\lambda>0$. Adding both sides by $2\log |y|$ and sending $|y|\rightarrow \infty$, we have $v(x)+2\log \lambda \leq \liminf\limits_{|y|\rightarrow \infty} (v(y)+2 \log|y|) $ for all $\lambda>0$. Hence, $\lim\limits_{|y|\rightarrow \infty} (v(y)+2\log|y|) = \infty$.

On the other hand, if $\bar{\lambda}(x) < \infty$ for some $x$, we have, by Step 2, $v^{x,\bar{\lambda}(x)}\equiv v$ in $\bR^n\setminus\{x\}$. Therefore, $\lim\limits_{|y|\rightarrow \infty} (v(y)+2\log|y|) = v(x)+2\log \bar{\lambda}(x) < \infty$. Step 3 is proved.

\textbf{Step 4:} We prove $\bar{\lambda} (x) < \infty$ for all $x\in \bR^n$. 

If not, we have, by Step 3, $\bar{\lambda} (x) = \infty$ for all $x\in \bR^n$. Consequently, $v^{x,\lambda} (y) \leq v(y)$ for all $|y-x| \geq \lambda>0$. This implies that $v$ is a constant. See e.g. \cite[Lemma 11.1]{MR2001065}
and \cite[Lemma A.1]{Li2021ALT}. Therefore $\lambda(\av) = 0 \notin \Gamma$, which is impossible.

Now, by Step 2, 3, and 4, $\bar{\lambda}(x) < \infty$ for all $x\in \bR^n$, and $v^{x,\bar{\lambda}(x)} \equiv v$. 
Therefore, 
\begin{equation*}
             v(x)=  \log \big(\frac{a}{1+b^2|x-\bar{x}|^2}\big),
\end{equation*}
for some $a,b>0$ and $\bar{x}\in\bR^n$.
See e.g. \cite[Lemma 5.8]{Li_integral} and \cite[Lemma A.2]{Li2021ALT}.
It follows that $\lambda(\av)\equiv 2b^2 a^{-2} \vec{e}$ in $\bR^n$. Plugging this back to the equation, we must have $p=0$, $f(2 b^2 a^{-2}\vec{e})=1$, and $ 2 b^2 a^{-2}\vec{e}\in \Gamma$. The theorem is proved.
\end{proof}

\subsection{Counterexamples in Remark \ref{rmk-230615-1129}}\label{sec-230622-2309}

In this subsection, we work with functions of one variable. The following formula is useful: for any $v=v(x_1)$,
\begin{equation} \label{eqn-230606-0409}
    \lambda(A[v]) = \frac{1}{2}(v')^2 e^{-2v} \left( - 2 (v')^{-2} v'' + 1, - 1, \ldots, - 1\right).
\end{equation}
\begin{example} \label{exp-230622-1021}
    Let $n\geq 2$, for any $s \in (0,1]$, there exists a cone $\Gamma$ satisfying \eqref{eqn-230331-0110} with $\mgn = s$, a symmetric homogeneous of degree $1$ function $f \in C^0(\overline{\Gamma}) \cap C^\infty(\Gamma)$ satisfying $f|_{\p \Gamma} = 0$ and $C^{-1} \leq \p f/\p\lambda_i \leq C$ for all $i$ in $\Gamma$, and a function $v \in C^\infty (\bR^n)$, such that $f(\lambda(\av)) = 1$, $\lambda(\av) \in \Gamma$ in $\bR^n$, but $v$ is not of the form \eqref{nbubble}.
\end{example}

We construct $v=v(x_1)$. Recall \eqref{eqn-230606-0409}, $\lambda(A[v])$ takes the form $(\lambda_1,\lambda_2,\ldots,\lambda_2)$. We solve for $v$ such that $\lambda_1 + \gamma \lambda_2 = 1$ for $\gamma\in \bR$
, which is equivalent to
\begin{equation} \label{eqn-230622-1018}
    v'' + \frac{\gamma - 1}{2} (v')^2 + e^{2v} = 0.
\end{equation}
\begin{lemma}\label{lem-230622-1020}
    For  $\gamma, v_0$, $w_0\in \bR$, there exists a unique smooth solution $v$ of \eqref{eqn-230622-1018} with $v(0) = v_0, v'(0)=w_0$ in $(-\infty,\infty)$ if and only if  $-1\leq \gamma\leq 1$
    or $-1-2w_0^{-2} e^{2 v_0}<\gamma< -1$.     
\end{lemma}
In the above, if $w_0=0$, the inequality is understood as
$-\infty<\gamma<-1$.
The proof of Lemma \ref{lem-230622-1020} is postponed till the end of \S \ref{sec-230622-2309}. Now assuming Lemma \ref{lem-230622-1020}, we finish the construction of Example \ref{exp-230622-1021}.
We start with
\begin{equation*}
    f^{(0)}(\lambda) := \max_k \{\lambda_k + s (n-1)^{-1}\sum_{j\neq k} \lambda_j\}, \,\,\lambda \in \bR^n
\end{equation*}
and $\Gamma^{(0)} := (f^{(0)})^{-1}(0,\infty)$. It is easy to see that $s (n-1)^{-1} \leq \p f^{(0)}/\p\lambda_i \leq 1$ a.e. $\bR^n$ for all $i$, $\Gamma^{(0)}$ satisfies \eqref{eqn-230331-0110}, and $(s,-1,\ldots,-1) \in \p\Gamma^{(0)}$, from which $\mu_{\Gamma^{(0)}}^- = s$. 
Take $v$ as an entire solution of \eqref{eqn-230622-1018} with $\gamma = s$. As explained before, for any $x_1 \in \bR$,
$\lambda(A[v]) (x_1)$ takes the form $(1 + s \theta(x_1), -\theta(x_1), \ldots, -\theta(x_1))$ with $\theta(x_1) \geq 0$. Hence, $f^{(0)}(\lambda(A[v] )) \equiv 1$ in $\bR^n$.

Note that $\Gamma^{(0)}$, $f^{(0)}$ and $v$ satisfy all the properties in Example \ref{exp-230622-1021} except that $f^{(0)}$ is in $C^{0,1}$, not yet $C^\infty$.  These examples  show that condition $\bl\notin\overline\Gamma$ cannot be removed in Theorem \ref{thm-230525-0321}.
To achieve smooth $f$, we take a mollification $f^{(\epsi)} : = f^{(0)} \ast \eta_\epsi$, where $\epsi$ is constant chosen to be less than $(2\sqrt{n})^{-1}$ and $\eta_\epsi \in C^\infty_c(B_\epsi)$ is a non-negative radial function with $\int \eta_\epsi = 1$.

We first check that we still have $f^{(\epsi)} (\lambda(A[v])) \equiv 1$ for $v$ defined above. Due to the aforementioned form of $\lambda(A[v])$, we have $\lambda(A[v]) \in O_1$, where $O_1 = \{\mu \in \bR^n, \mu_1 > \mu_j,\forall j > 1\}$. Actually, we have $B_\epsi (\lambda(A[v])) \subset O_1$ since for any $x_1 \in \bR$ and $\mu \in \p O_1$,
\begin{equation*}
    | \lambda(A[v])(x_1) - \mu | \geq \sqrt{ |1 + s\theta(x_1) - \mu_1|^2 + |-\theta(x_1) - \mu_1|^2} \geq 2^{-1/2} > 0.
\end{equation*}
Hence, we can deduce that $f^{(\epsi)} (\lambda(A[v])) = f^{(0)} (\lambda(A[v])) \equiv 1$ as $f^{(0)}(\lambda) = \lambda_1 + s (n-1)^{-1} \sum_{j=2}^n \lambda_j$ in $O_1$, which is linear, and $\eta_\epsi$ is radial.

To make $f$ homogeneous of degree $1$ and $f|_{\p\Gamma} = 0$ as promised in Example \ref{exp-230622-1021}, some modification on $f^{(\epsilon)} $ and $\Gamma^{(0)}$  is needed.

First, by choosing smaller $\epsi$, we are able to define $f$ to be the homogeneous of degree $1$ extension of $f^{(\epsi)}$ from $\Sigma := (f^{(\epsi)})^{-1} (1)$ to $\Gamma = \{c\lambda\mid  \lambda\in\Sigma, c>0\}$ and $f=0$ on $\p\Gamma$. Such extension can be done since $\nabla f^{(\epsi)} \in \Gamma_n$ and $\lambda \cdot \nabla f^{(\epsi)} > 0$ on $\Sigma$ (c.f. Lemma \ref{lem-230615-1015}), which we check as follows:
\begin{align*}
    \lambda \cdot \nabla f^{(\epsi)} 
    = 
    \int \eta_\epsi (\mu) (\lambda - \mu) \cdot \nabla f^{(0)} (\lambda - \mu) \,d\mu + \int \eta_\epsi (\mu) \mu \cdot \nabla  f^{(0)} (\lambda - \mu) \,d\mu
    =:
    I + II.
\end{align*}
By the $1$-homogeneity of $f^{(0)}$, $I = \eta_\epsi \ast f^{(0)}(\lambda) = f^{(\epsi)} (\lambda) = 1$ on $\Sigma$.
For $II$, we simply use $|\nabla f^{(0)}| \leq \sqrt{n}$ to obtain $|II| \leq \epsi \sqrt{n} < 1/2$. Hence $\lambda \cdot \nabla f^{(\epsi)} \geq 1/2$ on $\Sigma$. 

Since $f (s,-1,\ldots,-1) = 0$,  we have  $\mgn = s$ for  $\Gamma := \{f>0\}$. 
It is not difficult to see that $f$, $\Gamma$, and $v$ verify all the desired properties in 
Example \ref{exp-230622-1021}. 
\medskip

\noindent{\it Counterexamples in Remark \ref{rmk-230615-1129}.}\ 
Let $\Gamma$ be a  cone satisfying \eqref{eqn-230331-0110} with $\lambda^*\in \overline \Gamma$.  
Take $s=1$ and $f=f^{(\epsilon)}$ for small $\epsilon>0$ in the above construction for Example \ref{exp-230622-1021}. Then
$f$ satisfies \eqref{nincreasing}.   Let $v$ be an entire solution of (\ref{eqn-230622-1018}) with $\gamma=1$. 
Clearly, $f(\lambda(A[v]))=1$ in $\bR^n$, and $v$ is not of the form \eqref{nbubble}.

\begin{proof}[Proof of Lemma \ref{lem-230622-1020}]
Let $(T_-, T_+)$ be the maximal existence interval
of solution $v$ of  \eqref{eqn-230622-1018} with $v(0) = v_0$ and $v'(0)=w_0$, and let
    $\varphi = e^v$, $w = v'$, and $\delta = -\frac{\gamma-1}{2}$. Then $\varphi>0$ and $w$  satisfy
    \begin{equation} \label{eqn-231027-0452}
        \begin{cases}
            \varphi' = \varphi w,\\
            w' = -\varphi^2 + \delta w^2,\\
            \varphi(0) = \varphi_0,\quad w(0) = w_0,
        \end{cases}
    \end{equation}
    where $\varphi_0 = e^{v_0} > 0$. A calculation shows that \eqref{eqn-231027-0452} has a first integral
    \begin{equation*}
    I_\delta(\varphi,w)
    :=
    \begin{cases}
        \frac{1}{1-\delta}\varphi^{2-2\delta} + \varphi^{-2\delta}w^2 \,\,&\text{when $\delta \neq 1$},\\
        2\log \varphi + \varphi^{-2}w^2 \,\,&\text{when
        $\delta = 1$}, 
    \end{cases}
    \end{equation*}
    i.e.
    \begin{equation}\label{I}
    I_\delta(\varphi, w)\equiv I_\delta(\varphi_0, w_0)\ \ \ \text{on} \ (T_-, T_+). 
    \end{equation} 
   Thus we have, on $(T_-, T_+)$, that
   \begin{equation} \label{F}
          0\le w^2= F(\varphi)\equiv F_{\delta, \varphi_0, w_0}(\varphi)
         :=
         \begin{cases}
             I_\delta(\varphi_0,w_0) \varphi^{2\delta} - \frac{1}{1-\delta}\varphi^2
         \,\,&\text{when}\,\,\delta \neq 1,\\
         I_\delta(\varphi_0,w_0) \varphi^{2} - 2\varphi^2\log \varphi
         \,\,&\text{when}\,\,\delta = 1.
         \end{cases}
   \end{equation}
    We only need to consider $x_1 > 0$ since for $x_1 < 0$ we work with $v(-x_1)$ and $-w(-x_1)$ instead of $v(x_1)$ and $w(x_1)$.
    
    \noindent \textbf{Case 1}: $w(0)=w_0 \leq 0$.
    
    In this case, we prove that $T_+<\infty$ when $\delta <0$, and $T_+=\infty$ when $\delta\ge 0$.
    
    Since $w'|_{w=0} = -\varphi^2 <  0$, we have $w < 0$ on   $(0, T_+)$. 
    Hence we have, using $\varphi'=\varphi w$ and (\ref{F}),  that 
    \begin{equation}\label{equ-varphi}
        \varphi'= -\varphi \sqrt{F(\varphi)}<0\ ~\text{on}~(0,T_+).
    \end{equation} 
Thus $0<\varphi<\varphi_0$ in $(0,T_+)$.
Let $a\coloneqq \lim_{x_1\to (T_+)-}\varphi(x_1)\in [0,\varphi_0)$.
Integrating \eqref{equ-varphi} gives
    \begin{equation}\label{T}
        T_+= \int^{\varphi_0}_{a} \frac{ds}{s \sqrt{F(s)}} .
    \end{equation}
When $\delta<1$, we have $I_\delta(\varphi_0,w_0)>0$ and
$s\sqrt{F(s)} = (1+o(1))\sqrt{I_\delta(\varphi_0,w_0)}s^{1+\delta}$ for $s>0$ small. Clearly, when $\delta<0$, we have $T_+<\infty$. We next treat the case $\delta\geq 0$. We divide the discussion into two subcases: $a=0$ and $a>0$. When $a=0$, we also need the following asymptotic behaviors for small $s>0$: $s\sqrt{F(s)}=(1+o(1))\sqrt{\frac{1}{\delta-1}}s^2$ when $\delta>1$,  and $s\sqrt{F(s)}=(1+o(1))\sqrt{2|\log s|} s^2$ when $\delta=1$.    We see from (\ref{T}), using the asymptotics of $s\sqrt{F(s)}$, that $T_+=\infty$ when $\delta\ge 0$ and $a=0$. When $a>0$, by \eqref{T}, we have $T_+<\infty$.
On the other hand, since $\delta\ge 0$, we have, using \eqref{eqn-231027-0452}, that  $w'\geq -\varphi^2 \ge -\varphi_0^2$.  It follows that  $w\geq w_0-\varphi_0^2 T_+$ on $(0, T_+)$, so $|v'|=|w|$ is bounded in the interval.   We also know that $v=\log \varphi\in
[\log a, \log\varphi_0]$ on the interval, violating the maximality of $T_+<\infty$.
We have proved that when $\delta\ge 0$, we have $a=0$ and $T_+<\infty$.

    \noindent\textbf{Case 2}: $w(0) =w_0> 0$. 
    
    In this case, we prove that $T_+<\infty$ when $\delta<0$,
    $T_+=\infty$ when $0\le \delta\le 1$, 
    $T_+<\infty$ when $\delta>1$  and $I_\delta(\varphi_0, w_0)\ge 0$,
    and $T_+=\infty$, when $\delta>1$ and $I_\delta(\varphi_0, w_0)<0$.
    
    When $\delta>1$ and $I_\delta(\varphi_0,w_0) \geq 0$, we see from (\ref{F}) that 
    $w^2=F(\varphi)\ge \frac 1{\delta-1} \varphi^2>0$ in    $[0, T_+)$.  
    It follows, using $w(0)>0$,
    that $\varphi' = \varphi w = \varphi \sqrt{F(\varphi)}
    \ge (\delta-1)^{-1/2} \varphi^2$ in $[0, T^+)$, and therefore $T_+<\infty$.
    
 Next, we discuss the remaining cases: $\delta \leq 1$ or $\delta > 1$ and $I_\delta(\varphi_0,w_0) < 0$. 
 We only need to show that $w<0$ somewhere in $(0, T_+)$, since our desired results
 follow from Case 1 after a translation in $x_1$. 
 Suppose the contrary, $w\ge 0$ in $(0, T_+)$. Then
 $\varphi'=\varphi w\ge 0$ and 
 $\bar \varphi:= \lim_{ t\to (T_+)^-}  \varphi(t)\in [\varphi_0, \infty]$.  
By (\ref{I}),
$$
I_\delta(\varphi_0, w_0)\ge 
\begin{cases}
     \frac 1{1-\delta} \varphi^{ 2-2\delta} &\text{when}~\delta\neq 1\\
     2\log \varphi &\text{when}~\delta=1
    \end{cases}
 \ \ \text{in}\ 
(0, T_+).
$$
It follows that $\bar \varphi<\infty$.
Back to (\ref{I}), we have $\bar w := \lim_{ t\to (T_+)^-}  w(t) \in [0, \infty)$.
These imply  that $\sup _{ (0, T_+) } (|v|+|v'|)<\infty$, and thus $T_+=\infty$.
Back to equation (\ref{eqn-231027-0452}), we have 
$0=\lim_{ x_1\to \infty} \varphi'(x_1)= \bar \varphi \bar w$  and
$0= \lim_{ x_1\to \infty} w'(x_1)=-\bar \varphi^2 +\delta\bar w^2$, violating $\bar \varphi\ge \varphi_0>0$.
\end{proof}

\section{Liouville-type theorems for $f(\lambda(\av))=0$} \label{sec-230528-0351}

In this section, we study Liouville-type theorems for solutions or subsolutions of $\lambda(\av) \in \p\Gamma$.

\subsection{Proof of Theorem \ref{thm-230329-1327}}\label{sec-6.1} We divide our proof into two cases: $\bl\notin\overline\Gamma$ and $\bl\in\Gamma$.
Denote 
$v^{x,\lambda} := v^{\varphi_{x,\lambda}}$, where $\varphi_{x,\lambda} (y) := y + \lambda^2 (y-x)/|y-x|^2$.
\begin{proof}[Proof of Theorem \ref{thm-230329-1327} when $\Gamma$ satisfies $\bl\notin\overline\Gamma$]
     For any ball $B_\lambda(x)\subset\bR^n$, consider $v$ and its Kelvin transformation $v^{x,\lambda}$ in the ball $B_\lambda(x)$. By the M\"obius invariance, we know that 
	\begin{equation*}
	   \lambda(A[v])\in \p \Gamma \  \text{in} ~~B_\lambda(x), 
	\ 	\ \lambda(A [v^{x,\lambda}])\in \p \Gamma\ \text{in}  ~~B_\lambda(x)\setminus\{x\},
	\  \	v= v^{x,\lambda}\ \text{in} ~~\p B_{\lambda}(x).
	\end{equation*}
 Using Proposition \ref{removablesingularity}, $v^{x,\lambda}$ is lowerconical at $x$. By Lemma \ref{remove}, $\lambda(A[v^{x,\lambda}]) \in \overline{\Gamma}$ in $B_\lambda(x)$. Hence, by the comparison principle Theorem \ref{thm-220721-0849}, we obtain $v\leq v^{x,\lambda}$ in $B_\lambda(x)\setminus\{x\}$. This implies that $v$ must be a constant. See e.g. \cite[Lemma 11.1]{MR2001065}
and \cite[Lemma A.1]{Li2021ALT}. 
 \end{proof}
In the other case: $\bl\in\Gamma$, we actually have following Liouville-type theorem for subsolutions:

\begin{theorem} \label{thm-230528-0404}
        For $n\geq 2$, let $\Gamma$ satisfy \eqref{eqn-230331-0110} with $\bl\in\Gamma$.        
Then any $v$ satisfying $\lambda(A[v]) \in \bR^n\setminus \Gamma$ on $\bR^n$, in the viscosity sense, must be constant.
\end{theorem}

The following lemma is used in the proof of Theorem \ref{thm-230528-0404}:
\begin{lemma}\label{lem-220901-1050}
    Let $n$, $\Gamma$ satisfy the same assumptions as Theorem \ref{thm-230528-0404}, and $R>0$ be a constant.
\begin{enumerate}[label=(\alph*)]
    \item If $v\in USC(\bR^n\setminus B_R)$ is a radially symmetric function satisfying $\lambda(A[v]) \in \bR^n\setminus \Gamma$ in $\bR^n\setminus B_R$,	
	then $v(r)$ is non-increasing in $(R,\infty)$. 
    \item If $v\in USC(B_R\setminus\{0\})$ is a radially symmetric function satisfying $\lambda(A[v]) \in \bR^n\setminus \Gamma$ in $B_R\setminus \{0\}$,
	then $v(r) + 2 \log r$ is non-decreasing in $(0,R)$.
\end{enumerate}
\end{lemma}

\begin{remark} \label{rmk-230111-0943}
    Statement (a) and (b) in Lemma \ref{lem-220901-1050} are equivalent by taking a Kelvin transformation $v \mapsto v(R^2/r) - 2 \log (r/R)$.
\end{remark}

\begin{proof}[Proof of Lemma \ref{lem-220901-1050}]
	As explained in Remark \ref{rmk-230111-0943}, we only need to prove $(a)$. Clearly, $\bl\in\Gamma$ implies
 \begin{equation}\label{230917-1347}
    \exists \epsi>0,~\{c(1-\epsi,-1,\ldots,-1)\mid c\geq 0\}\subset \overline{\Gamma}.
\end{equation}

\sloppy 
Suppose the contrary that there exist $R<r_1<r_2<\infty$ such that $ v(r_1)< v(r_2)$. Fix a $\delta \in (0, v(r_2) - v(r_1))$ and let $\epsi$ be given by \eqref{230917-1347}. Consider
\begin{equation*}
    \overline{v} (r) := - 2\epsi^{-1} \log (C_7 r^{\epsi} + C_8),
\end{equation*}
where $C_7 = ({r_1^{\epsi} - r_2^{\epsi}})^{-1} [{ e^{-\frac{\epsi}{2}v(r_1)} - e^{-\frac{\epsi}{2} (v(r_2) - \delta) } }] < 0$ and 
$C_8 =({r_2^{\epsi} - r_1^{\epsi}})^{-1}[ { e^{-\frac{\epsi}{2}v(r_1)}r_2^{\epsi} - e^{ -\frac{\epsi}{2} ( v(r_2) - \delta) }r_1^{\epsi} } ] >0$ are chosen so that $\overline{v}(r_1)=v(r_1)~~\text{and}~~\overline{v}(r_2)=v(r_2)-\delta$.
Since $(C_7 r^{\epsi} + C_8)|_{r=r_2} > 0$ and $(C_7 r^{\epsi} + C_8)|_{r=+\infty} =-\infty$, there exists some $a\in (r_2 , \infty)$, such that $\overline{v}\in C^\infty [r_1,a)$ and  $\lim\limits_{r\rightarrow a-}\overline{v}(r) = +\infty$. By \eqref{230930-1027} with $v=\bar{v}$,
\begin{equation*}
    \lambda(A[\overline{v}]) = C(r) (1-\epsi,-1,\ldots,-1),\,\,\text{where}\,\,C(r) = 2^{-1} e^{-2 \overline{v}} \left( \overline{v}' (\overline{v} + 2\log r )' \right) >0.
\end{equation*}
By \eqref{230917-1347}, $\lambda(A[\overline{v}]) \in \overline{\Gamma}$ in $B_a \setminus \{0\}$. Since $\overline{v}(r_1) = v(r_1)$ and $\overline{v}(r_2) < v(r_2)$, according to the comparison principle Theorem \ref{thm-220721-0849}, we have $v > \overline{v}$ in $(r_2, a)$. Indeed, if $v(r_3) \leq \overline{v}(r_3)$ for some $r_3 \in (r_2, a)$, then, by the compairson principle, we have $v \leq \overline{v}$ in $(r_1,r_3)$, which contradicts to $\overline{v}(r_2) < v(r_2)$. It follows that $\lim\limits_{r\rightarrow a-} v (r)  = +\infty$, which is a contradiction.
\end{proof}

\begin{proof}[Proof of Theorem \ref{thm-230528-0404}]
By the comparison principle Theorem \ref{thm-220721-0849}, $v \leq \max_{\p B_r(x)} v$ on $B_r$, for $r>0$. This implies that $v^{x}(r)\coloneqq \max_{\partial B_r(x)} v$ is non-decreasing in $(0,\infty)$. On the other hand, as in the proof of Proposition \ref{removablesingularity}, $v^x(\cdot)$ is upper semi-continuous and $\lambda(A[v^x])\in \bR^n\setminus \Gamma$ in $\bR^n\setminus\{x\}$. By Lemma \ref{lem-220901-1050} (a), $v^x$ is also non-increasing in $(0,\infty)$. Therefore, 
for some constant $C_x$, $v^x (r) = C_x$ for all $r>0$. It follows that $v \leq C_x$ in $\bR^n\setminus\{x\}$, and, 
using the upper semi-continuity of $v$, $C_x = \limsup_{r \rightarrow 0} v^x(r) \leq v(x).$ Hence, $v(x) = \max_{\bR^n} v$ in $\bR^n$ for any $x$. The theorem is proved.
\end{proof}

\subsection{A Harnack inequality and Proofs of Theorem \ref{230918-1857} and \ref{230805-2045}}\label{sec-6.2}
\begin{proposition} \label{prop-230614-1100}
For $n\ge 2$, let $B_2\subset \bR^n$ and $\Gamma$ satisfy \eqref{eqn-230331-0110}. Suppose that $v$ satisfies $\lambda(A[v])\in \p \Gamma$ and $-\infty<a \leq v \leq b < \infty$ in $B_2$. Then, $v\in C^{0,1}_{loc}(B_2)$ and
\begin{equation}\label{lipregularityest}
	    |\nabla v|\leq C ~~ \text{in}~~ B_{1/2}.
	\end{equation}
where $C$ depends on $n$ and an upper bound of $b-a$. Moreover, if $\Gamma$ satisfies in addition $\bl\notin\p\Gamma$, then $C$ only depends on $n$.
\end{proposition}

\begin{remark}    
    The condition $\bl\notin\p\Gamma$ in the above proposition is optimal in the sense that for any cone $\Gamma$ with $\bl\in\p\Gamma$, $v_j = j x_1$ satisfies $\lambda(A[v_j]) \in \p\Gamma$ in $\bR^n$ and $|\nabla v_j (0)| \rightarrow \infty$ as $j\rightarrow \infty$.
\end{remark}

The proof is an adaptation of \cite[Proof of Theorem~1.10]{Li2009}. A key step is by proving a local H\"older estimate via blowup. The following quantity was introduced in \cite{Li2009}: For $w \in C^{0,1}_{loc}(\Omega)$ defined on an open subset $\Omega$, $\gamma \in (0,1)$, and $x\in \Omega$, define
\begin{equation*}
    \delta(w,x;\Omega, \gamma)\coloneqq 
\begin{cases}
	\infty &\text{if}~~\dist(x,\p\Omega)^\gamma [w]_{\gamma, dist(x,\partial\Omega)}(x) <1,\\
	\mu &\text{where}~~0<\mu \leq dist(x,\partial\Omega)~\text{and} ~\mu^\gamma [w]_{\gamma,\mu}(x)=1,
        \\&
        \text{if} ~~ 
	\dist(x,\p\Omega)^\gamma  [w]_{\gamma, dist(x,\partial\Omega)} (x) \geq 1,
 \end{cases}
\end{equation*}
where $[w]_{\gamma,s}(x) \coloneqq \sup\limits_{0<|y-x|< s} |x-y|^{-\gamma}|w(y)-w(x)|$.

\begin{proof}
    The Lipschitz regularity and estimate \eqref{lipregularityest} can be deduced from the comparison principle Theorem \ref{thm-220721-0849}, see \cite[Proof of Theorem 1.1]{MR3813247}.

   Now, we further assume that $\Gamma$ satisfies $\bl\notin\p\Gamma$ and prove \eqref{lipregularityest} with $C$ independent of $b-a$.
   Let $\gamma \in (0,1)$ be a fixed number. We first prove $[v]_{C^\gamma(B_{1/2})} \leq C$ for a constant $C$ depending only on $n$ and $\gamma$.
   Suppose the contrary that there exist a sequence of functions $v_i \in C^{0,1}_{loc}(B_2)$ satisfying $\lambda(A[v_i])\in\p\Gamma$ in $B_2$, but $[v_i]_{C^\gamma(B_{1/2})} \rightarrow \infty$. This implies
    \begin{equation}\label{contra-d}
    	\inf\limits_{x\in B_{1/2}} \delta \left(v_i, x\right) \longrightarrow 0,
    \end{equation}
    where for convenience, we denote $\delta\left(v_i, x\right)\coloneqq \delta\left(v_i, x ; B_2, \gamma \right)$.
    From $\eqref{contra-d}$, there exists $ |x_i|<1$, such that 
  \begin{equation*}
    \frac{1-|x_i|}{\delta\left(v_i, x_i\right)}= \sup\limits_{|x|\leq 1} \frac{1-|x|}{\delta\left(v_i, x\right)}\longrightarrow \infty.
  \end{equation*}
  Let
  \begin{equation*}
  	\sigma_i \coloneqq \frac{1-|x_i|}{2},~~\epsilon_i\coloneqq \delta\left(v_i, x_i\right).
  \end{equation*}
Then we have 
\begin{equation}\label{scale-d}
	\frac{\sigma_i}{\epsilon_i}\rightarrow \infty,~~ \epsilon_i \rightarrow 0, 
\end{equation}
and 
\begin{equation}\label{scale2-d}
	\epsilon_i \leq 2 \delta\left(v_i, x\right)~, \forall\, x\in B_{\sigma_i}(x_i).
\end{equation}
Let
\begin{equation}\label{231007-1140}
\widetilde{v}_i(y)\coloneqq v_i(x_i +\epsilon_i y) - v_i (x_i),~~|y|<\frac{\sigma_i}{\epsilon_i}.
\end{equation}
Then by definition we immediately see that 
$\widetilde{v}_i(0)=0$ and $[\widetilde{v}_i]_{\gamma, 1}(0)= \epsilon_i^\gamma [v_i]_{\gamma, \epsilon_i}(x_i) =1$, and by the triangle inequality, 
\begin{equation*}
[\widetilde{v}_i]_{\gamma, 1}(x) \leq 2^{-\gamma} \epsilon_i^\gamma \left(\sup\limits_{|z-(x_i +\epsilon_i x)|<\epsilon_i} [v_i]_{\gamma, \epsilon_i /2}(z) + [v_i]_{\gamma, \epsilon_i /2}(x_i +\epsilon_i x) \right),\quad \forall |x| < \frac{\sigma_i}{2\epsi_i}.
\end{equation*}
Combining with $\eqref{scale-d}$, $\eqref{scale2-d}$, 
we obtain that for large $i$,
\begin{equation*}
	\begin{array}{llll}
	[\widetilde{v}_i]_{\gamma, 1}(x) &\leq& \sup\limits_{|z-(x_i +\epsilon_i x)|<\epsilon_i} \delta(v_i, z)^\gamma [v_i]_{\gamma, \delta(v_i, z)}(z) \\
	&+& \delta(v_i, x_i +\epsilon_i x)^\gamma [v_i]_{\gamma, \delta(v_i, x_i +\epsilon_i x)}(x_i +\epsilon_i x) =2,
 \quad \forall |x| < \frac{\sigma_i}{2\epsi_i}.
	\end{array}
\end{equation*}
Therefore, since $\widetilde{v}_i (0)=0$, we obtain
\begin{equation*}
    |\widetilde{v}_i|\leq C(K)~~\text{in}~ B_K (0),~~\forall\, K > 1.
\end{equation*}
It is easy to see that $\lambda(A[\widetilde{v}_i])\in\p\Gamma$ in $B_{\sigma_i/\epsi_i}$.
Thus, by estimate \eqref{lipregularityest} we just proved, 
\begin{equation*}
	|\nabla \widetilde{v}_i | \leq C(K) ~~\text{in}~B_K (0), ~~\forall\, K>1.
\end{equation*}
Hence, there exists some $\widetilde{v}\in C^{0,1}_{loc}(\mathbb{R}^n)$ such that passing to a subsequence,
\begin{equation}\label{limit-d}
	\widetilde{v}_i \longrightarrow \widetilde{v} ~~\text{in} ~C^{\widetilde{\gamma}}_{loc}(\mathbb{R}^n),~~\forall\, \widetilde{\gamma} \in (0,1).
\end{equation}
Since $[\widetilde{v}_i]_{\gamma, 1}(0)=1$, for every $i$, by \eqref{limit-d} with $\widetilde{\gamma}>\gamma$, we also have $[\widetilde{v}]_{\gamma, 1}(0)=1$.
In particular, $\widetilde{v}$ cannot be a constant.

On the other hand, by the limit property of viscosity solutions, \eqref{limit-d} implies $\lambda(A[\widetilde{v}])\in\p\Gamma$ on $\bR^n$. By Theorem \ref{thm-230329-1327}, $\widetilde{v}\equiv \text{constant}$. We reach a contradiction, and hence, have proved $[v]_{C^\gamma(B_{1/2})} \leq C$.

Now we prove the gradient estimate. Let $\widehat{v}(x) := v(x) - v(0)$. From $[\widehat{v}]_{C^\gamma(B_{1/2})} = [v]_{C^\gamma(B_{1/2})} \leq C$ and $\widehat{v}(0) = 0$, we obtain a two-sided bound $\sup_{B_{1/2}}|\widehat{v}(x)| \leq C$. Since $\lambda(A[\hat{v}])\in \p\Gamma$ in $B_1$, we can apply estimate \eqref{lipregularityest} to obtain the desired gradient bound. 
\end{proof}

\begin{corollary}
    \label{realharnack}
For $n\geq 2$, let $B_R\subset\bR^n$ be a ball centered at the origin and $\Gamma$ satisfy \eqref{eqn-230331-0110} with $\bl\notin\p\Gamma$.
    Suppose that $v$ is a viscosity solution of \eqref{eqn-221204-0213} in $B_R\setminus\{0\}$. Then, for any $r\in (0,R/4)$, 
\begin{equation*}
	  \max\limits_{ B_{2r}\setminus B_r} v \leq   	 \min\limits_{ B_{2r}\setminus B_r} v +C,
\end{equation*}
where $C$ depends only on $n$. Moreover, $\limsup\limits_{x\rightarrow 0}v(x)\leq \liminf\limits_{x\rightarrow 0}v(x)+C$.
\end{corollary}

\begin{proof}
    The inequality follows from Proposition \ref{prop-230614-1100} and a standard rescaling argument. Namely, for every $r\in (0,R/4)$, let $v_r(y)\coloneqq v(r y)$, $1/4<|y|<4$. Then $v_r$ satisfies $\lambda(A[v_r])\in \p\Gamma$ in $\{1/4<|y|<4\}$. Applying Proposition \ref{prop-230614-1100}, $\max_{B_2\setminus B_1} v_r \leq   	 \min_{B_2\setminus B_1} v_r +C$, where $C$ depends only on $n$. Equivalently, $\max_{ B_{2r}\setminus B_r} v \leq \min_{ B_{2r}\setminus B_r} v +C$. 
    
    Next, we prove the ``Moreover" part. For any $0<r_1<r_2<R/4$, since $v$ is a supersolution of \eqref{eqn-221204-0213} in $B_{r_2}\setminus \overline{B_{r_1}}$, $\inf_{\p B_{r_1}\cup \p B_{r_2}}v$ is a supersolution of \eqref{eqn-221204-0213} in $B_{r_2}\setminus\overline{B_{r_1}}$, and $v\geq \inf_{\p B_{r_1}\cup \p B_{r_2}}v$ on $\p B_{r_1}\cup \p B_{r_2}$, by the comparison principle Theorem \ref{thm-220721-0849}, $\inf_{\p B_r}v\geq \min\{\inf_{\p B_{r_1}}v,\inf_{\p B_{r_2}}v\}$ for any $r\in (r_1,r_2)$. It follows that there exists some $r_0>0$ (may be very small) such that $\inf_{\p B_r}v$ is monotone in $r\in (0,r_0)$. In particular, $\lim_{r\rightarrow 0}\inf_{\p B_r}v$ exists and is equal to $\liminf_{x\rightarrow 0} v(x)$. By a similar argument, $\lim_{r\rightarrow 0}\sup_{\p B_r}v$ exists and is equal to $\limsup_{x\rightarrow 0} v(x)$. Hence, combining with the Harnack inequality we just proved, the desired conclusion is obtained.     
    \end{proof}

\begin{proof}[Proof of Theorem \ref{230805-2045}] This theorem is a direct corollary of Lemma \ref{remove}, Proposition \ref{removablesingularity}, Proposition \ref{230805-2330} and Proposition \ref{prop-230614-1100}.    
\end{proof}

\begin{proof}[Proof of Theorem \ref{230918-1857}] 
Recall the notation: $v^{x,\lambda} := v^{\varphi_{x,\lambda}}$ where $\varphi_{x,\lambda} (y) := y + \lambda^2 (y-x)/|y-x|^2$.
We divide the proof of the theorem into two cases.

\noindent
\textbf{Case 1:} $\bl\notin\overline\Gamma$.

For any $0<\lambda<|x|$, consider $v$ and its Kelvin transformation $v^{x,\lambda}$ in the ball $B_\lambda(x)$. By the M\"obius invariance, we know that 
	\begin{equation}\label{230918-2230}
	   \lambda(A[v])\in \p \Gamma \  \text{in} ~~B_\lambda(x), 
	\ 	\ \lambda(A [v^{x,\lambda}])\in \p \Gamma\ \text{in}  ~~B_\lambda(x)\setminus\{x,P\},
	\  \	v= v^{x,\lambda}\ \text{in} ~~\p B_{\lambda}(x),
	\end{equation}
	where $P = \varphi_{x, \lambda}(0)$. 
 Using Proposition \ref{removablesingularity}, $v^{x,\lambda}$ is lowerconical at $x$ and $P$. By Lemma \ref{remove}, $\lambda(A[v^{x,\lambda}]) \in \overline{\Gamma}$ in $B_\lambda(x)$. Hence, by the comparison principle Theorem \ref{thm-220721-0849}, we obtain $v\leq v^{x,\lambda}$ in $B_\lambda(x)\setminus\{x,P\}$. By the continuity of $v$, $v\leq v^{x,|x|}$ in $B_{|x|}(x) \setminus \{x\}$ for any $x\in \bR^n\setminus \{0\}$. This implies the radial symmetry of $v$ and the
 monotonicity. See \cite[Proof of Theorem 1.2]{LiJFA06}  and \cite[Proof of Theorem 1.9]{MR4458997}.

\noindent \textbf{Case 2:} $\bl\in\Gamma$.
 
 The proof is divided into three subcases.
 
\textbf{Case 2.1:}  $\liminf_{|x|\rightarrow 0} v(x)=\liminf_{|x|\rightarrow\infty} (v(x)+2\log|x|)=\infty$.

For any $0<\lambda<|x|$, we know that $v$ and $v^{x,\lambda}$ satisfy \eqref{230918-2230}, and 
$\liminf_{y\rightarrow x}v^{x,\lambda}(y)=\liminf_{y\rightarrow P} v^{x,\lambda}(y)=\infty$. Therefore, for  small $\delta>0$, we have
    \begin{equation*}
        v\leq v^{x,\lambda}~\text{on}~\p\Omega_\delta,~\text{where}~\Omega_\delta\coloneqq 
        B_\lambda(x)\setminus \overline{ B_\delta(x) \cup B_\delta(P)}.
    \end{equation*}
    By Theorem \ref{thm-220721-0849},
     $v\leq v^{x,\lambda}$ in $\Omega_\delta$. Sending $\delta\rightarrow 0$ gives  $v\leq v^{x,\lambda}$ in $B_\lambda(x)\setminus\{x,P\}$. This gives the  radial symmetry and the monotonicity of $v$ as in Case 1.

\textbf{Case 2.2:} $\liminf_{|x|\rightarrow 0} v(x)<\infty$. 

By Corollary \ref{realharnack},  $\limsup_{|x|\rightarrow 0}v(x)<\infty$. Property $\Gamma+\Gamma_n\subset\Gamma$ and $\bl\in\Gamma$ imply $-\bl\in\Gamma$. By Proposition \ref{230805-2330}, $v$ is upperconical at $x=0$. Hence, by Lemma \ref{remove}, we know that  $\lambda(\av)\in \bR^n\setminus\Gamma$ in $\bR^n$. Then it follows from Theorem \ref{thm-230528-0404} that $v\equiv \text{constant}$.

\textbf{Case 2.3:} $\liminf_{|x|\rightarrow\infty} (v(x)+2\log|x|)<\infty$.

Let $\hat{v}(x)\coloneqq v^{0,1}(x)= v(x/|x|^2)-2\log |x|$.  Then $\liminf_{|x|\rightarrow 0}\hat{v}(x)<\infty$. 
By the M\"obius invariance, $\lambda(A[\hat{v}])\in\p\Gamma$ in $\bR^n\setminus\{0\}$. 
By Case 2.2, $\hat{v}\equiv C$ which gives the desired properties of $v$.
\end{proof}

\section{Local gradient estimates}\label{sec-7}
\subsection{Local gradient estimates assuming a two-sided bound on solutions}\label{augmentedhessianest}

Let $(B_1,g)$ be a Riemannian geodesic ball of dimension $n\geq 2$, $v$ is a real-valued function defined on $B_1$. Let 
\begin{equation} \label{eqn-230602-1030}
    W^{\alpha,\beta}_{g,S}[v] \coloneqq \nabla^2 v +\alpha d v\otimes d v +\beta |\nabla v|^2 g +S
\end{equation}
be an augmented Hessian tensor, where $\alpha,\beta$ are two real numbers and $S=S_{ij}dx^i dx^j$ is a symmetric $2-$tensor. Here all differentiations and norms are taken with respect to $g$. In the following, for convenience we simply write $W$ when there is no ambiguity. 

We first derive an interior gradient estimate for 
\begin{equation}\label{ahequ}
    f( e^{\theta (x,v)} \lambda(W))=h(x,v),~\lambda(W)\in \Gamma, ~\text{in }B_1, 
\end{equation}
assuming a two-sided bound on $v$. 
\begin{theorem} \label{ahequlocalgradest}
Let $(B_1,g)$ be a $C^2$ Riemannian geodesic ball of dimension $n\geq 2$, and $(f,\Gamma)$ satisfy \eqref{eqn-230331-0110}, \eqref{basicf}, and \eqref{eqn-230530-1131}. Let $v\in C^3(B_1)$ solve \eqref{ahequ} in $B_1$. Assume that for some constants $a$ and $b$, $a \leq v \leq b$ on $B_1$.
     Assume also $h\in C^1 (B_1\times\bR)$ is positive, $\theta \in C^1(B_1 \times \bR)$, $S_{ij}\in C^1(B_1)$, and $\beta\neq 0$. Then
	\begin{equation*} 
		\left| \nabla_g v \right|_g \leq C \quad \text{in} ~B_{1/2},
	\end{equation*}
	where $C$ depends on $n$, $(f,\Gamma)$,
 upper bounds of $h, |h_x|, |\theta_x|$, $|\theta_v|$, $|\alpha|$, $|a|$, and $|b|$, lower bounds of 
 $\theta$ and $h_v$, positive upper and lower bounds of $|\beta|$, and bounds of $g$, $R_{ijkl}$ and $S$ together with its first order covariant derivatives with respect to $g$.
\end{theorem}  

\begin{remark}
When $\beta = 0$, local gradient estimates were established by Trudinger in \cite{MR1466315} for equation $\sigma_k(\lambda(\nabla^2 v))=1$, $\lambda(\nabla^2 v)\in \Gamma_k$, $k=1,\dots,n$.
\end{remark}

\begin{proof}[Proof of Theorem \ref{ahequlocalgradest}]
Consider 
\[ \Psi \coloneqq \rho \cdot e^{\phi (v)} \cdot|\nabla v|^2,\]
where $\rho$ is a cut-off function with 
\begin{equation*}
		\left\{
	\begin{array}{lr}
		 \rho = 1 ~~\text{in}~B_{1/2},\quad \rho=0 ~~\text{in}~ B_1\setminus B_{2/3}, \quad \rho >0 ~~\text{in}~ B_{2/3},\quad \rho \geq 0, ~~ \text{in}~B_1,\\[1ex]
		|\partial_k \rho| \leq C\sqrt{\rho} ~~ \text{in}~B_1,~~ \text{for some absolute constant $C$},~~\forall k,
	\end{array}
	\right.	
\end{equation*}
and $\phi(s)\coloneqq \epsilon e^{\Lambda \beta s}$ is an auxiliary function. Here, small constant $\epsi$ and large constant $\Lambda$ 
are chosen, depending only on $\alpha,\beta,a,b$, such that for some $c_1 = c_1 (\alpha,\beta,a,b)$,
\begin{equation}\label{phi}
      \beta \phi' \geq c_1 >0,\quad \phi'' +\alpha \phi' -\left(\phi'\right)^2 \geq 0 \quad \text{on}~~ [a,b].
\end{equation}
Due to the boundedness of $\phi$ on $[a, b]$ and the support of $\rho$, to prove the desired gradient estimate, it suffices to prove
\begin{equation*}
	\Psi(x_0) = \max\limits_{\overline{B_1}} \Psi \leq C .
\end{equation*}
Clearly, $x_0 \in B_{2/3}$.
Take some $g$-geodesic normal coordinates centered at $x_0$ such that $W$ is diagonal at $x_0$. More precisely, $g^{-1}W(x_0)$ is diagonal. Write $W_{ij}$ as the $(i,j)$ entry of $W$ under the chosen coordinates, we have, at $x_0$,
\begin{equation}\label{Wij}
    	\left\{
	\begin{array}{lr}
		 W_{ii}=v_{ii}+\alpha v_i^2 + \beta |\nabla v|^2 + S_{ii},~\forall i\\[1ex]
		v_{ij}=-\alpha v_iv_j-S_{ij},~\forall i\neq j,
	\end{array}
	\right.	
\end{equation}
where subindices always denote covariant differentiation with respect to $g$. We will keep using this notation throughout the proof.

Since $v\in C^3$ and $\Psi$ achieves its maximum at an interior point $x_0$, we have 
\begin{equation}\label{maxpt}
	\nabla \Psi(x_0)=0,\quad \nabla^2 \Psi(x_0)\leq 0.
\end{equation} 

Applying $\partial_k$ to equation \eqref{ahequ}, we deduce that
\begin{equation}\label{diffequ}
	f^i W_{ii,k} = (h_k + h_v v_k) e^{-\theta} - f^i W_{ii}\theta_v v_k - f^i W_{ii} \theta_k \quad \text{at}~~ x_0,
\end{equation}
where $h_k$ and $h_v$, similarly $\theta_k$ and $\theta_v$, represent partial derivatives on variables $x_k$ and $v$, respectively, and  
\begin{equation*}
W_{ii,k}= \partial_k W_{ii}, \quad f^{i}= \frac{\partial f}{\partial \lambda_i}( e^{\theta} \lambda(W)) .
\end{equation*}

Next, we write down relations in \eqref{maxpt}. By direct computations,
\[\Psi_i=2\rho e^\phi v_{ki}v_k + \left( \phi' v_i + \frac{\rho_i}{\rho}\right)\Psi.\]
Evaluating at $x_0$, by $\Psi_i(x_0)=0$, we obtain
\begin{equation}\label{grad0}
	2v_{ki}v_k =-\phi' |\nabla v|^2 v_i - \frac{\rho_i}{\rho} |\nabla v|^2, \quad \forall\, 1\leq i\leq n,\quad \text{at}~~x_0.
\end{equation}
Take the second order derivative on $\Psi$ and evaluate at $x_0$, \\
$$
\begin{array}{ll}
	0\geq \left(\Psi_{ij} \right) =& 2v_{kij}v_k e^\phi \rho + 2v_{ki}v_{kj} e^\phi \rho + 2v_{ki} v_k e^\phi \phi' v_j \rho \\
  & +2v_{ki}v_k e^\phi \rho_j +\left(\phi''v_i v_j +\phi' v_{ij} + \frac{\rho\rho_{ij}-\rho_i\rho_j}{\rho^2}\right) \rho e^\phi |\nabla v|^2.
\end{array}
$$
From \eqref{maxpt}, in particular, $0\geq \Psi_{ii}$ for each $i$. Multiplying both sides by $e^{-\phi}f^i>0$ and summing over $i$, we obtain, at $x_0$,
\begin{equation} \label{eqn-230529-1147}
\begin{split}
        0
    \geq
    e^{-\phi}f^i \Psi_{ii}
    &=
    2\rho f^i v_{kii}v_k + 2\rho f^i v^2_{ki} + 2\rho \phi' f^i v_{ki} v_k v_i + 2f^i v_{ki} v_k \rho_i 
    \\
	& \quad + f^i \left(\phi'' v_i^2 + \phi' v_{ii} +\frac{\rho\rho_{ii}-\rho_i^2}{\rho^2}\right)\rho|\nabla v|^2.
\end{split}
\end{equation}
Using $v_{kii}=v_{iik}+R_{kiim}v_m$, \eqref{Wij}, and \eqref{grad0}, we obtain
\begin{align*}
    &\text{RHS of \eqref{eqn-230529-1147}}
    \\&=
    2\rho f^i v_k \left( W_{ii} -\alpha v_i^2 -\beta|\nabla v|^2 \delta_{ii}-S_{ii}\right)_k + 2\rho f^i R_{kiim}v_k v_m+ 2\rho f^i v_{ki}^2
    \\& \quad - \rho \phi' f^i \left(|\nabla v|^2 \phi' v_i^2 + \frac{|\nabla v|^2 \rho_i v_i}{\rho}\right)- f^i \rho_i \left(|\nabla v|^2 \phi' v_i +\frac{|\nabla v|^2 \rho_i}{\rho}\right)
    \\& \quad +\rho \phi'' |\nabla v|^2 f^i v_i^2 + \rho \phi' |\nabla v|^2 f^i \left(W_{ii} -\alpha v_i^2 -\beta|\nabla v|^2 \delta_{ii}-S_{ii}\right) 
    \\& \quad + |\nabla v|^2 f^i \frac{\rho \rho_{ii}-\rho_i^2}{\rho}
\end{align*}
\begin{align*}
&=
    2\rho f^i \big( W_{ii,k}v_k +\alpha \phi' |\nabla v|^2 v_i^2 +\alpha \frac{v_i \rho_i}{\rho}|\nabla v|^2
      +\beta |\nabla v|^4 \phi' \delta_{ii} + \beta \frac{\rho_k v_k}{\rho} |\nabla v|^2 \delta_{ii} - S_{ii,k}v_k\big)
      \\& \quad 
      + 2\rho f^i R_{kiim}v_k v_m + 2\rho f^i v_{ki}^2 
      \\& \quad
      - \rho \phi' f^i \left(|\nabla v|^2 \phi' v_i^2 + \frac{|\nabla v|^2\rho_i v_i}{\rho}\right)-f^i \rho_i \left(|\nabla v|^2 \phi' v_i + \frac{|\nabla v|^2 \rho_i}{\rho}\right)
      \\& \quad 
      + \rho \phi'' |\nabla v|^2 f^i v_i^2 + \rho \phi' |\nabla v|^2 f^i \left( W_{ii} -\alpha v_i^2 -\beta|\nabla v|^2 \delta_{ii}-S_{ii}\right) + |\nabla v|^2 f^i \frac{\rho \rho_{ii}-\rho_i^2}{\rho}.
\end{align*}
Here, in the second equality, we applied \eqref{grad0} to replace $(\alpha v_i^2)_k$ and $(\beta |\nabla v|\delta_{ii})_k$. Moving $2\rho f^i v_{ki}^2$ to the left-hand side, using \eqref{diffequ} to replace $f^iW_{ii,k}$, then rearranging terms, we have, at $x_0$,
 \begin{align}
     0
     &\geq
     - 2\rho f^i v_{ki}^2 \nonumber
     \\&=
     2\rho h_k v_k e^{-\theta} - 2\rho f^i S_{ii,k} v_k - 2\rho\theta_k v_k f^i W_{ii} \nonumber
     \\ & \quad
     + 2\rho h_v |\nabla v|^2 e^{-\theta} + 2\rho f^i R_{kiim} v_k v_m + \rho \phi' |\nabla v|^2 f^i W_{ii} - 2\rho\theta_v|\nabla v|^2 f^i W_{ii} \nonumber
     \\&\quad -\rho \phi' f^i S_{ii}|\nabla v|^2 +|\nabla v|^2 f^i \frac{\rho\rho_{ii}-2\rho_i^2}{\rho} \nonumber
     \\ & \quad
     + 2\beta |\nabla v|^2 v_k \rho_k \sum\limits_i f^i +2(\alpha - \phi') |\nabla v|^2 f^i \rho_i v_i \nonumber
     \\ & \quad
     + \beta \phi' \rho |\nabla v|^4 \sum\limits_i f^i + \left(\phi'' +\alpha \phi' -(\phi')^2\right) \rho |\nabla v|^2 f^i v_i^2 \nonumber
\end{align}
\begin{align}
     &\geq
     (\sum_i f_i) \left( - C \rho |\nabla v| - C |\nabla v|^2 - C \sqrt{\rho} |\nabla v|^3 + c_1 \rho |\nabla v|^4 \right) \nonumber
     \\& \quad - C\rho ( |\nabla v| + |\nabla v|^2 )e^{-\theta} - C|f^i W_{ii}| \rho (|\nabla v| + |\nabla v|^2). \label{eqn-230530-1125}
 \end{align}
 
 Here, in the last inequality, we used \eqref{phi}. Note that we can always assume $( \rho|\nabla v |^2 ) (x_0) \geq \widetilde{C}$ for some large $\widetilde{C}$, since otherwise we are done. Now, if $\widetilde{C}$ is large enough, 
 \begin{equation*}
 \text{RHS of \eqref{eqn-230530-1125}}
 \,\, \geq
     \frac{c_1}{2}(\sum_i f_i)\rho |\nabla v|^4 - C e^{-\theta} ( 1 + |\sum_i f^i e^{\theta} W_{ii} |) \rho |\nabla v|^2.
 \end{equation*}
Using \eqref{eqn-230530-1131} with $C = \|h\|_{C^0 (B_1 \times [a,b])}$ and the fact that $W$ is diagonal at $x_0$, we reach $|\nabla v| (x_0) \leq C$ as desired. The theorem is proved.
\end{proof}

\subsection{Proofs of Theorem \ref{localgradest} and \ref{nlocalgradest}}\label{sec-7.2}
Recall the tensor $W^{\alpha,\beta}_{g,S}[v] = \nabla^2 v +\alpha d v\otimes d v +\beta |\nabla v|^2 g +S$ defined in \eqref{eqn-230602-1030}. We consider the equation
\begin{equation}\label{eqn-230602-1041}
    f( e^{\theta(v)} \lambda( W^{\alpha,\beta}_{g,S}[v] ))=h(x),~\lambda( W^{\alpha,\beta}_{g,S}[v] )\in \Gamma.
\end{equation}

\begin{theorem} \label{thm-230603-1124}
    Let $(B_2,g)$ be a $C^2$ Riemannian geodesic ball of dimension $n \geq 2$ and $(f,\Gamma)$ satisfy \eqref{eqn-230331-0110} and \eqref{basicf}--\eqref{eqn-230602-1040}. Suppose that $v \in C^3(B_2)$ satisfies \eqref{eqn-230602-1041} in $B_2$ for some $S\in C^1(B_2)$, $\alpha,\beta \in\bR$, and positive $h\in C^1(B_1)$. Then, if
    $\alpha + 2\beta = 0$, $\beta\neq 0$, $-\operatorname{sign}(\beta)\bl \notin \p\Gamma$,
    and for some constant $\Theta>0$,
    \begin{equation} \label{eqn-230603-0100}
        \theta(v(x)) \geq - \Theta,\quad |\theta_v(v(x))| \leq \Theta,\quad \text{on}\,\,B_2,
    \end{equation}
    we have
    \begin{equation*}
        |\nabla_g v|_g \leq C \quad \text{in} \,\, B_{1/2},
    \end{equation*}
    where $C$ depends only on $(f,\Gamma)$,
 upper bounds of $h, |h_x|$ in $B_2$, and $\Theta$, positive upper and lower bounds of $|\beta|$, and bounds of $g$, $R_{ijkl}$ and $S$ together with its first order covariant derivatives with respect to $g$.
\end{theorem}

\begin{proof}
The desired estimate under a two-sided bound assumption on $v$ has been given in Theorem \ref{ahequlocalgradest}. We only need to prove the estimate under a one-sided bound assumption on $v$,  and the proof is similar to that of Proposition \ref{prop-230614-1100}.

    For $\gamma \in (0,1)$, we first prove $[v]_{C^\gamma(B_{1/2})} \leq C$ by 
     contradiction. 
     Suppose not, there exists 
      a sequence of  $\{v_i\}$
      satisfying \eqref{eqn-230602-1041} in $B_2$, but $[v_i]_{C^\gamma(B_{1/2})} \rightarrow \infty$. For the clarity of the presentation, we only consider the sequence of solutions $v_i$ with fixed $\alpha,\beta,S,g,h$, and $\theta$. One can 
      modify the proof to allow all these to change with $i$ and obtain the desired dependence. 
      Define $\widetilde{v_i}$ by \eqref{231007-1140}, then follow the arguments below  \eqref{231007-1140} to obtain 
      that $[\widetilde{v_i}]_{\gamma,1}(0)=1$ and $|\widetilde{v_i}|\leq C(K)$ in $B_K(0)$ for any $K\geq 1$.

A calculation shows that $\widetilde{v}_i$ satisfies 
\begin{equation}\label{equv_i}
f( s_i \lambda(  W^{\alpha,\beta}_{g^{(i)}, \epsi_i^2 S} [\widetilde{v}_i] ) ) = h (x_i +\epsilon_i y), \,\, \lambda( W^{\alpha,\beta}_{g^{(i)}, \epsi_i^2 S} [\widetilde{v}_i]) \in \Gamma, ~~y \in B_{\sigma_i/ \epsilon_i}(0),
\end{equation}
where 
\begin{equation*}
    s_i := e^{\theta (v_i)} \epsi_i^{-2}, \quad g^{(i)} \coloneqq g^{(i)}(x_i+\epsilon_i \cdot)
\end{equation*}
and the eigenvalues are with respect to $g^{(i)}$.

Thus, by Theorem \ref{ahequlocalgradest}, 
\begin{equation*}
	|\nabla \widetilde{v}_i | \leq C(K) ~~\text{in}~B_K (0), ~~\forall\, K>1.
\end{equation*}
Here we have also used the fact that $\log(s_i) (v) := \theta (v + v_i(x_i)) - 2 \log\epsi_i$ still verifies \eqref{eqn-230603-0100}.
Hence, there exists some $\widetilde{v}\in C^{0,1}_{loc}(\mathbb{R}^n)$ such that, after passing to a subsequence,
\begin{equation}\label{limit}
	\widetilde{v}_i \longrightarrow \widetilde{v} ~~\text{in} ~C^{\widetilde{\gamma}}_{loc}(\mathbb{R}^n),~~\forall\, \widetilde{\gamma} \in (0,1).
\end{equation}
Since $[\widetilde{v}_i]_{\gamma, 1}(0)=1$,
we have, by \eqref{limit} with $\widetilde{\gamma}>\gamma$, that $[\widetilde{v}]_{\gamma, 1}(0)=1$.
In particular, $\widetilde{v}$ cannot be a constant.

On the other hand, in \eqref{equv_i}, we fix a normal coordinates of $g^{(i)}$ centered at $x_i$. By 
 \eqref{eqn-230602-1040}, 
 \eqref{limit}, $s_i \rightarrow \infty$, the limit property of viscosity solutions (cf. \cite[pp. 1316-1317]{Li2009}), and the cone property of $\p \Gamma$, we can deduce that
\begin{equation*}
	\lambda(\nabla^2 \widetilde{v} + \alpha \nabla \widetilde{v} \otimes \nabla \widetilde{v} + \beta |\nabla \widetilde{v}|^2 I_n ) \in \partial \Gamma ~\text{in}~ \mathbb{R}^n ~~\text{in the viscosity sense},
\end{equation*}
By further taking $\widetilde{v} \mapsto |2\beta|^{-1} \widetilde{v}$, using the cone property of $\p\Gamma$ and $\alpha + 2\beta = 0$, we know that $\widetilde{v}$ satisfies
\begin{equation*}
    \lambda(\nabla^2 \widetilde{v} - \operatorname{sign}(\beta) \nabla \widetilde{v} \otimes \nabla \widetilde{v} + \frac{1}{2}\operatorname{sign}(\beta) |\nabla \widetilde{v}|^2 I_n ) \in \partial \Gamma \quad \text{in}\,\,\bR^n.
\end{equation*}
By our assumptions on $\alpha,\beta$, and $\Gamma$, Theorem \ref{thm-230329-1327} and \ref{1327-prime}, $\widetilde{v}\equiv \text{constant}$. We reach a contradiction, and hence, have proved $[v]_{C^\gamma(B_1)} \leq C$.

Now we prove the gradient estimate. Let $\widehat{v}(x) := v(x) - v(0)$. From $[\widehat{v}]_{C^\gamma(B_1)} = [v]_{C^\gamma(B_1)} \leq C$ and $\widehat{v}(0) = 0$, we obtain a two-sided bound $\sup_{B_1}|\widehat{v}(x)| \leq C$. Since $\widehat{v}$ satisfies
\begin{equation*}
    f( e^{\theta(\widehat{v} + v(0))} \lambda( W^{\alpha,\beta}_{g,S}[\widehat{v}] )) = h(x),~\lambda( W^{\alpha,\beta}_{g,S}[\widehat{v}] )\in \Gamma, \quad \text{in} \,\, B_2,
\end{equation*}
we can apply Theorem \ref{ahequlocalgradest} to obtain the desired gradient bound. The theorem is proved.
\end{proof}

\begin{proof}[Proof  of Theorem \ref{localgradest}]
    This is a corollary of Theorem \ref{thm-230603-1124}. For $u \in C^3$ satisfying \eqref{pequ}, let $v:= - \frac{2}{n-2}\log u$. Then $v$ satisfies
    \begin{equation*}
        f( e^{2v} \lambda( W^{1,-1/2}_{g, A_g}[v] ) ) = h\quad \text{in}\,\, B_1,
    \end{equation*}
    i.e., \eqref{eqn-230602-1041} with $\alpha = 1, \beta = -1/2$, $S=A_g$, and $\theta(v) = 2v$. Now $u \leq C$ implies $v \geq - \frac{2}{n-2} C$, which implies \eqref{eqn-230603-0100} on $B_1$ instead of $B_2$. Hence, we can apply Theorem \ref{thm-230603-1124} to obtain the desired gradient estimate.
\end{proof}

The proof of Theorem \ref{nlocalgradest} uses the following lemma:
\begin{lemma}\label{upperbound}
        For $n\geq 3$ and constant $h_0>0$, let $(B_1,g)$ be a $C^2$ Riemannian geodesic ball centered at $0$, and $(f,\Gamma)$ satisfy \eqref{eqn-230331-0110} and \eqref{basicf}. Assume that $u\in C^2(B_1)$ satisfies $f(\lambda(-A_{g_u}))\geq h_0$ in $B_1$. Then
	\begin{equation*}
		u \leq C \quad \text{in } ~ B_{1/2},
	\end{equation*}
	where $C$ is a positive constant depending only on $n$, $h_0$, $(f,\Gamma)$ and a bound of $g$ and the Riemann curvature. 
\end{lemma}
\begin{proof}
    Without loss of generality, $g=g_{ij}(x)dx^i dx^j$, $g_{ij}(0)=\delta_{ij}$ and $\p_l g_{ij}(0)=0$.
     For any $\alpha>0$ and $r\in (0,1)$, let $u_{\alpha,r}(x)= (\alpha r)^{(n-2)/2}  \left({r^2-|x|^2}\right)^{-(n-2)/2}$, $|x|\coloneqq \sqrt{\sum_i x_i^2}<r$. 
     We will prove that there exist $\bar r>0$ and $\bar\alpha>1$ such that for any $\alpha\geq \bar\alpha$, $f(\lambda(-A_{g_{u_{\alpha,\bar r}}}))\leq h_0/2$, $|x|<\bar r.$
     To prove this, it is convenient to take $w=u_{\alpha,r}^{-2/(n-2)}=\alpha^{-1}(r-\frac{1}{r}|x|^2)$. Then $A_{w}\coloneqq A^{u_{\alpha,r}}=w\nabla^2 w-\frac{1}{2}|\nabla w|^2 I$, 
     $g_{u_{\alpha,r}}=u_{\alpha, r}^{4/(n-2)}g=w^{-2}g$,
     and $A_{g_{u_{\alpha,r}}}=w^{-1}\nabla^2_{g}w-\frac{1}{2}w^{-2}|\nabla_g w|_g^2 g+A_g$. 
     A computation gives
       $ g_{u_{\alpha, r}}^{-1}A_{g_{u_{\alpha, r}}}
        = A_{w}+\alpha^{-2} O(r^2)$, uniform for $|x|<r$, and $A_{w}\equiv -2\alpha^{-2}I$.
     Therefore,      $\lambda(A_{g_{u_{\alpha,r }}})=\alpha^{-2}\{-2(1,1,\dots,1)+O(r^2)\}$, uniform for $|x|<r$. By the property of $f$, we can choose the desired $\bar\alpha$ and $\bar r$.
     
 In the following, we simply write $u_{\alpha,\bar r}$ as $u_{\alpha}$.
 We will prove that $u(x)\leq u_{\bar\alpha}(x)$ for $|x|<\bar r$. Suppose not, then there exists some $\alpha>\bar\alpha$ such that $u_{\alpha}(x)\geq u(x)$ for  $|x|<\bar r$ and $u_{\alpha}(x_0)=u(x_0)$ at some $|x_0|<\bar r$. Therefore, $\nabla^2_g u_{\alpha}\geq \nabla_g^2 u$ and $\nabla_g u_{\alpha}=\nabla_g u$ at $x_0$. This implies that $-A_{g_{u_{\alpha}}}\geq -A_{g_u}$ at $x_0$. Hence,  $f(\lambda(-A_{g_{u_{\alpha}}}))\geq f(\lambda(-A_{g_u}))\geq h_0$ at $x_0$. We reach a contradiction. We have proved $u(x)\leq u_{\bar\alpha}(x)$ for $|x|\le \bar r$. The conclusion of Lemma \ref{upperbound} follows.
\end{proof}

\begin{proof}[Proof  of Theorem \ref{nlocalgradest}]
  Let $v:= \frac{2}{n-2}\log u$. Then $v$ satisfies \eqref{eqn-230602-1041} with $\alpha = -1, \beta = 1/2$, $S = - A_g$, and $\theta(v) = - 2v$. In this case, \eqref{eqn-230603-0100} can still be verified since now $u \leq C$ implies $v \leq \frac{2}{n-2}C$. Hence, we can apply Theorem \ref{thm-230603-1124} to obtain the desired gradient estimate.

By Lemma \ref{upperbound}, we can remove the dependence on $\sup_{B_1}u$ in the local gradient estimate \eqref{estimatee} assuming $h\geq h_0$ in $B_1$.
\end{proof}

\subsection{Counterexamples in Remark \ref{231006-1030}}\label{sec-7.3}

In this subsection, we construct solutions to show the failure of gradient estimates in the borderline cases, assuming only one-sided bound on solutions. Here, we work with more general equations than \eqref{pequ} and \eqref{equ} in $\bR^n$, $n\geq 2$.

We first construct examples for the negative curvature equation
\begin{equation} \label{eqn-230604-1235}
    f (\lambda ( - A[v] ) ) = h,\quad \lambda ( - A[v])  \in \Gamma \quad \text{in}\,\, B_1
\end{equation}
when $-\bl\in\p\Gamma$.
When $n \geq 3$ and $u = e^{\frac{n-2}{2}v}$, we have $A^u = A[v]$ and \eqref{eqn-230604-1235} becomes \eqref{equ}. In particular, Example \ref{exp-230606-0509} shows the failure of Theorem \ref{nlocalgradest} when $(f,\Gamma) = (\sigma_{n/2}^{2/n}, \Gamma_{n/2})$. Note that $-\bl\in\p\Gamma_{n/2}$.

\begin{example} \label{exp-230606-0509}
    There exists a sequence $v_j \in C^\infty(B_1)$, such that
    \begin{equation} \label{eqn-230607-0304-1}
        v_j \leq C \quad \text{in}\,\, B_{1},
    \end{equation}
        \begin{equation} \label{eqn-230607-0304-2}
        |\nabla v_j| \rightarrow \infty\,\,\text{uniformly on}\,\,B_{1/2},
    \end{equation}
    \begin{equation} \label{eqn-230607-0304-3}
        \lambda (-A[v_j]) \in \Gamma_{n/2}, \quad \sigma_{n/2}(\lambda (-A[v_j])) \rightarrow 1\,\,\text{in}\,\,C^m(B_1)\,\,\text{for each}\,\,m\in\bN.
    \end{equation}
\end{example}

For each $j\geq 1$, take
\begin{equation*}
    v_j = j (x_1 - 2) + w_j,
    \quad
    \text{where}
    \quad
    w_j =  j^{-n}e^{n j (x_1 - 2)}.
\end{equation*}
Clearly \eqref{eqn-230607-0304-1} holds for large $j$, and $|\nabla v_j| = |j + n j^{-n+1}e^{nj(x_1 - 2)}| \rightarrow \infty$, which implies \eqref{eqn-230607-0304-2}.
We are left to verify \eqref{eqn-230607-0304-3}. Since $(-1,1,\ldots,1) \in \p\Gamma_{n/2}$ and $v_j'' > 0$, we see from \eqref{eqn-230606-0409} that $\lambda(- A[v_j]) \in \Gamma_{n/2}$. A calculation shows
\begin{equation*}
    \sigma_{n/2}(\lambda (- A[v_j] ) ) = c_n
    e^{-n w_j} (1 + n w_j)^{n-2} \rightarrow c_n\,\,\text{in}\,\,C^m(B_1)
\end{equation*}
for each $m\in \bN$, where $c_n = n^2 2^{1 -\frac{n}{2} } 
    \begin{pmatrix}
        n-1\\ \frac{n}{2} -1
    \end{pmatrix}$. Replacing $v_j$ with $v_j + \frac{1}{n} \log c_n$, we reach the desired example.

\medskip

Actually, Theorem \ref{nlocalgradest} fails for any $(f,\Gamma)$ with $f$ being homogeneous of degree $1$ and $-\bl\in\p\Gamma$ as shown in the following example.

\begin{example} \label{exp-230606-0500}
    For any $(f,\Gamma)$ satisfying \eqref{eqn-230331-0110}, \eqref{basicf}, $f$ being homogeneous of degree $1$, and $-\bl\in\p\Gamma$, there exists a sequence $v_j \in C^\infty(B_1)$, such that \eqref{eqn-230607-0304-1} and \eqref{eqn-230607-0304-2} hold, and
    \begin{equation*}
        \lambda( - A[v_j]) \in \Gamma, \quad f (\lambda (-A[v_j])) \rightarrow 0 \,\,\text{in}\,\, C^m (B_1) \,\,\text{for each}\,\, m\in \bN.
    \end{equation*}
\end{example}
For $j\in \bN$, take
\begin{equation*}
    v_j = v_j(x_1) = -j \log (j^{-1} x_1 + C_j) + j \log (-j^{-1} + C_j),
\end{equation*}
where $5 j^{-1} \leq C_j \rightarrow 0$ to be chosen later. For such $C_j$, clearly $v_j \leq v_j (-1) = 0$ on $B_1$ and $|\nabla v_j| = |j^{-1} x_1 + C_j|^{-1} \rightarrow \infty$ uniformly on $B_1$. Moreover, from \eqref{eqn-230606-0409} and $-\bl\in\p\Gamma$,
$\lambda(-A[v_j]) = \frac{1}{2} (v_j')^2 e^{-2v_j} (2j^{-1}-1,1,\ldots,1) \in \Gamma$. 
Using the $1$-homogeneity of $f$ and denoting $\omega(j) := f(2j^{-1}-1,1,\ldots,1)$, a calculation shows in $B_1$
\begin{equation*} 
\begin{split}
    f( \lambda(-A[v_j]) ) 
    &=
    \frac{1}{2} (-j^{-1} + C_j)^{-2} \left( 1 + \frac{1 + x_1}{j C_j - 1} \right)^{2j-2} \omega(j)
    \\&\leq
    C C_j^{-2} e^{4/(C_j-1)} \omega(j) \leq C e^{8/C_j} \omega(j).
\end{split}
\end{equation*}
\sloppy Here, we have used $\sup_{s>0} s^{-1}\log(1 + s) = 1$ and $C_j \geq 5j^{-1}$. Similarly, $|(d^m/dx_1^m) f( \lambda(-A[v_j]) )| \leq C(m) e^{8/C_j} \omega(j)$. Since $\omega(j) \rightarrow 0$, we can choose $5j^{-1} \leq C_j \rightarrow 0$ such that $C(m) e^{8/C_j}\omega(j) \rightarrow 0$ for each $m\in \bN$.

\medskip

In a similar vein, we can construct examples for positive curvature equations. In particular, for $n\geq 3$, by taking $u_j = e^{\frac{n-2}{2} v_j}$, Example \ref{exp-230606-0459} shows the failure of Theorem \ref{localgradest} when $\bl\in\p\Gamma$.
\begin{example} \label{exp-230606-0459}
    For any $(f,\Gamma)$ satisfying \eqref{eqn-230331-0110}, \eqref{basicf}, $f$ being homogeneous of degree $1$, and $\bl\in\p\Gamma$, there exists a sequence $v_j \in C^\infty(B_1)$ satisfying \eqref{eqn-230607-0304-1}, \eqref{eqn-230607-0304-2}, and
    \begin{equation*}
        \lambda( A[v_j]) \in \Gamma, \quad f (\lambda (A[v_j])) \rightarrow 0 \,\,\text{in}\,\, C^m (B_1) \,\,\text{for each}\,\, m\in \bN.
    \end{equation*}
\end{example} 
The construction is very similar to that of Example \ref{exp-230606-0500}. For each $j \in \bN$, take
\begin{equation*}
    v_j = v_j(x_1) = j \log (j^{-1} x_1 + C_j) + j \log (C_j + j^{-1}),
\end{equation*}
where $C_j$ is chosen as before. Following the steps in the construction of Example \ref{exp-230606-0500}, we can verify all the desired properties.

\section{Proofs of Theorem \ref{exist-lcf}--\ref{exist-nlcf}}\label{sec-exist-and-compactness}
Let $(f,\Gamma)$ satisfy \eqref{eqn-230331-0110}, \eqref{lip-f-schouten} and \eqref{fgamma-convex}. We need the following approximation of $f$.

\textbf{Claim:} There exists a sequence of locally convex, homogeneous of degree $1$, symmetric function $\{f_j\}_j\subset C^\infty(\Gamma)\cap C^0(\overline{\Gamma})$ such that $\p_{\lambda_i} f_j\geq c(n,\delta)>0$ in $\Gamma$, $\forall i$, and $f_j\rightarrow {f}$ in $C^0_{loc}(\Gamma)$ as $j\to\infty$. Moreover, 
$\Gamma=\{c\lambda\mid c>0,~\lambda\in {f}_j^{-1}(1)\}$ for large $j$.

The above claim is proved as below.
Let $\widetilde{f}_j\coloneqq f*\eta_{1/j}$ with $\eta_{1/j}$ being the usual mollifier, $j=1,2,\dots$. We first show that
\begin{equation}\label{240805-2215}
    \Gamma=\{c\lambda\mid c>0,~\lambda\in \widetilde{f}_j^{-1}(1)\}\quad\text{for}~j~\text{large}.
\end{equation}
For any $s>0$, $\lambda\in f^{-1}(s)$ and $\lambda_0\in f^{-1}(0)$, we have $s={f}(\lambda)-{f}(\lambda_0)\leq \max_{\bR^n}|\nabla {f}|\cdot |\lambda-\lambda_0|\leq C|\lambda-\lambda_0|$. This implies that $\dist({f}^{-1}(s),{f}^{-1}(0))\geq s/C>0$. Therefore, we have $\widetilde{f}_{j}^{-1}(1)\subset\Gamma$ for large $j$. By the cone property of $\Gamma$, we have $\{c\lambda\mid c>0,~\lambda\in\widetilde{f}_j^{-1}(1)\}\subset\Gamma$. The other direction follows easily from the homogeneity of ${f}$. Property \eqref{240805-2215} is proved.

Similar to the arguments used in the proofs of Lemma \ref{lem-230617-1127} and \ref{lem-230615-1015}, the function $\varphi_j\in C^\infty(\Gamma)$ is well-defined by 
$\widetilde{f}_j(\varphi_j(\lambda)\lambda)=1$, $\lambda\in\Gamma$. Define
\begin{equation*}
    f_j(\lambda)\coloneqq
        \varphi_j(\lambda)^{-1}~\text{for}~\lambda\in\Gamma,~\text{and}~f_j(\lambda)\coloneqq 0~\text{for}~\lambda\in\p\Gamma. 
\end{equation*}
It is easy to check that $f_j$ has the desired properties. The claim is proved.

\begin{proof}[Proof of Theorem \ref{exist-lcf}]
   We only need to consider the case that $(M^n,g)$ is not conformally diffeomorphic to the standard sphere, since otherwise the result is well-known. 
   In the following, we use $\alpha$ and $C$ to denote some positive universal constants as those of \eqref{compactness-1.6-schouten} which may vary from line to line.
        
        By the arguments in \cite[Proof of Theorem $1.1'$]{LiLiacta}, all solutions $u$ of \eqref{q-schouten-critical} satisfy
        \begin{equation}\label{240330-2332}
            1/C\leq u\leq C~\quad \text{and}\quad |\nabla_g\log u|\leq C\quad \text{on}~M^n.
        \end{equation}
        
          We sketch the proof here for reader's convenience.
        Let $(\widetilde{M},\widetilde{g})$ be the universal cover of $(M^n,g)$ with $i:\widetilde{M}\rightarrow M^n$ being a covering map and $\widetilde{g}=i^* g$.
         Since $R_g>0$, by the theorem of Schoen and Yau \cite{Schoen-Yau}, there exists a conformal injective immersion $\Phi:(\widetilde{M},\widetilde{g})\rightarrow (\mathbb{S}^n,g_0)$. Let $\Omega=\Phi(\widetilde{M})$. There are two possibilities.

        \noindent
        \textbf{Case 1:} $\Omega=\mathbb{S}^n$. We have $(\Phi^{-1})^*\widetilde{g}=\eta^{4/(n-2)}g_0$ on $\mathbb{S}^n$, where $\eta$ is a positive smooth function on $\mathbb{S}^n$. Rewriting the equation on $\mathbb{S}^n$, we have
        \begin{equation*}
         f(\lambda(A_{[(u\circ i\circ \Phi^{-1})\eta]^{4/(n-2)}g_0}))=1\quad \text{on}~\mathbb{S}^n.
        \end{equation*}
       By the Liouville-type theorem \cite[Corollary 1.6]{LiLiCPAM2003}, $(u\circ i\circ \Phi^{-1})\eta=a|J_\varphi|^{(n-2)/2n}$ for some conformal diffeomorphism $\varphi: \mathbb{S}^n\rightarrow \mathbb{S}^n$ and some constant $a>0$. By $\varphi^*g_0=|J_\varphi|^{2/n}g_0$ and the equation, we have 
       $a=f(\lambda(A_{g_0}))^{(n-2)/4}=f((n-1)\vec{e})^{(n-2)/4}\in [1/C, C]$.
       The estimate \eqref{240330-2332} follows as the same arguments as in \cite{LiLiacta}. 

       \noindent 
       \textbf{Case 2:} $\Omega\neq\mathbb{S}^n$. By the results in \cite{Schoen-Yau}, $\Omega$ is an open and dense subset of $\mathbb{S}^n$, and $(\Phi^{-1})^*\widetilde{g}=\eta^{4/(n-2)}g_0$ on $\Omega$, where $\eta$ is a positive smooth function on $\Omega$ satisfying $\lim_{x\to\p\Omega}\eta(x)=\infty$. 
       This ensures that we can implement the method of moving spheres 
       to obtain \eqref{240330-2332} as in \cite{LiLiacta}.
       
       In the following, we derive $C^{2,\alpha}$ estimates for solutions $u$.
       By Proposition \ref{lem-240826-0153}, $f$ can be extended to a convex function on $\bR^n$, still denoted by $f$, such that $f$ is locally Lipschitz, symmetric, homogeneous of degree $1$, and for some constant $c = c(\delta) >0$, $\p_{\lambda_i} {f} \geq c$ a.e. $\bR^n$ for any $i$. By Lemma \ref{240902-1319}, it holds that $\p_{\lambda_i}f\leq f(\vec{e})$ a.e. $\bR^n$ for any $i$.   
        
        We first prove that
        \begin{equation}\label{C1alpha-estimate}
           \|u\|_{C^{1,\alpha}(M^n,g)}\leq C.         
        \end{equation}
     
        Denote $F(\nabla^2 u, \nabla u, u, x)\coloneqq u^{\frac{n+2}{n-2}}f(\lambda(A_{g_u}))-u^{\frac{n+2}{n-2}}$ and 
        \begin{equation*}
            F_\epsi(M,x)\coloneqq \epsi^2 F(\frac{1}{\epsi^2}M, \nabla u(\epsi x), u(\epsi x),\epsi x),~M\in \cS^{n\times n},
        \end{equation*}
        where $\cS^{n\times n}$ is the set of $n\times n$ symmetric matrices and $\epsi>0$ is some universal small constant to be chosen later. 
        Clearly, the ellipticity constants of $F_\epsi$ is the same as that of $F$, and $F_\epsi(\nabla^2 u_\epsi,x)=0$, where $u_\epsi(x)\coloneqq u(\epsi x)$. 
        By \eqref{240330-2332}, we may assume $F_\epsi(0,x)\equiv 0$. Also by \eqref{240330-2332}, we have   
        \begin{align*}
            \beta(x) \coloneqq \sup_{M\in S^{n\times n}}\frac{|F_\epsi(M,x)-F_\epsi(M,0)|}{\|M\|+1} &\leq \sup_{M\in S^{n\times n}}{|F_\epsi(M,x)-F_\epsi(M,0)|} \\
            \leq C \epsi^2 |N(\epsi x)-N(0)| &\leq C\epsi^2, 
            \end{align*}
        where $N(x)\coloneqq \frac{2n}{(n-2)^2}u^{-1}(x) \nabla u(x)\otimes \nabla u(x)-\frac{2}{(n-2)^2} u^{-1}(x)|\nabla u(x)|^2 I +u(x) A_g(x)$. In the last inequality above, we choose $\epsi>0$ small. 
        By the above and \cite[Corollary 5.7]{Caffarelli1995FullyNE}, we can apply \cite[Theorem 2]{Caff89} (with $F=F_\epsi$ and $f=0$ there, see also \cite[Theorem 8.3 and Remark 4] {Caffarelli1995FullyNE})) to obtain $\|u_\epsi\|_{C^{1,\alpha}}\leq C$. The estimate \eqref{C1alpha-estimate} follows.

        We next prove that
       \begin{equation}\label{C2alpha-estimate}
           \|u\|_{C^{2,\alpha}(M^n,g)}\leq C.         
        \end{equation}
     Indeed, denote $F$ as before and
        \begin{equation*}
            \widetilde{F}(M,x)\coloneqq F(M,\nabla u(x), u(x), x),~M\in \cS^{n\times n}.
        \end{equation*} 
        By the uniform ellipticity of $\widetilde{F}$ and \eqref{240330-2332}, we may assume $\widetilde{F}(0,0)=0$. 
        By \eqref{C1alpha-estimate}, $\widetilde{F}(M,x)$ is $C^\alpha$ in $x$.
        Applying \cite[Theorem 3]{Caff89} (with $F=\widetilde{F}$ and $f\equiv 0$ there, see also \cite[Theorem 8.1 and Remark 1] {Caffarelli1995FullyNE}), the estimate \eqref{C2alpha-estimate} holds.             
        
        Combining \eqref{240330-2332} and \eqref{C2alpha-estimate}, we have proved estimate \eqref{compactness-1.6-schouten}.

        Finally, we solve \eqref{q-schouten-critical} for $u\in C^{4,\alpha}(M^n)$ by degree theory. 
     By the claim at the beginning of this section, we may 
    assume without loss of generality that $f\in C^\infty(\Gamma)$.
    For $0\leq t\leq 1$, let $f^t(\lambda)\coloneqq f(t\lambda+(1-t)\sigma_1(\lambda)\vec{e})$ and $\Gamma^t\coloneqq\{ \lambda\in\bR^n\mid t\lambda+(1-t)\sigma_1(\lambda)\vec{e}\in\Gamma\}$. Consider 
        \begin{equation}\label{app-degree-equ}
            f^t(\lambda(A_{g_u}))=1,~\lambda(A_{g_u})\in\Gamma^t,~\text{on}~M^n.
        \end{equation}
        
        By \eqref{compactness-1.6-schouten} and the Schauder estimates, all $C^{4,\alpha}$ solutions $u$ of the equation \eqref{app-degree-equ} satisfy
        \begin{equation}\label{C_4_alpha}
        \|u\|_{C^{4,\alpha}(M^n,g)}+ \|1/u\|_{C^{4,\alpha}(M^n,g)}\leq C,    \end{equation}
        where $C>0$ is a constant independent of $t\in[0,1]$. Note that though \eqref{compactness-1.6-schouten} is proved for specific $(f,\Gamma)$, it applies for $(f_t,\Gamma_t)$ with universal dependence.

        By \eqref{C_4_alpha} and $f|_{\p\Gamma}=0$, there exists $\delta>0$ independent of $t$ such that all solutions $u$ of \eqref{app-degree-equ} satisfy $\operatorname{dist}(\lambda(A_{g_u}),\p\Gamma^t)\geq 2\delta$. Define $O^t\coloneqq \{u\in C^{4,\alpha}(M^n,g)\mid \lambda(A_{g_u})\in\Gamma^t,~\operatorname{dist}(\lambda(A_{g_u}),\p\Gamma^t)>\delta,~u>0,~\|u\|_{C^{4,\alpha}(M^n,g)}+ \|1/u\|_{C^{4,\alpha}(M^n,g)}<2 C           
        \}$.
        By \cite{Li_degree}, $d^t\coloneqq \deg (F^t-1,O^t,0)$, $t\in[0,1]$, is well-defined, where $F^t[u]\coloneqq f^t(\lambda(A_{g_u}))$. Furthermore, $d_t\equiv d_0$, $t\in[0,1]$, and in particular, $d_1=d_0$. 
        By the result of Schoen in \cite{Schoen_on_the_number} for the Yamabe equation, $d_0=-1$. Therefore, $d_1\neq 0$ and the existence of $C^{4,\alpha}$ solutions of \eqref{q-schouten-critical} is proved. 
        \end{proof} 

        \begin{proof}[Proof of Theorem \ref{exist-nlcf-compactness}]
        In the following, we use $\alpha$ and $C$ to denote some positive constants depending only on $(M^n,g)$, $(f,\Gamma)$ and $q$, which may vary from line to line.
        It suffices to derive an upper bound of solutions:
        \begin{equation}\label{240412-2142}
            u\leq C\quad \text{on}~M^n.
        \end{equation}
        Indeed, assume that \eqref{240412-2142} holds at this moment. By the local gradient estimate Theorem \ref{localgradest} (See also Remark \ref{remark-1.8}), we get $|\nabla\log u|\leq C$ on $M^n$. Here, $\lambda^*\notin\overline\Gamma$ and $q\leq 4/(n-2)$ are used. By a maximum principle argument as the proof of Theorem \ref{exist-lcf}, we have $\max_{M^n}u\geq 1/ C$. Here, $q<4/(n-2)$ is used. Combining the above, we get $u\geq 1/C$ on $M^n$. Since $f$ is uniformly elliptic and convex, applying Caffarelli's estimates \cite[Theorem 2,3]{Caff89} as the proof of Theorem \ref{exist-lcf} gives the desired $C^{2,\alpha}$ estimate \eqref{compactness-1.6-schouten}.
        
        Next we prove \eqref{240412-2142} by contradiction. Assume that there is a sequence $\{u_j\}$ of positive smooth function satisfying $f(\lambda(A_{g_{u_j}}))=u_j^{-q}$ on $M^n$, but 
        \begin{equation*}
          u_j(x_j)=\max_{M^n} u_j\to \infty,
        \end{equation*}
        where $x_j\to x_\infty$ in the topology induced by $g$. To make our presentation neat, we only prove for a fixed metric $g$. Define 
        \begin{equation*}
            \Phi_j: T_{x_j} M^n \to M^n \quad \Phi_j(x)\coloneqq \exp_{x_j}({u_j(x_j)^{-p}}\cdot x ),
        \end{equation*}
        where $p=2^{-1}(4/(n-2)-q)$. Let 
        \begin{equation*}
            \widetilde{u}_j(x)\coloneqq u_j(x_j)^{-1}\cdot  u_j\circ \Phi_j (x), \quad |x|< i_0 u_j(x_j)^p,
        \end{equation*}
        where $i_0$ is the injectivity radius of $(M^n,g)$.
        Then $\widetilde{u}_j$ satisfies 
        \begin{equation*}
    f(\lambda(A_{{\widetilde{u}_j}^{\frac{4}{n-2}}\widetilde{g}_j}))=\widetilde{u}_j^{-q},\quad |x|<i_0 u_j(x_j)^p,       
    \end{equation*}
    where $\widetilde{g}_j\coloneqq u_j(x_j)^{p/2}\cdot  \Phi_j^* g$. Since $\widetilde{u}_j\leq 1$ on $\{|x|<i_0 u_j(x_j)^p\}$, using the local gradient estimate Theorem \ref{localgradest} as above, we obtain $|\nabla\log \widetilde{u}_j|\leq C(K)$ on $K$ of $\bR^n$. Since $\widetilde{u}_j(0)=1$, we have $u\geq C(K)^{-1}$ on $K$. As in the proof of Theorem \ref{exist-lcf}, applying Caffarelli's estimates \cite[Theorem 2,3]{Caff89} implies that $        \|\widetilde{u}_j\|_{C^{2,\alpha}(K)}\leq C(K)$ on any compact subset $K$ of $\bR^n$. Passing to a subsequence, $\widetilde{u}_j$ converges in $C^{2,\alpha}_{loc}(\bR^n)$ to some positive function $\widetilde{u}\in C^2(\bR^n)$ which satisfies
    \begin{equation*}
    {f}(\lambda(A^{\widetilde{u}}))=\widetilde{u}^{-q}\quad \text{on}~\bR^n.      
    \end{equation*}  
    This is a contradiction with the Liouville-type theorem Theorem \ref{thm-230525-0321}. The upper bound \eqref{240412-2142} is proved. Hence, \eqref{compactness-1.6-schouten} is established.
    \end{proof}

\begin{proof}[Proof of Theorem \ref{exist-nlcf}]
    As for Theorem \ref{exist-lcf}, we prove for $f\in C^\infty(M^n)$.     
    We solve \eqref{q-schouten-subcritical} for $u\in C^{4,\alpha}(M^n)$ by degree theory. 
    For $0\leq t\leq 1$, let $f^t$ and $\Gamma^t$ be defined same as those in the proof of Theorem \ref{exist-lcf}. Consider 
        \begin{equation}\label{app-degree-equ-2}
            f^t(\lambda(A_{g_u}))=u^{-q
},~\lambda(A_{g_u})\in\Gamma^t,~\text{on}~M^n.
        \end{equation}
        
        By Theorem \ref{exist-nlcf-compactness} and the Schauder estimates, all $C^{4,\alpha}$ solutions $u$ of the equation \eqref{app-degree-equ-2} satisfy
        \begin{equation}\label{C_4_alpha-2}
        \|u\|_{C^{4,\alpha}(M^n,g)}+ \|1/u\|_{C^{4,\alpha}(M^n,g)}\leq C,    \end{equation}
        where $C>0$ is a constant independent of $t\in[0,1]$. 
        Note that though Theorem \ref{exist-nlcf-compactness} is stated for $(f,\Gamma)$, the positivity of scalar curvature of $(M^n,g)$ ensures that it applies to $(f_t,\Gamma_t)$ with a universal dependence.

        By \eqref{C_4_alpha-2} and $f|_{\p\Gamma}=0$, there exists $\delta>0$ independent of $t$ such that all solutions $u$ of \eqref{app-degree-equ-2} satisfy $\operatorname{dist}(\lambda(A_{g_u}),\p\Gamma^t)\geq 2\delta$. Define $O^t\coloneqq \{u\in C^{4,\alpha}(M^n,g)\mid \lambda(A_{g_u})\in\Gamma^t,~\operatorname{dist}(\lambda(A_{g_u}),\p\Gamma^t)>\delta,~u>0,~\|u\|_{C^{4,\alpha}(M^n,g)}+ \|1/u\|_{C^{4,\alpha}(M^n,g)}<2 C           
        \}$.
        By \cite{Li_degree}, $d^t\coloneqq \deg (F^t-u^{-q},O^t,0)$, $t\in[0,1]$, is well-defined, where $F^t[u]\coloneqq f^t(\lambda(A_{g_u}))$. Furthermore, $d_t\equiv d_0$, $t\in[0,1]$, and in particular, $d_1=d_0$. 
        By \cite[Proof of Theorem 1.5]{MR3165241}
         (see the computation of $deg(G_0,\mathcal{O}_0,0)$ there), $d_0=-1$. Therefore, $d_1\neq 0$ and the existence of $C^{4,\alpha}$ solutions of \eqref{q-schouten-subcritical} is proved. 
        \end{proof}

\section{Further examples}\label{examplesection}
In addition to those given in \S \ref{intro-example-sec}, we provide further examples of pairs $(f,\Gamma)$. We also collect properties of these examples for the purpose that readers may easily match them with the assumptions of our theorems.

The first family of examples is ordered linear combinations of eigenvalues. Recall the convention stated in \S \ref{intro-example-sec}, we order $\lambda=(\lambda_1,\dots,\lambda_n)$ as $\lambda_1\geq \dots\geq \lambda_n$, and only specify the definition of a symmetric function in this region.
 \begin{example}\label{ex-ordered}
    For $n\geq 2$ and $\mu\in\overline{\Gamma_n}\setminus\{0\}$, consider $f^\mu(\lambda)\coloneqq \mu_1\lambda_1+\dots+\mu_n\lambda_n$, where $\lambda\in\bR^n$, and define
$\Gamma^{\mu}\coloneqq \{\lambda\in\bR^n\mid f^\mu(\lambda)>0\}$.
    \end{example}
    The cone $\Gamma^\mu$ satisfies \eqref{eqn-230331-0110}. It holds that $\bl\notin\p\Gamma^\mu$ if and only if $\mu_1\neq \mu_2+\dots+\mu_n$, and $\bl\notin\overline{\Gamma^\mu}$ if and only if $\mu_1<\mu_2+\dots+\mu_n$.    Note also that $\Gamma^{\widetilde{\mu}}=\bR^n\setminus\overline{-\Gamma^\mu}$ with $\widetilde{\mu}=(\mu_n,\dots,\mu_1)$.     
When $\mu_1\geq \dots\geq \mu_n$, $f^\mu$ is convex in  $\bR^n$ since $f^\mu(\lambda)=\sup \{\sigma(\mu)_1\lambda_1+\dots+\sigma(\mu)_n\lambda_n\}$, where $\sup$ is taken over any permutation $\sigma$ of $\bR^n$. 
When $\mu_1\leq \dots\leq \mu_n$, $f^\mu$ is concave in  $\bR^n$ since $f^\mu(\lambda)=\inf \{\sigma(\mu)_1\lambda_1+\dots+\sigma(\mu)_n\lambda_n\}$, where $\inf$ is taken over any permutation $\sigma$ of $\bR^n$. 

    \medskip
    
Among the pairs $(f^\mu,\Gamma^\mu)$, some of them have appeared naturally in geometry. Denote $\vec{e_i}\coloneqq (0,\dots,1,\dots,0)$, where the $i$-th argument is $1$.

\medskip

For $n\geq 3$ and $p=1,\dots,n-1$,
take $\mu(p)\coloneqq p\sum_{i\leq n-p}\vec{e_i}+(n-p)\sum_{i>n-p}\vec{e_i}$,
then $f^{\mu(p)}\equiv G_p$, where $G_p$,   defined in \S \ref{intro-example-sec}, arises naturally from the Weitzenb\"ock formula for $p$-forms.

When $p\leq n/2$, $G_p$ is concave, and when $p\geq  n/2$, $G_p$ is convex. It holds that $\bl\notin \overline{\Gamma^{\mu(p)}}$ when $p<n-1$, while $\bl\in \p\Gamma^{\mu(n-1)}$. Moreover, $\Gamma^{\mu(n-p)}=\bR^n\setminus\overline{-\Gamma^{\mu(p)}}$. 

\medskip

Another class of functions $\Lambda_{p,q}$ arises naturally from calibrated geometry. For $n\geq 2$, $1\leq p\leq n$ and $0\leq q\leq n-p$, take $\mu(p,q)\coloneqq \vec{e_p}+\dots+\vec{e_{p+q}}$, then 
\begin{equation*}
f^{\mu(p,q)}\equiv \Lambda_{p,q}(\lambda)\coloneqq \lambda_p+\lambda_{p+1}+\dots+\lambda_{p+q}. 
\end{equation*}
See \cite{MR2487853} and the references therein.   

When $p+q=n$, $\Lambda_{p,q}$ is concave, and when $p=1$, $\Lambda_{p,q}$ is convex.
It holds that $\bl\notin\overline{\Gamma^{\mu(p,q)}}$ if and only if $\max\{p,q\}\geq 2$;  $\bl\in\p\Gamma^{\mu(p,q)}$ if and only if $p=q=1$. Moreover, $\Gamma^{\mu(n+1-p-q,q)}=\bR^n\setminus\overline{-\Gamma^{\mu(p,q)}}$.

\medskip
 
The second family of examples consists of  circular cones, with the vertex at the origin and the axis given by $\{t\vec{e}\mid t>0\}$, $\vec{e}=(1,1,\dots,1)$. 
\begin{example} For $n\geq 2$ and $c\in [-1,1]$, consider $f^c(\lambda)\coloneqq \sigma_1(\lambda)+c|\lambda|$, where $\lambda\in\bR^n$, and define $\Gamma^c\coloneqq \{\lambda\in\bR^n\mid f^c(\lambda)>0\}$. 
\end{example}

The cone $\Gamma^c$ satisfies \eqref{eqn-230331-0110}.
 It holds that $\bl\notin\p\Gamma^c$ if and only if $c\neq (n-2)/\sqrt{n}$, and $\bl\notin\overline{\Gamma^c}$ if and only if $c<(n-2)/\sqrt{n}$. 
 Note also that $\Gamma^{-c}=\bR^n\setminus\overline{-\Gamma^c}$ and $\Gamma^{-1}=\Gamma_2$.
 The function $f^c$ satisfies 
 $\sum_i \p_{\lambda_i}f^c\geq n-\sqrt{n}$ in $\Gamma^c$.
 When $c\geq 0$, $f^c$ is convex in $\bR^n$; When $c\leq 0$,  $f^c$ is concave in $\bR^n$.

\medskip

The last family of examples consists of $f$ defined on cones $\Gamma$ with $\bR^n\setminus\Gamma$ convex which are constructed in Appendix \ref{app-a}.
\begin{example} \label{prop-230614-1016}
        For $n\geq 2$, let $\Gamma$ satisfy \eqref{eqn-230331-0110} with $\bR^n \setminus \Gamma$ convex. There exists a convex and homogeneous of degree $1$ function $f \in C^\infty(\Gamma) \cap C^0(\overline{\Gamma})$ which satisfies \eqref{basicf} and $\sum_i f_{\lambda_i} \geq \delta >0$ in $\Gamma$ for some constant $\delta$. 
\end{example}

\appendix
\section{Equivalent theorems in term of Ricci tensor}\label{riccisection}
For reader's convenience, we reformulate Theorem \ref{thm-230525-0321}--\ref{230805-2045} and Theorem \ref{thm-230528-0404} in terms of Ricci tensor. Ricci tensor and Schouten tensor are related by a linear transformation $T$ as explained below.

Recall that on a Riemannian manifold $(M,g)$ of dimension $n\geq 3$, we have 
\begin{equation*}
    Ric_g=(n-2)A_g + (2(n-1))^{-1}{R_g}\cdot g.
\end{equation*}
Clearly, $\lambda(Ric_g)\equiv T\lambda(A_g)$, where $T=(n-2)I+\vec{e}\otimes\vec{e}$. As $(f,\Gamma)$ denotes pairs for eigenvalues of Schouten tensor, we use $(\hat{f},\hat{\Gamma})$ to denote pairs for eigenvalues of Ricci tensor. Let
\begin{equation*}
    f(\lambda)\coloneqq \hat{f}(T\lambda),\quad \lambda\in\Gamma\coloneqq T^{-1}\hat{\Gamma}.
\end{equation*}
Then we have $f(\lambda(A_g))\equiv \hat{f}(\lambda(Ric_g))$, and $\lambda(A_g)\in\Gamma$ if and only if $\lambda(Ric_g)\in\hat\Gamma$.

Let $(\hat{f},\hat\Gamma)$ satisfy the following conditions:
 \begin{equation} \label{eqn-230331-0110-hat}
   \begin{cases}
    \hat\Gamma \subsetneqq \bR^n \,\, \text{is a non-empty open symmetric cone with vertex at the origin},\\
    \hat\Gamma+T\Gamma_n\subset\hat\Gamma,
    \end{cases}
\end{equation}
\begin{equation} \label{eqn-240223-0305-ricci}
\begin{cases}
\hat{f} \in C^{0,1}_{loc}(\hat\Gamma)\,\,\text{is a symmetric function satisfying} ~~
    T(\nabla \hat{f})\in c(K)\vec{e}+\Gamma_n\\
    \text{a.e. $K$, $c(K)>0$, for any compact subset} ~K~\text{of}~\hat\Gamma,
\end{cases}
\end{equation}
where $T=(n-2)I+\vec{e}\otimes\vec{e}$. It is easy to check that $T\Gamma_n$ is the interior of the minimal convex symmetric cone with vertex at the origin containing $(n-1,1,\dots,1)$. 
Note that condition \eqref{eqn-240223-0305-ricci} allows $\frac{\p \hat{f}}{\p \lambda_i}<0$ for some $i$.

For constant $q\geq 0$,
consider the equation
\begin{equation}\label{nequation-ricci-re-hat}
         \hat{f}(\lambda(Ric_{\bar{g}_u}))= u^{-q}\quad \text{on}~\bR^n,
\end{equation}
where
$\bar{g}_u= u^{\frac{4}{n-2}}\bar{g}$, $\bar{g}=|dx|^2$ is the flat metric, $u$ is a positive function on $\bR^n$, and 
$\lambda(Ric_{g})$ denotes, as usual, eigenvalues of $Ric_{g}$ with respect to $g$.

Denote
\begin{equation*}
    \hat{\lambda}^*\coloneqq (0,-1,\dots,-1).
\end{equation*}
The following rigidity theorem is equivalent to Theorem \ref{thm-230525-0321}.
\begin{theorem}\label{Liouville-1-re-hat}
     For $n\geq 3$ and $q\geq 0$, let $(\hat{f},\hat\Gamma)$ satisfy \eqref{eqn-230331-0110-hat}, \eqref{eqn-240223-0305-ricci}, and $\hat{\lambda}^*\notin\overline{\hat\Gamma}$.
 Assume that a positive function $u\in C^{1,1}_{loc}(\bR^n)$ satisfies \eqref{nequation-ricci-re-hat} almost everywhere. Then $q=0$ and $u\equiv e^{\frac{n-2}{2}v}$, where $v$ is of the form \eqref{nbubble} with $\bar{x}\in \bR^n$ and $a,b>0$ satisfying $\hat{f}(4(n-1)b^2 a^{-2} \vec{e})=1$.
\end{theorem}
The above result fails when $\hat{\lambda}^*\in\overline{\hat\Gamma}$; see counterexamples in Remark \ref{rmk-230615-1129}.

\smallskip

Next we state our rigidity results for viscosity solutions of the following equation:
\begin{equation} \label{eqn-221204-0213-ricci-re-hat}
    \lambda(Ric_{\bar{g}_u})\in\p\hat{\Gamma}.
\end{equation}
The following two theorems are equivalent to Theorem \ref{thm-230329-1327} and \ref{230918-1857}, respectively.
\begin{theorem}\label{Liouville-2-re-hat}
  For $n\geq 3$,  let $\hat{\Gamma}$ satisfy \eqref{eqn-230331-0110-hat} with $\hat{\lambda}^*\notin\p\hat\Gamma$.
    Then any positive continuous viscosity solution of \eqref{eqn-221204-0213-ricci-re-hat} in $\bR^n$ must be constant.
\end{theorem}
    
\begin{theorem}\label{Liouville-3-re-hat}
    For $n\geq 3$,  let $\hat{\Gamma}$ satisfy \eqref{eqn-230331-0110-hat} with $\hat{\lambda}^*\notin\p\hat\Gamma$.
    Then any positive continuous viscosity solution of \eqref{eqn-221204-0213-ricci-re-hat} in $\bR^n\setminus\{0\}$ is radially symmetric and non-increasing in the radial direction.    
\end{theorem}
Neither of above results holds when $\hat{\lambda}^*\in\p{\hat\Gamma}$; see counterexamples in Remark \ref{remark-1.3}. 

\smallskip

Theorem \ref{localgradest} and \ref{nlocalgradest} are equivalent to the two theorems below.   The following conditions on $(\hat{f},\hat\Gamma)$ will be assumed: $\hat\Gamma$ satisfies \eqref{eqn-230331-0110-hat} and $\hat{f}$ satisfies 
\begin{equation}\label{basicf-ricci}
     \left\{
	       \begin{array}{lr}
                \text{$\hat{f}\in C^0(\overline{\hat\Gamma})\cap C^1 (\hat\Gamma)$ is 
                symmetric in $\lambda_i$,}\\[2ex]
	       	\text{$\hat{f}>0$, \ $ T(\nabla\hat{f})\in\Gamma_n$
          $\text{in}$ $\hat\Gamma$, and~ $\hat{f}=0$ ~on $\partial{\hat\Gamma}$,
         }
	       	\end{array}
	\right.	
\end{equation}

\begin{equation} \label{eqn-230530-1131-ricci}
    \sum_i \p_{\lambda_i}\hat{f}(\lambda) \geq \delta \big( 1+ \big|\lambda_i \cdot \p_{\lambda_i}\hat{f}(\lambda)\big| \big)\,\,\text{on}\,\,\{ \lambda \in \hat\Gamma:  \hat{f} (\lambda) \in (0, C]\},~~ \forall C > 0,
\end{equation}
where $\delta=\delta(C) > 0$, and
\begin{equation} \label{eqn-230602-1040-ricci}
    \liminf_{s\rightarrow \infty} \hat{f}(s\lambda) = +\infty \,\, \text{uniformly on}\,\, \{\lambda \in \hat\Gamma: \hat{f}(\lambda) \geq C\}, ~~ \forall C>0.
\end{equation}
The equivalence of \eqref{eqn-230530-1131} and \eqref{eqn-230530-1131-ricci} can be easily verified using the fact $T\vec{e}=2(n-1)\vec{e}$.
\begin{theorem}\label{localgradest-ricci}
    Let $(B_1,g)$ be a $C^3$ Riemannian geodesic ball and $(\hat{f},\hat\Gamma)$ satisfy \eqref{eqn-230331-0110-hat} and \eqref{basicf-ricci}--\eqref{eqn-230602-1040-ricci} with $\hat{\lambda}^*\notin\p\hat\Gamma$. Let $h\in C^1(B_1)$ be positive
    and $u\in C^3(B_1)$ satisfy 
    \begin{equation*}
        \hat{f}(\lambda(Ric_{g_u}))=h\quad\text{in}~B_1.
    \end{equation*}
    Then \eqref{estimatee} holds with constant $C$ depending only on $(\hat{f},\hat\Gamma)$, an upper bound of $\sup_{B_{1}} u$ and $\left\| h \right\|_{C^1(B_1)}$, and a bound of $g$ and the Riemann curvature
 tensor together with its first covariant derivative with respect to $g$. 
\end{theorem}
The above result fails when $\hat{\lambda}^*\in\p{\hat\Gamma}$; see counterexamples in Remark \ref{231006-1030}.
\begin{theorem}\label{nlocalgradest-ricci}
 Let $(B_1,g)$ be a $C^3$ Riemannian geodesic ball and $(\hat{f},\hat\Gamma)$ satisfy \eqref{eqn-230331-0110-hat} and \eqref{basicf-ricci}--\eqref{eqn-230602-1040-ricci} with $-\hat{\lambda}^*\notin\p\hat\Gamma$. Let $h\in C^1(B_1)$ be positive
    and $u\in C^3(B_1)$ satisfy 
    \begin{equation*}
        \hat{f}(\lambda(-Ric_{g_u}))=h\quad\text{in}~B_1.
    \end{equation*}
Then \eqref{estimatee} holds with $C$ of the same dependence as that of Theorem \ref{localgradest-ricci}. Moreover, if in addition $\inf_{B_1}h\geq h_0>0$, then $C$ in \eqref{estimatee}, depending on $h_0$, is independent of $\sup_{B_1} u$.
\end{theorem}
The above result fails when $-\hat{\lambda}^*\in\p{\hat\Gamma}$; see counterexamples in Remark \ref{231006-1030}.

\smallskip

Next we reformulate Theorem \ref{exist-lcf}--\ref{exist-nlcf}, imposing following conditions on $(\hat{f},\hat\Gamma)$.
\begin{equation}\label{lip-f-ricci} 
\begin{cases}
    \hat{f}\in C^{0,1}_{loc}(\hat\Gamma)\cap C^0(\overline{\hat\Gamma})~\text{is symmetric, homogeneous of degree $1$,} \\ \hat{f}\big|_{\p\hat\Gamma}=0,~\text{and}~T(\nabla \hat{f})\in \delta\vec{e} +\Gamma_n~\text{a.e.}~\hat\Gamma,~\text{for some constant}~\delta>0,
    \end{cases}
\end{equation}
and
\begin{equation}\label{fgamma-convex-ricci}
    \hat{f}~\text{is locally convex in}~\hat\Gamma~\text{and}~\bR^n\setminus\hat\Gamma~\text{is convex}.
\end{equation}

\begin{theorem}\label{exist-lcf-ricci}
        Let $(M^n,g)$ be a closed, smooth, locally conformally flat $n$-dimensional Riemannian manifold with positive scalar curvature, $n\geq 3$, and $(\hat{f},\hat\Gamma)$
    satisfy \eqref{eqn-230331-0110-hat}, \eqref{lip-f-ricci} and \eqref{fgamma-convex-ricci}.
     Then, for some $\alpha\in(0,1)$ depending only on $(M^n,g)$ and $(\hat{f},\hat\Gamma)$, there exists a positive function $u\in C^{2,\alpha}(M)$ satisfying
     \begin{equation}\label{q-ricci-critical}
  \hat{f}(\lambda(Ric_{g_u}))=1\quad \text{on}~M^n.  
\end{equation}

Moreover, if $(M^n,g)$ is not conformally diffeomorphic to the standard sphere, 
    all $C^2$ solutions $u$ of \eqref{q-ricci-critical} satisfy 
    \eqref{compactness-1.6-schouten} with constant $C>0$ depending only on $(M^n,g)$ and $(\hat{f},\hat\Gamma)$.
\end{theorem}

\begin{theorem}
     Let
    $(M^n,g)$ be a closed, smooth $n$-dimensional Riemannian manifold, $n\geq 3$, $(\hat{f},\hat\Gamma)$
    satisfy \eqref{eqn-230331-0110-hat}, \eqref{lip-f-ricci}, \eqref{fgamma-convex-ricci} and $\hat{\lambda}^*\notin\overline{\hat\Gamma}$, and $0<q<4/(n-2)$ be a constant.
    Assume that $\lambda(Ric_g)\in\hat\Gamma$ on $M^n$ and $\hat{f}\in C^1(\hat\Gamma)$.
    Then,  for any $\alpha\in(0,1)$, 
    all positive $C^2$ solutions of \eqref{q-ricci-subcritical} satisfy \eqref{compactness-1.6-schouten} with $C$ depending only on $(M^n,g)$, $(\hat{f},\hat\Gamma)$, $q$ and $\alpha$.
\end{theorem}

\begin{theorem}\label{exist-nlcf-ricci}
     Let
    $(M^n,g)$ be a closed, smooth $n$-dimensional Riemannian manifold with positive scalar curvature, $n\geq 3$,   $(\hat{f},\hat\Gamma)$
    satisfy \eqref{eqn-230331-0110-hat}, \eqref{lip-f-ricci}, \eqref{fgamma-convex-ricci} and $\hat{\lambda}^*\notin\overline{\hat\Gamma}$, and $0<q<4/(n-2)$ be a constant. Then, for some $\alpha\in(0,1)$ depending only on $(M^n,g)$, $(\hat{f},\hat\Gamma)$ and $q$, there exists a positive function $u\in C^{2,\alpha}(M)$ satisfying 
    \begin{equation}\label{q-ricci-subcritical} 
      \hat{f}(\lambda(Ric_{g_u}))=u^{-q}\quad\text{on}~M^n 
    \end{equation}
\end{theorem}

Theorem \ref{230805-2045} is reformulated as below.
\begin{theorem}
    For $n\geq 3$, let $\hat\Gamma$ satisfy \eqref{eqn-230331-0110-hat} with $\hat{\lambda}^*\notin\overline{\hat\Gamma}$ and $-\hat{\lambda}^*\in\hat\Gamma$. Assume that $u$ is a positive continuous viscosity solution of
    $\lambda(Ric_{\bar{g}_u})\in\p\Gamma$ in $B\setminus\{0\}$ satisfying $u(x)=o(|x|^{2-n})~\text{as}~x\rightarrow 0$. Then $u$ can be extended as a positive function in
     $C^{0,1}_{loc}(B)$ which  satisfies $\lambda(Ric_{\bar{g}_u})\in\p\Gamma$ in $B$.
\end{theorem}

Finally, we reformulate Theorem \ref{thm-230528-0404}.
\begin{theorem}
    For $n\geq 3$, let $\hat\Gamma$ satisfy \eqref{eqn-230331-0110-hat} with $\hat{\lambda}^*\in\hat\Gamma$.        
Then any positive viscosity solution $u$ of $\lambda(Ric_{\bar{g}_u}) \in \bR^n\setminus \hat\Gamma$ on $\bR^n$ must be constant.
\end{theorem}
The above theorem is also equivalent to the following: 
For $n\geq 3$, let $\hat\Gamma$ satisfy \eqref{eqn-230331-0110-hat} with $-\hat{\lambda}^*\notin\overline{\hat\Gamma}$.        
Then any positive viscosity solution $u$ of $\lambda(-Ric_{\bar{g}_u}) \in \overline{\hat\Gamma}$ on $\bR^n$ must be constant.

\section{Construction of Example \ref{prop-230614-1016}}\label{app-a}

We work with open symmetric set $V \subset \bR^n$ with smooth boundary $\p V \neq \emptyset$. Denote 
$$
    \Gamma(V) : = \{cV: c > 0\}.
    $$
    to be the cone generated by $V$ and for $\lambda \in \p V$, $\nu(\lambda)$ to be the inward unit normal of $\p V$. The following condition will be discussed:
\begin{equation} \label{eqn-230615-0934}
    \nu (\lambda) \in \Gamma_n \quad \text{and} \quad \lambda\cdot \nu(\lambda) > 0, \quad \forall \lambda \in \p V.
\end{equation}
As before, we denote $\vec{e} = (1,\ldots,1)$.
\begin{lemma} \label{lem-230617-1127}
    Let $V \subset \bR^n$ be an open symmetric set with a non-empty smooth boundary satisfying \eqref{eqn-230615-0934}. Then we have $V+\Gamma_n \subset V$ and $0 \notin \overline{V}$. Moreover, the generated cone $\Gamma(V)$ satisfies 
    \eqref{eqn-230331-0110}.
\end{lemma}
\begin{proof}
    We first prove $V + \Gamma_n \subset V$. Suppose the contrary that there exist $\lambda \in V$ and $\mu \in \Gamma_n$, such that $\lambda + \mu \notin V$. Let $\bar{t} := \inf\{t>0: \lambda + t \mu \notin V\}$. Since $V$ is open and $\lambda + \mu \notin V$, we have $\bar{t} \in (0,1]$ and $\lambda + \bar{t} \mu \in \p V$. At such point we have $\mu \cdot \nu (\lambda + \bar{t} \mu) \leq 0$, which is impossible since both vectors belong to $\Gamma_n$, a contradiction.

    Next, we show $\Gamma(V) + \Gamma_n \subset \Gamma(V)$. This is clear since each $\lambda \in \Gamma(V)$ can be written as $c \overline{\lambda}$ for some $c>0$ and $\overline{\lambda} \in V$. Hence, for each $\mu \in \Gamma_n$, we have $\lambda + \mu = c (\overline{\lambda} + c^{-1} \mu) \in c (V + \Gamma_n) \subset c V \subset \Gamma(V)$.

    Finally, we show $\Gamma(V) \neq \bR^n$ and $0 \notin \overline{V}$. For this, we prove $\overline{V} \cap (- \overline{\Gamma_n}) = \emptyset$. This is because from \eqref{eqn-230615-0934}, we can deduce that $\p V \cap (-\overline{\Gamma_n}) = \emptyset$. Hence, either $V \cap (-\Gamma_n) = \emptyset$ or $(-\Gamma_n) \subset V$. The latter cannot happen since otherwise $\bR^n = (-\Gamma_n) + \Gamma_n \subset V + \Gamma_n \subset V$, which contradicts to $\p V \neq \emptyset$. 
\end{proof}
Noting Lemma \ref{lem-230617-1127}, for each $\lambda \in \Gamma(V)$,
    we are able to define a function $\varphi: V \rightarrow \bR_+$, such $\varphi(\lambda) \lambda \in \p V$. Actually, from \eqref{eqn-230615-0934} we have $\{c: c \lambda \in V\} = (\varphi(\lambda), \infty)$. 
\begin{lemma} \label{lem-230615-1015}
    Let $V \subset \bR^n$ be an open symmetric set with a non-empty smooth boundary satisfying \eqref{eqn-230615-0934}. Then there exists a unique function $f \in C^\infty (\Gamma(V)) \cap C^0 (\overline{\Gamma(V)})$, such that $f=1$ on $\p V$, $f=0$ on $\p \Gamma(V)$, $f(t\lambda) = t f(\lambda)$ for any $\lambda \in \Gamma(V)$ and $t>0$, $f_{\lambda_i} > 0$ in $\Gamma(V)$ for all $i$, and the following formula holds
    \begin{equation} \label{eqn-230615-1116}
        \sum_i f_{\lambda_i} (\lambda) = \frac{\vec{e} \cdot \nu(\lambda)}{ \lambda \cdot \nu (\lambda)}, \quad
        \forall \lambda \in \p V.
    \end{equation}
    Moreover, if the set $\bR^n \setminus V$ ($V$, respectively) is convex, $f$ is also convex (concave, respectively).
\end{lemma}
When the set $V$ itself is convex, such $f$ was constructed and was proved to be concave in \cite{LiLiacta}. 
\begin{proof}
    For $\varphi$ defined as before, let $f(\lambda) := \varphi(\lambda)^{-1}$ and $f|_{\p \Gamma(V)} = 0$. Following \cite[Appendix~B]{LiLiacta}, such $f$ satisfies all the desired properties, except \eqref{eqn-230615-1116} and the convexity. Note that in this step, the convexity assumption of $V$ in \cite[Appendix~B]{LiLiacta} was not used.

    Next, we verify that \eqref{eqn-230615-1116} holds for such $f$. By a calculation,
    \begin{equation*}
        \sum_i f_{\lambda_i} (\lambda) = \vec{e} \cdot \nabla f = \frac{\vec{e} \cdot \nabla f}{\lambda \cdot \nabla f} = \frac{\vec{e} \cdot (\nabla f / |\nabla f|)}{\lambda \cdot (\nabla f/ |\nabla f|)} = \frac{\vec{e} \cdot \nu(\lambda)}{ \lambda \cdot \nu (\lambda)}.
    \end{equation*}
    Here, we used the the fact that $\lambda \cdot \nabla f = f = 1$ on $\p V$. Hence, \eqref{eqn-230615-1116} is proved.

    Finally, assuming the convexity of $\bR^n \setminus V$, we show that $f$ is convex, i.e., for any $\lambda, \mu \in \Gamma (V)$ and $t\in (0,1)$, $f(t\lambda + (1-t) \mu) \leq t f(\lambda) + (1-t) f(\mu)$. Using $f = 1/ \varphi$ and $\varphi (t\cdot) = t^{-1}\varphi(\cdot)$, this is equivalent to show $ 1 / \varphi(t\lambda + (1-t) \mu) \leq 1/ \varphi(t \lambda) + 1/\varphi((1-t) \mu)$. In other words, we only need to prove
    \begin{equation} \label{eqn-230615-1153}
        \left( 1/ \varphi(t \lambda) + 1/\varphi((1-t) \mu) \right)^{-1} (t \lambda + (1-t) \mu) \notin V, \quad \forall \lambda, \mu \in \Gamma(V) \,\,\text{and}\,\, t \in (0,1).
    \end{equation}
    This can be seen from the definition, $\varphi(t \lambda) t \lambda, \varphi((1-t) \mu) (1-t) \mu \in \p V$,
    \begin{equation*}
        \text{LHS of \eqref{eqn-230615-1153}} = \frac{\varphi((1-t) \mu)}{\varphi(t \lambda) + \varphi((1-t) \mu)} \varphi(t \lambda) t \lambda + \frac{\varphi(t \lambda)}{\varphi(t \lambda) + \varphi((1-t) \mu)} \varphi((1-t) \mu) (1-t) \mu,
    \end{equation*}
    and the convexity of $\bR^n \setminus V$.
\end{proof}

Now for an open symmetric set $V \subset \bR^n$ satisfying $\p V \neq \emptyset$ and \eqref{eqn-230615-0934}, let $\varphi(\vec{e})$ be defined as before. We define a reflection
\begin{equation*}
    \Psi_V: \bR^n \rightarrow \bR^n, \quad \lambda \mapsto - \lambda + 2 \varphi (\vec{e}) \vec{e}.
\end{equation*}
Clearly $\Psi_V^2 = id$. 
\begin{lemma} \label{lem-230617-1156}
    Let $V \subset \bR^n$ be an open symmetric set with a non-empty smooth boundary satisfying \eqref{eqn-230615-0934}. Denote $\widetilde{V} := \bR^n \setminus \overline{\Psi_V(V)}$ and $\widetilde{\nu}$ be the inward unit normal of $\p \widetilde{V}$. Then, for any $\lambda \in \p V$, we have
    \begin{equation} \label{eqn-230617-1220}
        \widetilde{\nu} (\Psi_{V} (\lambda)) = \nu (\lambda)
    \end{equation}
    and
    \begin{equation} \label{eqn-230616-1233}
        \Psi_{V} (\lambda) \cdot \widetilde{\nu} (\Psi_{V} (\lambda)) = \lambda \cdot \nu (\lambda) \left( -1 + 2 \varphi(\vec{e}) \frac{\vec{e} \cdot \nu(\lambda)}{\lambda \cdot \nu(\lambda)} \right) .
    \end{equation}
    If we further assume that $V$ is convex, then $\widetilde{V}$ also satisfies \eqref{eqn-230615-0934} with $\Gamma(\widetilde{V}) = \bR^n \setminus ( - \overline{\Gamma(V)} )$.
\end{lemma}
\begin{proof}
    Clearly, the map $\Psi_V$ flips the normal, from which $\widetilde{\nu} (\Psi_V (\lambda)) = \nu (\lambda)$. By a calculation,
    \begin{equation*}
        \Psi_V (\lambda) \cdot \widetilde{\nu} (\Psi_V (\lambda))
        =
        (-\lambda + 2 \varphi (\vec{e}) \vec{e}) \cdot \nu (\lambda)
        =
        \lambda \cdot \nu (\lambda) \left( -1 + 2 \varphi (\vec{e}) \frac{\vec{e} \cdot \nu(\lambda)}{\lambda \cdot \nu(\lambda)} \right).
    \end{equation*}
    Next, assuming the convexity of $V$, we prove the desired properties of $\widetilde{V}$. Since $V$ is convex and $\varphi (\vec{e}) \vec{e} \in \p V$, we have 
    \begin{equation*}
        \left( \varphi (\vec{e}) \vec{e} - \lambda \right) \cdot \nu(\lambda) \geq 0, \quad \forall \lambda \in \p V.
    \end{equation*}
    Using \eqref{eqn-230616-1233}, we obtain $\mu \cdot \widetilde{\nu} (\mu) \geq \Psi_V^{-1}(\mu) \cdot \nu (\Psi_V^{-1}(\mu)) > 0$ for any $\mu \in \p \widetilde{V}$. Also, from condition \eqref{eqn-230615-0934} for $V$ and \eqref{eqn-230617-1220}, we have $\widetilde{\nu}(\mu) = \nu (\Psi_V^{-1}(\mu)) \in \Gamma_n$ for any $\mu\in \p\widetilde{V}$. Hence, we have verified condition \eqref{eqn-230615-0934} for $\widetilde{V}$

Finally, we show $\Gamma(\widetilde{V}) = \bR^n \setminus ( - \overline{\Gamma(V)} )$. The ``$\supset$'' direction can be derived as follows: since $\Gamma(V)$ satisfies \eqref{eqn-230331-0110}, we can deduce that $\bR^n\setminus (-\overline{\Gamma(V)})$ also satisfies \eqref{eqn-230331-0110}. Hence, $\bR^n \setminus (-\overline{\Gamma(V)}) = \Gamma ( \bR^n \setminus (-\overline{\Gamma(V)}) + 2 \varphi (\vec{e}) \vec{e}) \subset \Gamma (\bR^n \setminus (-\overline{V}) + 2 \varphi (\vec{e}) \vec{e}) = \Gamma (\widetilde{V})$. Note that in this part, we do not use the convexity of $V$. For the ``$\subset$'' direction, note that we have proved $\widetilde{V}$ satisfies \eqref{eqn-230615-0934}. Combining with Lemma \ref{lem-230617-1127}, we obtain $\Gamma(\widetilde{V})$ also satisfies \eqref{eqn-230331-0110}. Now, applying the proof of ``$\supset$'' part with $V$ being replaced by $\widetilde{V}$, we have $\Gamma(V) \supset \bR^n\setminus (-\overline{\Gamma(\widetilde{V})})$. The desired result follows.
\end{proof}

Now we are ready to give a construction of $f$ in Example \ref{prop-230614-1016}:

    Let $\Gamma^{(1)} = \bR^n\setminus (-\overline{\Gamma})$. Since $\Gamma^{(1)}$ is convex and satisfies \eqref{eqn-230331-0110}, from \cite[Appendix~A]{2020arXiv200100993L}, there exists a concave and homogeneous of degree $1$ function $f^{(1)} \in C^\infty (\Gamma^{(1)}) \cap C^0 (\overline{\Gamma^{(1)}})$, satisfying \eqref{basicf}. Take $V^{(1)} := \{f^{(1)}>1\}$. It is easy to see that $V^{(1)}$ is convex with a non-empty, smooth boundary satisfying \eqref{eqn-230615-0934}. Now we define $V:= \bR^n \setminus \overline{\Psi_{V^{(1)}}(V^{(1)})}$. Then, by Lemma \ref{lem-230617-1156}, $V$ satisfies all the conditions in Lemma \ref{lem-230615-1015}, which allows us to define $f$ verifying all the desired properties.

\section{Extension of locally convex functions}
\begin{proposition}\label{lem-240826-0153}
    For $n\geq 2$, let $\Gamma$ satisfy \eqref{eqn-230331-0110} with $\bR^n\setminus \Gamma$ convex. Then any locally convex function $f$ on $\overline{\Gamma}$ satisfying \eqref{lip-f-schouten} can be extended to a convex symmetric function $\widetilde{f}$ on $\bR^n$, such that $\widetilde{f}$ is homogeneous of degree $1$, and for some constant $c = c(\delta,\Gamma) >0$, $\p_{\lambda_i} \widetilde{f} \geq c$ a.e. $\bR^n$ for any $i$.   
\end{proposition} 

We first collect some facts about functions $f$ defined on cones $\Gamma$, which are needed in our proof of Proposition \ref{lem-240826-0153} and could be of independent interest.

\begin{lemma}\label{lem-240827-1246}
    For $n\geq 2$, let $\Gamma$ satisfy \eqref{eqn-230331-0110}. Suppose that there is a function $f \in C^{0,1}_{loc}(\Gamma)\cap C^0(\overline{\Gamma})$ satisfying $f = 0$ on $\p\Gamma$ and $\p_{\lambda_i} f\geq \delta>0$ a.e. $\Gamma$ for any $i$. Then it must hold that $(1,0,\dots,0) \in \Gamma$ and $(-1,0,\dots,0) \in \bR^n\setminus {\overline{\Gamma}}$.
\end{lemma}
\begin{proof}
    By \eqref{eqn-230331-0110}, we have $(1,0,\dots,0)\in\overline\Gamma$ and $\{(t+\epsi,\epsi,\dots,\epsi)\mid \epsi,t>0\}\subset\Gamma$. Using $\p_{\lambda_1}f\geq \delta$ a.e. $\Gamma$ and the continuity of $f$ in $\overline\Gamma$, we have $f(t,0,\dots,0)\geq f(0)+ \delta t$ for $t>0$. Combining this with $f=0$ on $\p\Gamma$, we have $(1,0,\dots,0)\in\Gamma$. Similarly, one can prove $(-1,0,\dots,0)\in\bR^n\setminus\overline\Gamma$.
\end{proof}

\begin{lemma} \label{240902-1319}
    For $n\geq 2$, let $\Gamma$ satisfy \eqref{eqn-230331-0110}. Suppose that $f\in C^{0,1}_{loc}(\Gamma)$ is symmetric, homogeneous of degree $1$, locally convex (resp. locally concave).
    Then $\sum_{i=1}^n \p_{\lambda_i} f \leq
     f(\vec{e})$ (resp. $\sum_{i=1}^n \p_{\lambda_i} f \geq
     f(\vec{e})$) a.e. $\Gamma$. 
\end{lemma}

\begin{proof}
We only need to prove for the case that $f$ is locally convex since the other can be obtained by taking $-f$. By \eqref{eqn-230331-0110}, $\Gamma$ is star-shaped with respect to $\vec{e}$. Hence, by the local convexity and the homogeneity of $f$, 
$f(\vec{e})\geq f(\lambda)+\nabla f(\lambda)\cdot (\vec{e}-\lambda)= \sum_{i=1}^n \p_{\lambda_i}f(\lambda)$ a.e. $\lambda\in\Gamma$.
\end{proof}

\begin{lemma} \label{lem-240826-0245}
    For $n\geq 2$, let $\Gamma$ satisfy \eqref{eqn-230331-0110} 
    and $(-1,0,\dots,0) \notin \overline{\Gamma}$, i.e. $\mg<\infty$. Suppose that  
    $f\in C^{0,1}_{loc}(\Gamma) \cap C^0(\overline{\Gamma})$ is a symmetric, locally convex function satisfying $f\geq 0$ in $\Gamma$ and $f = 0$ on $\p\Gamma$.
     Then for $c=1/(1+\mg)>0$, there holds
        \begin{equation}\label{eqn-240826-02224}
            \p_{\lambda_i} f \geq c \sum_{j=1}^n \p_{\lambda_j}f\geq 0\,\,\text{a.e.}~\Gamma,~\text{for any}\,\,i.
        \end{equation}
\end{lemma}

\begin{proof}
For any $\mu\notin\overline\Gamma$ and $\lambda\in \Gamma$, by \eqref{eqn-230331-0110}, there exists some $t_\mu\in(0,\infty)$ such that $\lambda+t_\mu \mu \in\p\Gamma$, and $\lambda+t\mu\in\Gamma$ for $0\leq t<t_\mu$. 
Since $f(\lambda+t_\mu\mu)=0$ and $f(\lambda+t\mu)\geq 0$ for $0\leq t<t_\mu$,
the convexity of $f$ implies that $f(t\mu+\lambda)$ is non-increasing on $t\in[0,t_\mu]$. Hence, for any $\mu\notin\overline\Gamma$,
\begin{equation}\label{240907-1507}
        \frac{d}{dt} f(\lambda+t\mu)\big|_{t=0}\leq 0\quad\text{a.e.}~\lambda\in\Gamma.
    \end{equation}
Take $\mu=(-1,0,\dots,0)$ in \eqref{240907-1507}, we have $\p_{\lambda_1}f\geq 0$ a.e. $\Gamma$. On the other hand, take $\mu=(-\mgn-\epsi,1,\dots,1)$ for any $\epsi>0$ in \eqref{240907-1507} and then send $\epsi$ to $0$, \eqref{eqn-240826-02224} is proved for $i=1$.
Similarly, we can prove \eqref{eqn-240826-02224} for $2\leq i\leq n$
\end{proof}

\begin{lemma}\label{lem-240905-0216}
    For $n\geq 2$, let $\Gamma$ satisfy \eqref{eqn-230331-0110} and $(1,0,\dots,0)\in\Gamma$, $i.e.$ $\mgn<\infty$. Assume that $\Gamma$ is convex, and $f\in C^{0,1}_{loc}(\Gamma)$ is a symmetric, concave function satisfying $f\geq 0$ in $\Gamma$.
    Then \eqref{eqn-240826-02224} holds for $c=1/(1+\mgn)>0$.
\end{lemma}

\begin{proof}
For any $\mu\in\overline\Gamma$ and $\lambda\in\Gamma$, 
by \eqref{eqn-230331-0110} and the convexity of $\Gamma$, we have $\lambda+t\mu\in\Gamma$ for any $t\geq 0$. By the concavity and non-negativity of $f$,
   \begin{equation}\label{240907-1657}
       \frac{d}{dt}f(\lambda+t\mu)\big|_{t=0}\geq 0\quad \text{a.e.}~\lambda\in\Gamma.
    \end{equation}
Take $\mu=(1,0,\dots,0)$ in \eqref{240907-1657}, we have $\p_{\lambda_1}f\geq 0$ a.e. $\Gamma$. On the other hand, take $\mu=(\mgn,-1,\dots,-1)\in\p\Gamma$ in \eqref{240907-1657},
inequality \eqref{eqn-240826-02224} is proved for $i=1$, while \eqref{eqn-240826-02224} for $2\leq i\leq n$ can be proved similarly.
\end{proof}

\begin{proof}[Proof of Proposition \ref{lem-240826-0153}]
Since $\bR^n\setminus \Gamma$ is a convex cone, for $\lambda \in \bR^n$, we can define $F(\lambda):= \sup \vec{p}_{\overline{\lambda}} \cdot (\lambda - \overline{\lambda})$, where the supremum is taken with respect to all points $\overline{\lambda} \in \p\Gamma \cap S^{n-1}$ and all exterior unit normal vectors  $\vec{p}_{\overline{\lambda}}$ (of support hyperplanes) of $\bR^n\setminus \Gamma$ at $\overline{\lambda}$. From the definition, clearly $F(\lambda)\leq -\dist(\lambda,\p\Gamma)<0$ for $\lambda\in \bR^n\setminus \overline{\Gamma}$ and $|\nabla F| \leq 1$ on $\bR^n$. Since $\vec{p}_{\overline{\lambda}} \cdot \overline{\lambda} = 0$ due to the cone property, we can rewrite $F(\lambda) = \sup \vec{p}_{\overline{\lambda}} \cdot \lambda$, from which $F$ is also homogeneous of degree $1$. Now, applying Lemma \ref{lem-240827-1246} to $(f,\Gamma)$, we can deduce that $(-1,0,\ldots,0) \notin \overline{\Gamma}$. Hence, we are able to apply Lemma \ref{240902-1319} and \ref{lem-240905-0216} with $(f,\Gamma)=(-F(-\lambda),\bR^n\setminus\overline{-\Gamma})$ to obtain $\p F/\p \lambda_i \geq -cF(-1,\ldots,-1) > 0$ for some $c>0$, on $\bR^n\setminus \Gamma$. 

Define the extension of $f$ as
    \begin{equation*}
        \widetilde{f}
        :=
        \begin{cases}
            f(\lambda), \,\, \lambda \in \overline{\Gamma},\\
            {\delta} F(\lambda), \,\, \lambda \in \bR^n\setminus \overline{\Gamma}.
        \end{cases}
    \end{equation*}
    
We are only left to prove the convexity of $\widetilde{f}$. Since $F$ is convex in $\bR^n$ and $f$ is locally convex in $\Gamma$, it suffices to check: 
    \begin{equation} \label{eqn-240814-1054}
        {\big(\widetilde{f} (\lambda) + \widetilde{f} (\mu)\big)/2} \geq \widetilde{f} \big(({\lambda + \mu})/{2}\big),
        \quad
        \forall \lambda \in \Gamma, \mu \in \bR^n\setminus \overline{\Gamma}
        \,\,
        \text{with}
        \,\,
        {(\lambda + \mu)}/{2} \in \p\Gamma.
    \end{equation}
    For this, take any exterior unit normal vector $\vec{p}$ of the convex set $\bR^n\setminus \Gamma$ at $(\lambda+\mu)/2 \in \p\Gamma$. Since $\mu \in \bR^n\setminus \overline{\Gamma}$, it must hold $\vec{p} \cdot (\mu - (\lambda + \mu)/2) = \vec{p} \cdot (\mu - \lambda)/2 \leq 0$. Since $f((\lambda+\mu)/2)=0$ and $\lambda - ((\lambda - \mu)\cdot \vec{p}/2)\vec{p} = \lambda - ((\lambda - (\lambda+\mu)/2) \cdot \vec{p}) \vec{p} \in \overline{\Gamma}$,
    \begin{align*}
    \widetilde{f}(\lambda) - \widetilde{f}({(\lambda+\mu)}/{2}) 
    = 
    f(&\lambda) 
    \geq 
    f(\lambda) - f(\lambda - ((\lambda - \mu)\cdot \vec{p}/2) \vec{p})
   \\ =
    \int_{-1}^0 \frac{d}{dt} f(\lambda + t((\lambda - &\mu)\cdot \vec{p}/2) \vec{p})dt
    =
    ((\lambda - \mu)\cdot \vec{p}/2)\cdot 
    \\ \int_{-1}^0 \vec{p} \cdot \nabla f(\lambda + t((\lambda - \mu)\cdot \vec{p}/2) \vec{p}) & dt
     \geq  ((\lambda - \mu)\cdot \vec{p}/2)\delta \vec{p}\cdot \vec{e} \geq    ((\lambda - \mu)\cdot \vec{p}/2) \delta.
    \end{align*}
Here in the last inequality, we used, by \eqref{eqn-230331-0110}, $\vec{p} \in \overline{\Gamma_n}$, and $\nabla f \in \delta\vec{e} + \Gamma_n$ a.e. in $\Gamma$. Since $F((\lambda+\mu)/2) = 0$ and $F(\mu) \geq \vec{p} \cdot \mu$,
\begin{align*}
\widetilde{f}((\lambda+\mu)/2) - \widetilde{f}(\mu)
=
- {\delta} F(\mu)
\leq
-{\delta} \vec{p}\cdot \mu
=
- {\delta} \vec{p}\cdot (\mu - {(\lambda+\mu)}/{2})
= {\delta} \vec{p} \cdot {(\lambda - \mu)}/{2}.
\end{align*}
Here, we used $\vec{p}\cdot (\lambda+\mu)/2=0$, which comes from the cone property. Combining these, the lemma is proved.
\end{proof}

\section*{Acknowledgement}
B.Z. Chu and Y.Y. Li are partially supported by NSF Grant DMS-2247410. Z. Li is partially supported by an AMS-Simons travel grant. Most of this paper was completed while Z. Li was affiliated with Rutgers University.

\end{document}